\documentclass[10pt]{amsart}
\usepackage{stmaryrd}
\usepackage[all]{xy}
\usepackage{tikz}
\usepackage{graphicx}
\usetikzlibrary{matrix,arrows}

\theoremstyle{definition}

\newtheorem{cor}[equation]{Corollary}
\newtheorem{thm}[equation]{Theorem}

\newtheorem{lemma}[equation]{Lemma}
\newtheorem{prop}[equation]{Proposition}

\theoremstyle{remark}
\newtheorem{rmk}[equation]{Remark}

\newcommand{\ZZ}{\mathbb{Z}}

\newcommand{\FF}{\mathbb{F}}
\newcommand{\defined}{\mathrel{\mathop:}=}

\newcommand{\GG}{\mathbb{G}}

\newcommand{\mM}{\mathcal{M}}
\newcommand{\Spec}{\operatorname{Spec}}
\newcommand{\Ext}{\operatorname{Ext}}
\newcommand{\TMF}{\operatorname{TMF}}
\newcommand{\Tmf}{\operatorname{Tmf}}
\newcommand{\tmf}{\operatorname{tmf}}

\usepackage{amsmath}
\usepackage{amssymb}

\allowdisplaybreaks[1]

\numberwithin{equation}{subsection}
\numberwithin{figure}{subsection}

\title{On the homotopy of $Q(3)$ and $Q(5)$ at the prime $2$}
\date{\today}
\author{Mark Behrens}
\address{Department of Mathematics\\
  University of Notre Dame\\
  Notre Dame, IN 46556}
\email{mbehren1@nd.edu}
\author{Kyle Ormsby}
\address{Department of Mathematics\\
  Reed College\\
  Portland, OR 97202}
\email{ormsbyk@reed.edu}

\begin{document}

\begin{abstract}
We study modular approximations $Q(\ell)$, $\ell = 3,5$, of the $K(2)$-local
sphere at the prime $2$ that arise from $\ell$-power degree isogenies
of elliptic curves.  We develop Hopf algebroid level tools for working with
$Q(5)$ and record Hill, Hopkins, and Ravenel's computation of the
homotopy groups of $\TMF_0(5)$.  Using these
tools and formulas of Mahowald and Rezk for $Q(3)$ we determine the
image of Shimomura's $2$-primary divided $\beta$-family in the Adams-Novikov
spectral sequences for $Q(3)$ and $Q(5)$.  Finally, we use low-dimensional
computations of the homotopy of $Q(3)$ and $Q(5)$ to explore
the r\^{o}le of these spectra as approximations to $S_{K(2)}$.
\end{abstract}

\maketitle

\tableofcontents

In \cite{Behrens}, motivated by \cite{GHMR}, the $p$-local spectrum
$Q(\ell)$ ($p \nmid \ell$) is defined as the totalization of  an
explicit semi-cosimplicial $E_\infty$-ring spectrum of the form
$$ Q(\ell)^\bullet = \bigg( \TMF \Rightarrow \TMF_0(\ell) \times \TMF \Rrightarrow \TMF_0(\ell) \bigg). $$
The spectrum $Q(\ell)$ serves as a kind of approximation to the
$K(2)$-local sphere; see \S\ref{subsec:Qdef} for more details on its construction.
In \cite{Behrensbldg}, it is proven that there is an equivalence
$$ Q(\ell)_{K(2)} \simeq (E_2^{h\Gamma_\ell})^{h\mathrm{Gal}} $$
where $\Gamma_\ell$ is a certain subgroup of the Morava stabilizer group $\mathbb{S}_2$ coming from isogenies of elliptic curves.  The subgroup $\Gamma_\ell$ is dense if $p$ is odd and $\ell$ generates a dense subgroup of $\ZZ_p^\times$ \cite{BehrensLawson}.  Based on this, it is conjectured that there are fiber sequences
\begin{equation}\label{eq:Qfiber}
D_{K(2)} Q(\ell) \rightarrow S_{K(2)} \xrightarrow{u} Q(\ell)
\end{equation}
for such choices of $\ell$ (and the case of $\ell = 2$ and $p = 3$ is handled by explicit computation in \cite{Behrens}, and is closely related to \cite{GHMR}).    
Density also is used in \cite{Behrensbeta} to show that for such $\ell$, $Q(\ell)$ detects the exact divided $\beta$ family pattern for $p \ge 5$. 

However, in the case of $p = 2$, $\ZZ_2^\times$ is not topologically cyclic, and the closure of $\Gamma_\ell$ in $\mathbb{S}_2$ is the inverse image of the closure of the subgroup $\ell^{\ZZ} < \ZZ_2^\times$ under the reduced norm
$$ N: \mathbb{S}_2 \rightarrow \ZZ_2^\times. $$
It is not altogether clear in this case what the analog of the conjecture (\ref{eq:Qfiber}) should be, though one possibility is suggested in \cite{BehrensLawson}.  Although the 2-primary ``duality resolution'' of Goerss, Henn, Mahowald, and Rezk (see \cite{Henn}) seems to take the form of a fiber sequence like (\ref{eq:Qfiber}), we will observe that the mod $(2, v_1)$-behavior of $Q(3)$ actually precludes $Q(3)$ from being half of the duality resolution (see Remark~\ref{rmk:Q3wontwork}). The non-density of $\Gamma_3$, together with the appearance of both $\TMF_0(3)$ and $\TMF_0(5)$ factors in $\TMF \wedge \TMF$ also suggests that, from a $\TMF$-resolutions perspective, $Q(3)$ alone may not be seeing enough homotopy, and that a combined approach of $Q(3)$ and $Q(5)$ may be required at the prime $2$.  

The goal of this paper is to explore such an approach by extending the work of Mahowald and Rezk \cite{MahowaldRezk} on $Q(3)$, and initiating a similar study of $Q(5)$.   

The first testing ground for the effectiveness of $Q(3)$ or $Q(5)$ at
detecting $v_2$-periodic homotopy at the prime $2$ is Shimomura's
$2$-primary divided beta family \cite{Shimomura}.  To the authors'
surprise, $Q(3)$ was found to exactly detect Shimomura's divided beta
patterns on the $2$-lines of the $E_2$ term of its Adams-Novikov
spectral sequence, as we shall explain in Section~\ref{sec:beta}.
Hence $Q(3)$ is all that is needed to detect the shape of the divided
beta family.  The authors were equally surprised to find no such
phenomenon for $Q(5)$ - the beta family for $Q(5)$ has greater
$v_1$-divisibility than that for the sphere.  On the other hand, the
$K(2)$-localization of $Q(5)$ is built out of homotopy fixed point
spectra of groups with larger $2$-torsion than $Q(3)$.  This raises
the possibility that while $Q(5)$ may be less effective when it comes
to beta elements, it could detect exotic torsion in higher
cohomological degrees that is invisible to $Q(3)$.  This possibility
is explored through some low dimensional computations.

We now summarize the contents of this paper.  In Section
\ref{sec:level5} we review and expand the theory of
$\Gamma_0(5)$-structures on elliptic curves.  A
$\Gamma_0(5)$-structure is an elliptic curve equipped with a cyclic
subgroup of order 5.  We recall an explicit description of the scheme
representing $\Gamma_1(5)$-structures (elliptic curves with a point of
order 5) in terms of \emph{Tate normal form} curves and use this
description to present several Hopf algebroids that stackify
to the moduli space of $\Gamma_0(5)$-structures.  We then use these
Hopf algebroids and the geometry of elliptic curves to determine the
maps defining $Q(5)^\bullet$.

In Section \ref{sec:HFP} we compute the homotopy fixed point
spectral sequence
\[
  H^*(\FF_5^\times;\pi_*\TMF_1(5)) \implies \pi_*\TMF_0(5).
\]
The ring $\pi_*\TMF_1(5)$ and the action of $\FF_5^\times$ on it are
determined by Tate normal form, allowing us to produce a detailed
group cohomology computation.  We then compute the differentials and
hidden extensions in the spectral sequence by a number of methods:
$\TMF$-module structure, transfer-restriction arguments, and
comparision with the homotopy orbit spectral sequence.  Our use of the
homotopy orbit spectral sequence to determine hidden extensions is
somewhat novel and may find use in other contexts.  Note that the computation of $\pi_* \TMF_0(5)$ was first due to Mahowald and Rezk (unpublished) using this descent spectral sequence.  Hill, Hopkins, and Ravenel then rediscovered this computation using the slice spectral sequence \cite{HHRunpub}.

Since $Q(\ell)$ is the totalization of a cosimplicial spectrum, we can
compute the $E_2$-term of its Adams-Novikov spectral sequence as the
cohomology of a double complex.  The differentials in the double
complex are either internal cobar differentials for the Weierstrass or
$\Gamma_0(5)$ Hopf algebroids or external differentials determined by
the cosimplicial structure of $Q(\ell)^\bullet$.  In Section
\ref{sec:Ql} we review formulas for the external differentials in the
$\ell=3$ and $\ell=5$ cases.  The $Q(3)$ formulas are due to Mahowald
and Rezk \cite{MahowaldRezk} while those for $Q(5)$ are derived from Section
\ref{sec:level5}.

In Section \ref{sec:beta} we compute several chromatic spectral
sequences related to $Q(3)$ and $Q(5)$.  Definitions are stated in
Section \ref{subsec:chrom} and the technique we use is
carefully laid out in Section \ref{subsec:technique}.  Stated
precisely, we compute
$H^{0,*}(M_0^2 C^*_{tot}(Q(3)))$ and $H^{0,*}(M_1^1 C^*_{tot}(Q(5)))$,
both of which are related to the divided $\beta$ family in the
$Q(\ell)$ spectra.  We compare these groups to Shimomura's $2$-primary divided $\beta$
family for the sphere spectrum (i.e.~the groups $\Ext^{0,*}(M_0^2
BP_*)$, reviewed in Theorem \ref{thm:Shim}).  In Theorem \ref{thm:Q3beta} we find that $\Ext^{0,*}(M_0^2
BP_*)$ is isomorphic to $H^{0,*}(M_0^2C^*_{tot} Q(3))$, so $Q(3)$
precisely detects the divided $\beta$ family.  In contradistinction,
Theorem \ref{thm:Q5beta} and Corollary \ref{cor:Q5beta} show that the divided $\beta$
family for $Q(5)$ has extra $v_1$-divisibility.

Finally, in Section \ref{sec:lowdim} we compute $\pi_nQ(3)$ and $\pi_n
Q(5)$ for $0 \le n< 48$.  More precisely, what we actually compute is the portion of these homotopy groups detected by connective versions of $\TMF$.\footnote{It is likely that what we are computing is a ``connective'' version of $Q(\ell)$ built out of the connective versions of $\TMF_0(\ell)$ recently constructed in \cite{HillLawson}, though we do not pursue this here.}  These computations give evidence for some homotopy which $Q(5)$ detects which is not detected by $Q(3)$.

In this paper we assume the reader has some familiarity with the theory of elliptic curves, level structures, and the stacks which parameterize these objects.  We also assume the reader is familiar with $\TMF$, and its variants.  To give extensive background on these subjects would take us outside of the scope of this paper.  For the reader looking for outside resources, we recommend the 2007 Talbot conference proceedings \cite{TMF}.  These proceedings are to be published in a forthcoming book, but an online reference is given in the bibliography.  The expository articles contained there should point the inquiring reader in the right direction.  This paper is itself extending computations of \cite{Behrens} and \cite{MahowaldRezk}.  The reader is encouraged to have some familiarity with these cases before jumping into the computations contained herein.  

{\bf Acknowledgments}  The authors would like to extend their gratitude to Mike Hill, Mike Hopkins, and Doug Ravenel, for generously sharing their slice spectral sequence computations of the homotopy groups of $\TMF_0(5)$.  The first author would also like to express his debt to Mark Mahowald, for sharing his years of experience at the prime $2$, and providing preprints with exploratory computations.  The homotopy fixed point computations were aided by a key insight of Jack Ullman, who pointed out that the slice spectral sequence agreed with the homotopy orbit spectral sequence in a range, giving us the idea to use homotopy orbits to resolve hidden extensions.  The authors also would like to thank Agnes Beaudry and Zhouli Xu for pointing out some omissions in the low dimensional computations at the end of this paper, and the referee, for his or her careful comments on many aspects of this paper, in particular on the details of the computation of the homotopy fixed point spectral sequence for $\TMF_0(5)$.
The authors were both supported by grants from the NSF. 

\section{Elliptic curves with level $5$ structures}
\label{sec:level5}

We consider the moduli problems of $\Gamma_1(5)$- and
$\Gamma_0(5)$-structures on elliptic curves. An elliptic curve with a 
$\Gamma_1(5)$-structure over a commutative $\ZZ[1/5]$-algebra $R$ is a pair
$(C,P)$
where $C$ is an elliptic curve over $R$, and $P \in C$ is a point of exact order $5$.  
An elliptic curve with a $\Gamma_0(5)$-structure is a pair $(C,H)$ with $C$
an elliptic curve over $R$ and $H<C$ a subgroup of order $5$.  Let $\mM_i(5)$
denote the moduli stack (over $\Spec(\ZZ[1/5])$) of $\Gamma_i(5)$-structures.
 
Let $\mM_i^1(5)$ denote  
the moduli stack of tuples $(C,P,v)$ (respectively $(C,H,v)$) where
$v$ is a non-zero tangent vector at $0 \in C$.  This is equivalent to the
moduli problem in which an non-trivial invariant differential is recorded.  Note that in the case where $i
= 1$, we can use translation by $P$ to equivalently specify this
structure as a tuple $(C,P,v')$ where $v'$ is a non-zero tangent vector at
$P$.

As we proceed, we will
freely move between moduli problems of the form $\mM_i(\ell)$ and
$\mM_i^1(\ell)$, so we will comment briefly here on the significance in
topological modular forms of recording or not recording a tangent vector.
As is customary in the subject, let $\omega$ denote the invertible sheaf of invariant
differentials on the moduli stack of elliptic curves, $\mM$, so that sections of $\omega^{\otimes t}$ correspond to modular forms of weight $t$ which are meromorphic at the cusp.  Recall
that the elliptic spectral sequence takes the form
\[
  E_{2}^{s,t} = H^s(\mM;\omega^{\otimes t})\implies \pi_{t-s}\TMF.
\]
Now consider the stack of elliptic curves with the data of a non-zero tangent vector at $0$,
$\mM^1$.  This stack is equipped with a $\GG_m$-action which scales
this vector
$$
\GG_m \times \mM^1 \rightarrow \mM^1,
$$
which on points is given by
$$ (z, (C,v)) \mapsto (C, z\cdot v). $$
Let $\pi$ denote the forgetful map
$$ \pi: \mM^1 \rightarrow \mM $$
which on points is given by
$$ (C,v) \mapsto C. $$
The stack $\mM^1$ is a $\GG_m$-torsor over $\mM$, and $\omega$ is the associated line bundle over $\mM$.  We take a moment to spell this out in more concrete terms.

If $\mathcal{X}$ is any scheme or stack with a $\GG_m$-action, the structure sheaf admits a decomposition
$$ \mathcal{O}_{\mathcal{X}} \cong \bigoplus_{ t \in \ZZ} \mathcal{O}_{\mathcal{X}} (t) $$
where the sections of $\mathcal{O}_{\mathcal{X}}(t)$ consist of those functions $f$ on $\mathcal{X}$ satisfying
$$ f(z \cdot x) = z^t \cdot f(x). $$
One source of $\GG_m$-equivariant stacks arises from the stackification of commutative Hopf algebroids which are graded.  Suppose that $(T,\Xi)$ is a commutative Hopf algebroid with a grading, and let $\mathcal{X}$ be the associated stack:
$$ \mathcal{X} = \Spec(T)//\Spec{\Xi}. $$
In this setting, the grading on $T$ endows the scheme $\Spec(T)$ with a $\GG_m$-action, and the grading on $\Xi$ ensures that this $\GG_m$ action descends to the quotient, endowing $\mathcal{X}$ with the structure of a $\GG_m$-action.

In the case of the $\GG_m$-action on $\mathcal{M}^1$, we have
$$ \pi_* \mathcal{O}_{\mM^1}(1) = \omega. $$
 The cohomology ring
$H^*(\mM^1;\mathcal{O}_{\mM^1})$ inherits an additional integer grading; we
will write this bigraded ring as
$$ H^{s,t}(\mM^1;\mathcal{O}_{\mM^1}) := 
H^{s}(\mM^1;\mathcal{O}_{\mM^1}(t)),  $$
so there is an isomorphism
\[
  H^{s,t}(\mM^1;\mathcal{O}_{\mM^1}) \cong H^s(\mM;\omega^{\otimes t}).
\]
Similar statements hold for $\mM_i^1(\ell)$.  As such, if our interest
is in computing the $E_2$-term of the
elliptic spectral sequence for $\TMF_i(\ell)$, then it suffices to
study moduli problems in which we record a non-zero tangent vector
(or equivalently, a non-trivial invariant differential).  For the remainder of this section, all presentations of $\GG_m$-equivariant moduli stacks by Hopf algebroids shall be implicitly graded, with generators named ``$g_k$'' to implicitly lie in degree $k$.

The maps in the cosimplicial $E_\infty$ ring $Q(5)^\bullet$ arise
by evaluating the $\TMF$-sheaf $\mathcal{O}^{top}$ on maps
$\mM_0(5)\to \mM$ and $\mM_0(5)\to \mM_0(5)$.  Recall that the
Weierstrass Hopf algebroid $(A,\Gamma)$ stackifies to $\mM^1$; we
review the structure of $(A,\Gamma)$ in \ref{subsec:rep}.  In this
section we produce a Hopf algebroid $(B^1,\Lambda^1)$ representing
$\mM_0^1(5)$ and produce Hopf algebroid formulas for the maps in 
(the cohomology of) the 
semi-simplicial stack associated to $Q(\ell)^\bullet$.

\subsection{Representing $\mM_1(5)$}

In this section we
will give explicit presentations of $\mM_1(5)$ and $\mM_1^1(5)$.
Consider the rings
\[
\begin{aligned}
  B &\defined \ZZ[1/5, b,\Delta^{-1}]\\
  B^1 &\defined \ZZ[1/5, a_1,a_2,a_3,\Delta^{-1}]/(a_2^3+a_3^2-a_1a_2a_3)
\end{aligned}
\]
where $\Delta$ is given respectively by:
\begin{align*}
  \Delta(b) & = b^5(b^2-11b-1), \\
\Delta(a_1, a_2, a_3) & = -8a_1^2a_3^2a_2^2 + 20a_1a_3^3a_2 - a_1^4a_3^2a_2 - 11a_3^4 + a_1^3a_3^3.
\end{align*}
We have the following theorem.

\begin{thm}\label{thm:TN}
The stacks $\mM_1(5)$ and $\mM_1^1(5)$ are affine schemes, given by
\[
\begin{aligned}
  \mM_1(5) &= \Spec(B),\\
  \mM_1^1(5) &= \Spec(B^1).
\end{aligned}
\]
\end{thm}

\begin{proof}
We first use the techniques of \cite[\S 4.4]{Huse} (which is a
recapitulation of a method from \cite{Kubert}) to produce an explicit model
for $\mM_1^1(5)$ as an affine scheme.  The procedure is exhibited
graphically in Figure \ref{fig:TN}.

Suppose $(C,P)$ is a
$\Gamma_1(5)$-structure over a commutative ring $R$ in Weierstrass form with $P=(\alpha,\beta)$.  For
$r,s,t\in R$ and $\lambda\in R^\times$ let
$\varphi_{r,s,t,\lambda}$ denote the coordinate change
\[
\begin{aligned}
  x&\mapsto \lambda^{-2}x+r\\
  y&\mapsto \lambda^{-3}y+\lambda^{-2}sx+t.
\end{aligned}
\]
Move $P$ to $(0,0)$ via the coordinate change $\varphi_{-\alpha,0,-\beta,1}:(C,P)\to
(C_{\underline{a}'},(0,0))$ where $C_{\underline{a}'}$ has Weierstrass form
\[
  y^2+a_1'xy+a_3'y = x^3+a_2'x^2+a_4'x.
\]
(Note that $a_6'=0$ because $(0,0)$ is on the curve.)  Next eliminate
$a_4'$ by applying the transformation $\varphi_{0,-a_4'/a_3',0,1}$.
The result is a smooth Weierstrass curve
\begin{equation}\label{eqn:homTN}
  T^1(a_1,a_2,a_3): y^2 + a_1xy + a_3y = x^3+a_2x^2
\end{equation}
with $\Gamma_1(5)$-structure $(0,0)$ which we call the
\emph{homogeneous Tate normal form} of $(C,P)$.

\begin{figure}\label{fig:TN}
\begin{tikzpicture}
  \matrix[matrix of math nodes,row sep=0em,column sep=0em,font=\tiny]
  {|[name=arb]|\includegraphics[width=0.23\linewidth]{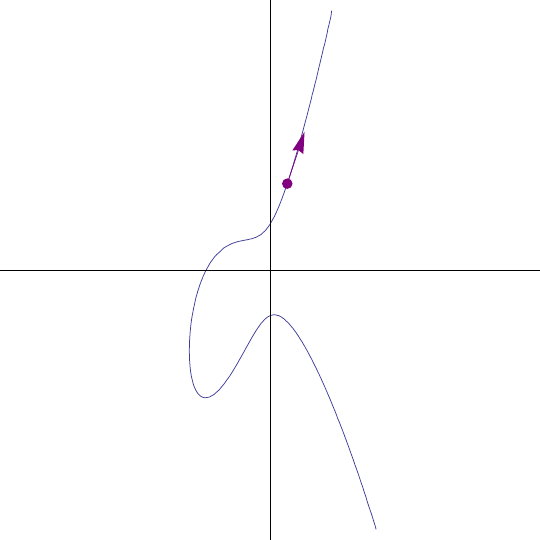}
    &|[name=00]|\includegraphics[width=0.23\linewidth]{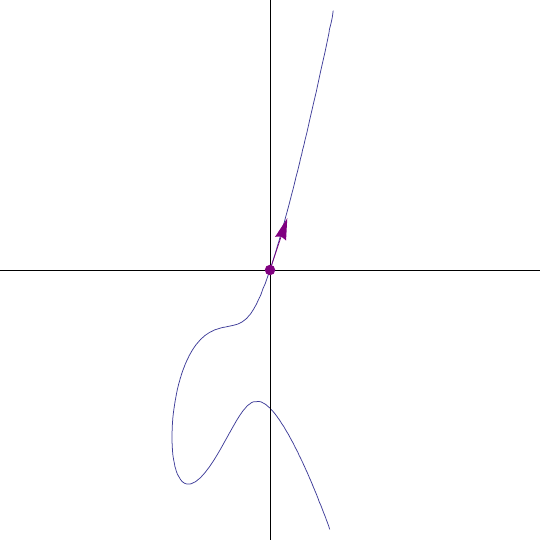}
  &|[name=hom]|\includegraphics[width=0.23\linewidth]{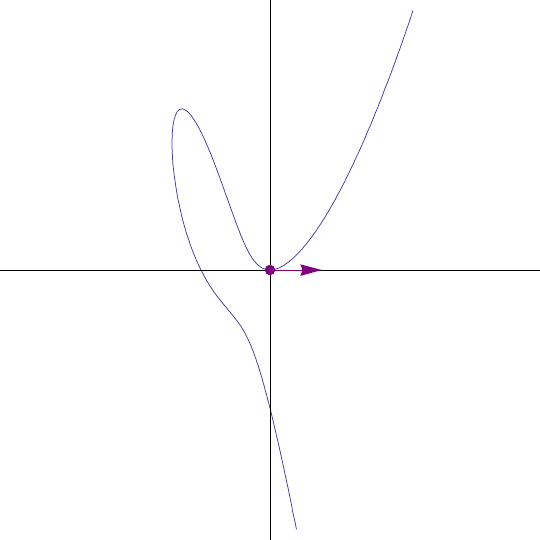}
  &|[name=TN]|\includegraphics[width=0.23\linewidth]{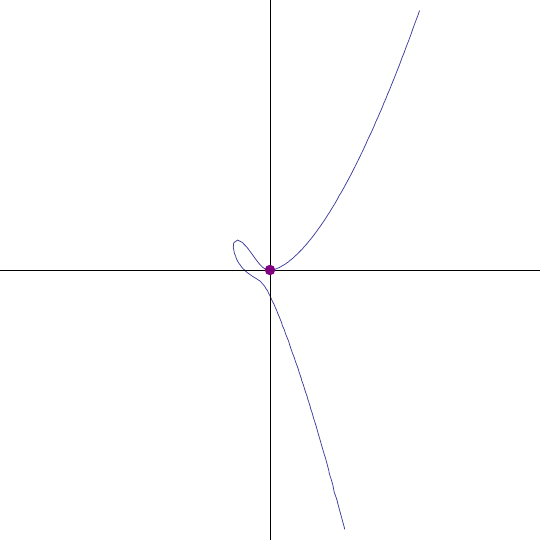}\\
  |[name=AA]|C &|[name=A]|C_{\mathbf{a}'}
  &|[name=a]|T^1:y^2+a_1xy+a_3y
  &|[name=bc]|T:y^2+(1-b)xy-by\\
  &&|[name=aLow]|= x^3+a_2x^2 &|[name=bcLow]|= x^3-bx^2\\};
  \path[->,font=\scriptsize]
  (arb.60) edge[bend left] node[above] {$\varphi_{-\alpha,0,-\beta,1}$} (00.120)
  (00.60) edge[bend left] node[above] {$\varphi_{0,-a_4'/a_3',0,1}$} (hom.120)
  (hom.60) edge[bend left] node[above] {$\varphi_{0,0,0,a_3/a_2}$} (TN.120);
\end{tikzpicture}
\caption{The procedure for putting a $\Gamma_1(5)$-structure in
  (homogeneous and non-homogeneous) Tate
  normal form}
\end{figure}

Since $(0,0)$ has order $5$ in $T^1(a_1,a_2,a_3)$ we must have
\begin{equation}\label{eqn:order5}
  [3](0,0) = [-2](0,0)
\end{equation}
where $[n]$ denotes the $\ZZ$-module structure of the elliptic curve
group law.  Explicitly expanding the left- and right-hand sides of
this equation in projective coordinates, we find that
\[
  (a_2(-a_1a_2a_3+a_3^2) : a_1a_2a_3^2-a_2^3a_3-a_3^3 : a_2^3) = (-a_2
  : 0 : 1).
\]
It follows that $a_1,a_2,a_3$ must satisfy
\begin{equation}\label{eqn:f1}
  a_2^3+a_3^2 = a_1a_2a_3
\end{equation}
in order for $(T^1(a_1,a_2,a_3), (0,0))$ to be a
$\Gamma_1(5)$-structure.  (The referee points out that one can also arrive at this condition by contemplating the geometric meaning of (\ref{eqn:order5}).)

We may compute the discriminant of
$T^1(a_1,a_2,a_3)$ as
\begin{equation}\label{eqn:discT1}
  \Delta = -8a_1^2a_3^2a_2^2 + 20a_1a_3^3a_2 - a_1^4a_3^2a_2 - 11a_3^4 + a_1^3a_3^3.
\end{equation}
Let $f^1(a_1,a_2,a_3) \defined a_2^3+a_3^2-a_1a_2a_3$ and let
\[
  B^1 \defined \ZZ[a_1,a_2,a_3,\Delta^{-1}]/(f^1).
\]
Then
\[
  \mM_1^1(5) = \Spec(B^1).
\]

We now consider $\Gamma_1(5)$-structures without distinguished
tangent vectors and produce a (non-homogeneous) Tate normal form which
is the univeral elliptic curve for $\mM_1(5)$.  Begin with a
$\Gamma_1(5)$-structure $(C,P)$ and change coordinates to put it in
homogeneous Tate normal form $T^1(a_1,a_2,a_3)$.  Now apply the
coordinate transformation $\varphi_{0,0,0,a_3/a_2}$.  (This
transformation is permissible because $(0,0)$ has order greater than
$3$.)  After applying the transformation, the coefficients of $y$ and
$x^2$ are equal.  Let
\begin{equation}\label{eqn:preTN}
  T(b,c) : y^2+(1-c)xy-by = x^3-bx^2.
\end{equation}
denote the resulting smooth Weierstrass curve.

Since $(0,0)$ has order $5$ we
know (\ref{eqn:order5}) holds; it follows that
\begin{equation}\label{eqn:f}
  b = c
\end{equation}
in (\ref{eqn:preTN}).  Abusing notation, let
\begin{equation}\label{eqn:TN}
  T(b) : y^2+(1-b)xy-by = x^3-bx^2;
\end{equation}
we call this the (\emph{non-homogeneous}) \emph{Tate normal form} of
$(C,P)$.  The discriminant of $T(b)$ is
\begin{equation}\label{eqn:discTb}
  \Delta = b^5(b^2-11b-1).
\end{equation}

Let
\[
  B \defined \ZZ[1/5, b,\Delta^{-1}].
\]
The preceding two paragraphs show that
\[
  \mM_1(5) = \Spec(B).
\]
\end{proof}

\begin{cor}\label{cor:a1u}
The moduli space $\mM_1^1(5)$ is represented by
\[
  \Spec(\ZZ[1/5, a_1,u^{\pm 1},\Delta^{-1}])
\]
where
\[
  \Delta = -11u^{12} + 64a_1u^{11} - 154a_1^2u^{10} + 195a_1^3u^9 - 135a_1^4u^8 + 46a_1^5u^7 - 4a_1^6u^6 - a_1^7u^5.
\]
\end{cor}
\begin{proof}
The rings in question are isomorphic via the homomorphism
\[
  B^1\to
\ZZ[1/5,a_1,u^{\pm 1},\Delta^{-1}]
\]
taking
\[
  a_1\mapsto a_1,~a_2\mapsto u(a_1-u),~a_3\mapsto u^2(a_1-u).
\]
(Note that $u$ corresponds to $a_3/a_2$, and both $a_2$ and $a_3$ are
invertible in $B^1$.)  This takes $T^1(a_1,a_2,a_3)$ to the curve
\[
  y^2 + a_1xy + u^2(a_1-u)y = x^3+u(a_1-u)x^2
\]
whose discriminant may be computed manually.
\end{proof}

The simple structure of $\mM_1(5)$ has an immediate topological
corollary that we record here.

\begin{cor}\label{cor:K2local}
The $K(2)$-localization of $\TMF_1(5)$ is a height 2 Lubin-Tate
spectrum for the formal group law $\widehat{T(b)}$ defined over $\FF_2$:
\[
  \TMF_1(5)_{K(2)} \simeq E_2(\FF_2,\widehat{T(b)}).
\]
\end{cor}
\begin{proof}
The $K(2)$-localization of $\TMF_1(5)$ is controlled by the
$\FF_2$-supersingular locus of $\mM_1(5)$, $\mM_1(5)^{ss}_{\FF_2}$.
The $2$-series of the formal group law for $T = T(b)$ takes the form
\[
  [2]_{\widehat{T}}(z) = 2z+(b-1)z^2+2bz^3+(b^2-2b)z^4+\cdots.
\]
(This is easily deduced from the standard formla for the formal group
law of a Weierstrass curve found, \emph{e.g.}, in
\cite[p.120]{Silverman}.) Hence $\widehat{T}$ is supersingular over $\FF_2$ if and only if
$b=1$.  Note that $\Delta(T(1)) = -11$, a unit in $\ZZ_2$ and $\FF_2$.  It follows that
\[
  \mM_1(5)^{ss}_{\FF_2} = \Spec(\FF_2).
\]

Let $E_2 = E_2(\FF_2,\widehat{T})$ with $\pi_0E_2 = \ZZ_2[[u_1]]$.  The map
\[
  \Spec \pi_0E_2 \to \mM_1(5)
\]
induced by
\[
\begin{aligned}
  B &\to \pi_0E_2\\
  b &\mapsto u_1+1.
\end{aligned}
\]
induces the $K(2)$-localization of $\TMF_1(5)$.
\end{proof}

\subsection{Representing maps $\mM_1^1(5)\to \mM^1$}
\label{subsec:rep}

 There are two important maps $\mM_1^1(5)\to \mM^1$ which we analyze.  On
the level of points, the
first is the forgetful map
\[
\begin{aligned}
  \mM_1^1(5) &\xrightarrow{f} \mM^1\\
  (C,P) &\mapsto C.
\end{aligned}
\]
The second is the quotient map
\[
\begin{aligned}
  \mM_1^1(5) &\xrightarrow{q} \mM^1\\
  (C,P) &\mapsto C/\langle P\rangle.
\end{aligned}
\]
Let $(A,\Gamma)$ denote the usual Weierstrass curve Hopf algebroid with
\[
  A = \ZZ[a_1,a_2,a_3,a_4,a_6, \Delta^{-1}],
  \Gamma = A[r,s,t]
\]
that stackifies to $\mM^1$.  (Note that $\Gamma$ does not have a
polynomial generator $\lambda$ precisely because the coordinate change
$\varphi_{r,s,t,\lambda}$ preserves tangent vectors if and only if
$\lambda=1$.)

\begin{thm}\label{thm:fqM1}
The morphisms $f$ and $q$ above are represented by
\[
\begin{aligned}
  A &\xrightarrow{f^*} B^1\\
  a_i &\mapsto
    \begin{cases}
      a_i &\text{if }i=1,2,3,\\
      0 &\text{if }i=4,6,
    \end{cases}
\end{aligned}
\]
and
\[
\begin{aligned}
  A&\xrightarrow{q^*} B^1,\\
  a_i &\mapsto
    \begin{cases}
      a_i &\text{if }i=1,2,3,\\
      5a_1^2a_2 - 10a_1a_3 - 10a_2^2 &\text{if }i=4,\\
      a_1^4a_2 - 2a_1^3a_3 - 12a_1^2a_2^2 + 19a_2^3 - a_3^2 &\text{if }i=6.
    \end{cases}
\end{aligned}
\]
The associated maps $\Gamma\to B^1$ take $r,s,t\mapsto 0$ since
$\mM^1_1(5)$ is a scheme.
\end{thm}

Computing $q$ requires that we find a Weierstrass curve representation
of $C/\langle P\rangle$ in terms of the Weierstrass coeffcients of
$C$.  This procedure is well-studied by number theorists under the name \emph{V\'{e}lu's formulae} (see \cite{Velu},
\cite[\S 2.4]{Kohel}) and is implemented in the computer algebra system
\texttt{Magma}.  In fact, if $\phi$ is an isogeny on $C$ in
Weierstrass form with kernel $H$, then V\'{e}lu's formulae compute
Weierstrass coefficients for the target of $\phi$ in terms of the
Weierstrass coefficients of $C$ and the defining equations of the subgroup
scheme $H$.  We briefly review the formulae here for reference.

Suppose $H<C$ is a finite subgroup with ideal sheaf generated by a
monic polynomial $\psi(x)$ where $C$ is a Weierstrass curve of the
form
\[
  y^2 + a_1xy + a_3y = x^3+a_2x^2+a_4x+a_6.
\]
For simplicity, assume that the isogeny $\phi:C \to
C/H$ has odd degree.  (The even degree case can be handled as a
separate case, but we will not need it in this paper.)  Write
\[
  \psi(x) = x^n - s_1x^{n-1} +\cdots +(-1)^n s_n.
\]
Then $C/H$ has Weierstrass equation
\[
  y_H^2 + a_1x_Hy_H + a_3y_H = x_H^3+a_2x_H^2 + (a_4-5t)x_H + (a_6-b_2t-7w)
\]
where
\[\begin{aligned}
  t &= 6(s_1^2-2s_2) + b_2s_1+nb_4,\\
  w &= 10(s_1^3-3s_1s_2+3s_3) + 2(b_2(s_1^2-2s_2)+3b_4s_1+nb_6,
\end{aligned}\]
and
\[\begin{aligned}
  b_2 &= a_1^2+4a_2,\\
  b_4 &= a_1a_3+2a_4,\\
  b_6 &= a_3^2+4a_6.
\end{aligned}\]
V\'{e}lu's formulae also give explicit equations for the isogeny
$\phi:(x,y)\mapsto (x_H,y_H)$, but they are cumbersome to write down
and we will not need them here.

\begin{proof}[Proof of Theorem~\ref{thm:fqM1}]
The representation of $f$ is obvious:  $T^1(a_1,a_2,a_3)$ is already
in Weierstrass form with $a_4,a_6=0$.

Consider the case of $C = T(a_1,a_2,a_3)$ with $H = \langle P\rangle$
an order $5$ subgroup.
Using the elliptic curve addition law we see that $H$ is the subgroup
scheme of $C$ cut out by the polynomial $x(x+a_2)$.  Putting this data
into V\'{e}lu's formulae, we find that $C/H$ has Weierstrass form
\begin{equation}\label{eqn:Velu}
\begin{aligned}
  y^2 + a_1xy + a_3y = x^3 &+ a_2x^2 \\
  &+(5a_1^2a_2 - 10a_1a_3 - 10a_2^2)x \\
  &+(a_1^4a_2 - 2a_1^3a_3 -
  12a_1^2a_2^2 + 19a_2^3 - a_3^2)
\end{aligned}
\end{equation}
from which our formula for $q$ follows.
\end{proof}

\subsection{Hopf algebroids for $\mM_0^1(5)$}

Recall that a $\Gamma_0(5)$-structure $(C,H)$ consists of an elliptic
curve $C$ along with a subgroup $H<C$ of order $5$.  Unlike the moduli
problem of $\Gamma_1(5)$-structures, $\mM_0(5)$ is not representable by a scheme.
Still, it is the case that $\mM_1(5)$ admits a $C_4 =
\FF_5^\times$-action such that $\mM_0(5)$ is the geometric quotient
$\mM_1(5)//\FF_5^\times$.  For $g\in \FF_5^\times$, $g$ takes $(C,P)$
to $(C,[\tilde{g}]P)$ for $\tilde{g}$ any lift of $g$ in $\ZZ$.
Similarly, we can write $\mM_0^1(5) = \mM_1^1(5)//\FF_5^\times$.

While it is typically easier to use this quotient stack presentation of $\mM_0(5)$ and $\mM_0^1(5)$ (and this will be the perspective we will be taking in the computations later in this paper), we will note that there is also a presentation of these moduli stacks by `$(r,s,t)$' Hopf algebroids.
Let $B^1$ be as before and define
\[
  \Lambda^1 \defined B^1[r,s,t]/\sim
\]
where $\sim$ consists of the relations
\[
\begin{aligned}
  3r^2 &= 2st+a_1rs+a_3s+a_1t-2a_2r,\\
  t^2 &= r^3+a_2r^2-a_1rt-a_3t,\\
  s^6 &= -3a_1s^5 + 9rs^4 + 3a_2s^4 - 3a_1^2s^4 + 4ts^3 \\ &\phantom{=}+ 
20a_1rs^3 + 6a_1a_2s^3 + 2a_3s^3 - a_1^3s^3 + 6a_1ts^2 \\ &\phantom{=}- 
27r^2s^2 - 18a_2rs^2 + 12a_1^2rs^2 - 3a_2^2s^2 + 3a_1^2a_2s^2 \\ &\phantom{=}+ 3a_1a_3s^2 - 12rts - 4a_2ts + 2a_1^2ts - 33a_1r^2s \\ &\phantom{=}- 
20a_1a_2rs - 6a_3rs + a_1^3rs - 3a_1a_2^2s - 2a_3a_2s \\ &\phantom{=}+ 
a_1^2a_3s + 4t^2 - 2a_1rt - 2a_1a_2t + 4a_3t + 27r^3 \\ &\phantom{=}+ 27a_2r^2 - 2a_1^2r^2 + 9a_2^2r - a_1^2a_2r - a_1a_3r.
\end{aligned}
\]

\begin{thm}\label{thm:HopfAlgd}
The rings $(B^1,\Lambda^1)$ form a Hopf algebroid
stackifying to $\mM_0^1(5)$.  The
structure maps are given by
\[
\begin{aligned}
  \eta_R(a_1) &= a_1+2s\\
  \eta_R(a_2) &= a_2+3r-s^2-a_1s\\
  \eta_R(a_3) &= a_3+2t+a_1r\\
  \psi(r) &= r\otimes 1 + 1\otimes r\\
  \psi(s) &= s\otimes 1 + 1\otimes s\\
  \psi(t) &= t\otimes 1 + s\otimes r + 1\otimes t.
\end{aligned}
\]
\end{thm}
\begin{proof}
The reader will note that the structure maps are identical to those
for the standard Weierstrass Hopf algebroid $(A,\Gamma)$.  The
relations $\sim$ are precisely those required so that $\varphi_{r,s,t,1}$ transforms $T^1(a_1,a_2,a_3)$ such that $a_2^3+a_3^2 = a_1a_2a_3$ into another homogeneous Tate normal curve.
\end{proof}

There are forgetful and quotient maps on $\mM_0^1(5)$ that on points take
\[
\begin{aligned}
   \mM_0^1(5)&\xrightarrow{f} \mM^1\\
   (C,H)&\mapsto C
\end{aligned}
\]
and
\[
\begin{aligned}
   \mM_0^1(5)&\xrightarrow{q} \mM^1\\
   (C,H)&\mapsto C/H.
\end{aligned}
\]
(We elide the tangent vectors for concision.)

\begin{cor}\label{cor:fqM1}
The maps $f$ and $q$ on $\mM_0^1(5)$ are represented by
\[
\begin{aligned}
   (A,\Gamma)&\xrightarrow{f^*} (B^1,\Lambda^1)\\
   a_i &\mapsto
   \begin{cases}
     a_i &\text{if }i=1,2,3,\\
     0 &\text{if }i=4,6
   \end{cases}\\
   r,s,t&\mapsto r,s,t
\end{aligned}
\]
and
\[
\begin{aligned}
  (A,\Gamma)&\xrightarrow{q^*} (B^1,\Lambda^1)\\
  a_i &\mapsto
    \begin{cases}
      a_i &\text{if }i=1,2,3,\\
      5a_1^2a_2 - 10a_1a_3 - 10a_2^2 &\text{if }i=4,\\
      a_1^4a_2 - 2a_1^3a_3 - 12a_1^2a_2^2 + 19a_2^3 - a_3^2 &\text{if }i=6.
    \end{cases}\\
   r,s,t&\mapsto r,s,t
\end{aligned}
\]
\end{cor}
\begin{proof}
This is a consequence of Theorems \ref{thm:fqM1} and \ref{thm:HopfAlgd}.
\end{proof}

\subsection{The Atkin-Lehner dual}\label{sec:AL}
We will now compute the Atkin-Lehner dual 
$$t:\mM_0^1(5)\to \mM_0^1(5). $$
Each $\Gamma_0(5)$-structure
$(C,H)$ can also be represented as a pair $(C,\phi)$ where $\phi:C\to C'$
has kernel $H$.  On points, the Atkin-Lehner dual takes
\[
\begin{aligned}
  \mM_0^1(5) &\xrightarrow{t} \mM_0^1(5)\\
  (C,\phi) &\mapsto (C/H,\widehat{\phi})
\end{aligned}
\]
where $\widehat{\phi}$ is the dual isogeny to $\phi$.

We can lift $t$ to stacks closely related to $\mM_1^1(5)$.  Recall
(\cite[\S 2.8]{KM}) that for each
$\Gamma_0(5)$-structure $(C,\phi)$ there is an associated
scheme-theoretic Weil pairing
\[
  \langle -,- \rangle_\phi: \ker \phi \times \ker \widehat{\phi} \to \mu_5.
\]
Choose a primitive fifth root of unity $\zeta$.  For a
$\Gamma_1(5)$-structure $(C,P)$ let $(C,\phi_P)$ denote the associated
$\Gamma_0(5)$-structure where $\phi_P:C\to C'$ is an isogeny with kernel
$\langle P\rangle$.  If we work in $\mM_1^1(5)_\zeta$,
i.e.~$\mM_1^1(5)$ considered as a $\ZZ[\frac{1}{5},\zeta]$-scheme, then there is a unique $Q\in \ker
\widehat{\phi_P}$ such that $\langle P,Q\rangle_\phi = \zeta$.  We
define
\[
  t_\zeta:\mM_1^1(5)_\zeta \to \mM_1^1(5)_\zeta
\]
in the obvious way so that $t_\zeta(C,P) = (C',Q)$.

The maps $t$ and $t_\zeta$ fit in the commutative diagram
\[\xymatrix{
  \mM_1^1(5)_\zeta \ar[r]^{t_\zeta}\ar[d] &\mM_1^1(5)_\zeta\ar[d]\\
  \mM_0^1(5)_\zeta \ar[r]_t &\mM_0^1(5)_\zeta
}\]
where the vertical maps take $(C,P)$ to $(C,\phi_P)$.

We can gain some computational control over $t$ via the following
method.  First, recall from Corollary \ref{cor:a1u} that for each homogeneous Tate normal curve
$T^1(a_1,a_2,a_3)$ there is a unit $u$ such that $a_2 = u(a_1-u)$ and
$a_3 = u^2(a_1-u)$.  Abusing notation, denote the same curve by $T^1(a_1,u)$,
and let $H$ denote the canonical cyclic subgroup of order 5 generated
by $(0,0)$.  The defining polynomial for $H$ is $x(x+u(a_1-u))$.
Denote the isogeny with kernel $H$ by $\phi$.  Note that the range of
$\phi$ is the curve $C/H$ given by V\'{e}lu's formulae in (\ref{eqn:Velu}).

Using Kohel's formulas \cite{Kohel} (as implemented by the computer algebra
system \texttt{Magma}), we can determine
that the kernel of $\widehat{\phi}$ is the subgroup scheme determined
by
\[
  f \defined x^2 + (a_1^2 - a_1u + u^2)x + \frac{1}{5}(a_1^4 - 7a_1^3u - 11a_1^2u^2 + 47a_1u^3 - 29u^4).
\]
Then over the ring
$R \defined \ZZ[\frac{1}{5},\zeta][a_1,u^\pm]$ the polynomial $f$ splits and we find that 
\[
  (\ker \widehat{\phi})(R) = \{\infty, (x_0,y_{00}),(x_0,y_{01}),(x_1,y_{10}),(x_1,y_{11})\} 
\]
where
\[
\begin{aligned}
  x_0 &= \frac{1}{5}(\zeta^3 + \zeta^2 - 2)a_1^2 + \frac{1}{5}(9\zeta^3 +
  9\zeta^2 + 7)a_1u + \frac{1}{5}(-11\zeta^3 - 11\zeta^2 - 8)u^2,\\
  x_1 &= \frac{1}{5}(-\zeta^3 - \zeta^2 - 3)a_1^2 + \frac{1}{5}(-9\zeta^3 -
  9\zeta^2 - 2)a_1u + \frac{1}{5}(11\zeta^3 + 11\zeta^2 + 3)u^2,\\
  y_{00} &= \frac{1}{5}(\zeta^2 + 2\zeta + 2)a_1^3 +
  \frac{1}{5}(\zeta^3 + 7\zeta^2 + 17\zeta + 5)a_1^2u \\
      &\phantom{=} + \frac{1}{5}(9\zeta^3 - 29\zeta^2 - 31\zeta - 14)a_1u^2 + 
    \frac{1}{5}(-11\zeta^3 + 22\zeta^2 + 11\zeta + 8)u^3,\\
y_{01} &= \frac{1}{5}(-\zeta^3 - 2\zeta^2 - 2\zeta)a_1^3 +
  \frac{1}{5}(-10\zeta^3 - 16\zeta^2 - 17\zeta - 12)a_1^2u \\
      &\phantom{=} + \frac{1}{5}(2\zeta^3 + 40\zeta^2 + 31\zeta + 
    17)a_1u^2 + \frac{1}{5}(11\zeta^3 - 22\zeta^2 - 11\zeta -
    3)u^3,\\
  y_{10} &= \frac{1}{5}(2\zeta^3 + \zeta + 2)a_1^3 +
  \frac{1}{5}(16\zeta^3 - \zeta^2 + 6\zeta + 4)a_1^2u \\
      &\phantom{=} + \frac{1}{5}(-40\zeta^3 - 9\zeta^2 - 38\zeta - 23)a_1u^2 + 
    \frac{1}{5}(22\zeta^3 + 11\zeta^2 + 33\zeta + 19)u^3,\\
  y_{11} &= \frac{1}{5}(-\zeta^3 + \zeta^2 - \zeta + 1)a_1^3 +
  \frac{1}{5}(-7\zeta^3 + 10\zeta^2 - 6\zeta - 2)a_1^2u \\
      &\phantom{=} + \frac{1}{5}(29\zeta^3 - 2\zeta^2 + 38\zeta + 
    15)a_1u^2 + \frac{1}{5}(-22\zeta^3 - 11\zeta^2 - 33\zeta - 14)u^3.
\end{aligned}
\]

Choose $(x_0,y_{00})$ as a preferred generator of $\widehat{H}$.  Let
$\zeta' = \langle (0,0), (x_0,y_{00})\rangle_\phi$.  Then
applying the method of Theorem \ref{thm:TN} we can put
$(C/H,(x_0,y_{00}))$ in homogeneous Tate normal form.  What we find is
a curve $T^1(t_{\zeta'}^*(a_1),t_{\zeta'}^*(u))$ with
\begin{equation}\label{eqn:AL}
\begin{aligned}
  t_{\zeta'}^*(a_1) &= \frac{1}{5}(-8\zeta^3 - 6\zeta^2 - 14\zeta - 7)a_1 +
  \frac{1}{5}(14\zeta^3 - 2\zeta^2 + 12\zeta + 6)u,\\
  t_{\zeta'}^*(u) &= \frac{1}{5}(-\zeta^3 - 7\zeta^2 - 8\zeta - 4)a_1 + \frac{1}{5}(8\zeta^3 + 6\zeta^2 + 14\zeta + 7)u.
\end{aligned}
\end{equation}

\begin{rmk}\label{rmk:zetaChoice}
We could produce similar formulas for any of the $(x_i,y_{ij})$ and
these would correspond to different choices of $\zeta'$ for the
Atkin-Lehner dual on $\Gamma_1(5)$-structures.  The applications below
will be invariant of this choice.
\end{rmk}

Equation (\ref{eqn:AL}) permits a description of the Atkin-Lehner dual
on the ring of $\Gamma_0(5)$-modular forms.  We refer the reader to
\cite[\S 3.1.1]{Hida} for a thorough description of modular forms as
global sections.  Recall briefly that for a
congruence subgroup $\Gamma\le SL_2(\ZZ)$, level $\Gamma$-modular
forms are precisely the global sections of (the tensor powers of) the
dualizing sheaf $\omega^{\otimes *}$ on the moduli stack $\mM(\Gamma)$
of level $\Gamma$-structures,
\[
  MF(\Gamma) = H^0(\mM(\Gamma);\omega^{\otimes *})
\]
Let $MF(\Gamma_1(5))_\zeta$ denote the ring of $\Gamma_1(5)$-modular
forms over the ring $\ZZ[\frac{1}{5},\zeta]$; it is isomorphic to
$\ZZ[\frac{1}{5},\zeta][a_1,u^\pm,\Delta^{-1}]$ since $\mM_1(5)$ is a scheme.  Then
\[
  MF(\Gamma_0(5)) = (MF(\Gamma_1(5))_\zeta^{Gal})^{\FF_5^{\times}}
\]
where $Gal$ denotes the copy of $\FF_5^\times$ acting on coefficients.

\begin{thm}\label{thm:MFAL}
The map $t^*:MF(\Gamma_0(5))\to MF(\Gamma_0(5))$ induced by the
Atkin-Lehner dual is the restriction of the unique map on
$MF(\Gamma_1(5)_\zeta)$ determined by (\ref{eqn:AL}).
\end{thm}

\section{The homotopy groups of $\TMF_0(5)$}
\label{sec:HFP}

By \'etale descent along the cover
$$ \mM_1(5) \rightarrow \mM_1(5)//\FF_5^\times = \mM_0(5). $$
we have $\TMF_0(5) \simeq \TMF_1(5)^{h\FF_5^\times}$.  
and we may thus compute the associated 
homotopy point spectral sequence
\[
  E_2^{s,t} = H^s(\FF_5^\times;\pi_t \TMF_1(5)) \implies \pi_{t-s} \TMF_0(5).
\]
The referee indicates that the first computation of this spectral sequence actually dates back to as early as 2003, with calculations of Mahowald and Rezk.  Hill, Hopkins, and Ravenel computed $\pi_* \TMF_0(5)$ in \cite{HHRunpub}.
As a self-contained homotopy fixed point spectral sequence computation of $\pi_* \TMF_0(5)$ is not yet available in the literature, we reproduce it in this section (though we note that the homotopy fixed point spectral sequence is actually a localization of the slice spectral sequence, and therefore the structure of this spectral sequence can actually be culled from \cite{HHRunpub}).

\subsection{Computation of the $E_2$-term}\label{sec:E2comp}

Consider the representation of $\mM^1_1(5)$ implicit in Corollary
\ref{cor:a1u}.  In the context of spectral sequence computations, we
will let $x = u$ and let $y = a_1-u$.  Let $\sigma$ denote the reduction of
2 in $\FF_5^\times$, a generator.

\begin{lemma}\label{lemma:sigmaAction}
The action of
$\FF_5^\times$ on $\pi_*\TMF_1(5) = \ZZ[1/5,x,y,\Delta^{-1}]$ is
determined by
\begin{equation}\label{eqn:action}\begin{aligned}
  \sigma \cdot x &= y\\
  \sigma \cdot y &= -x.
\end{aligned}\end{equation}
\end{lemma}
\begin{proof}
Consider the Tate normal curve $T$ with $a_1 = x+y$, $a_2 = xy$, and $a_3
= x^2y$.  (This is the Tate normal curve of Corollary \ref{cor:a1u}
under our coordinate change $x=u$, $y = a_1-u$.)  We can compute
$[2](0,0) = (-xy,xy^2)$.  The lemma then amounts to noting that the
Tate normal curve associated with the $\Gamma_1(5)$-structure
$((-xy,xy^2), T)$ has $a_1 = y-x$, $a_2 = -xy$, $a_3 = -xy^2$.
\end{proof}

Note that we may manually compute the discriminant as
\[
  \Delta = x^5y^5(x^2-11xy-y^2),
\]
so $x$ and $y$ are invertible elements of $\pi_*\TMF_1(5)$.

\begin{thm}\label{thm:HFPE2}
The $E_2$-term of the homotopy fixed point spectral sequence for
$\TMF_0(5)$ is given by
\[
 H^*(\FF_5^\times;\pi_*\TMF_1(5)) = \ZZ[1/5][b_2,b_4,\delta,\eta,\nu,\gamma,\xi,\Delta^{-1}]/\sim
\]
where $ \Delta = \delta^2(b_4-11\delta)$ and  $\sim$ consists of the relations

\begin{figure}
\includegraphics[width=\textwidth]{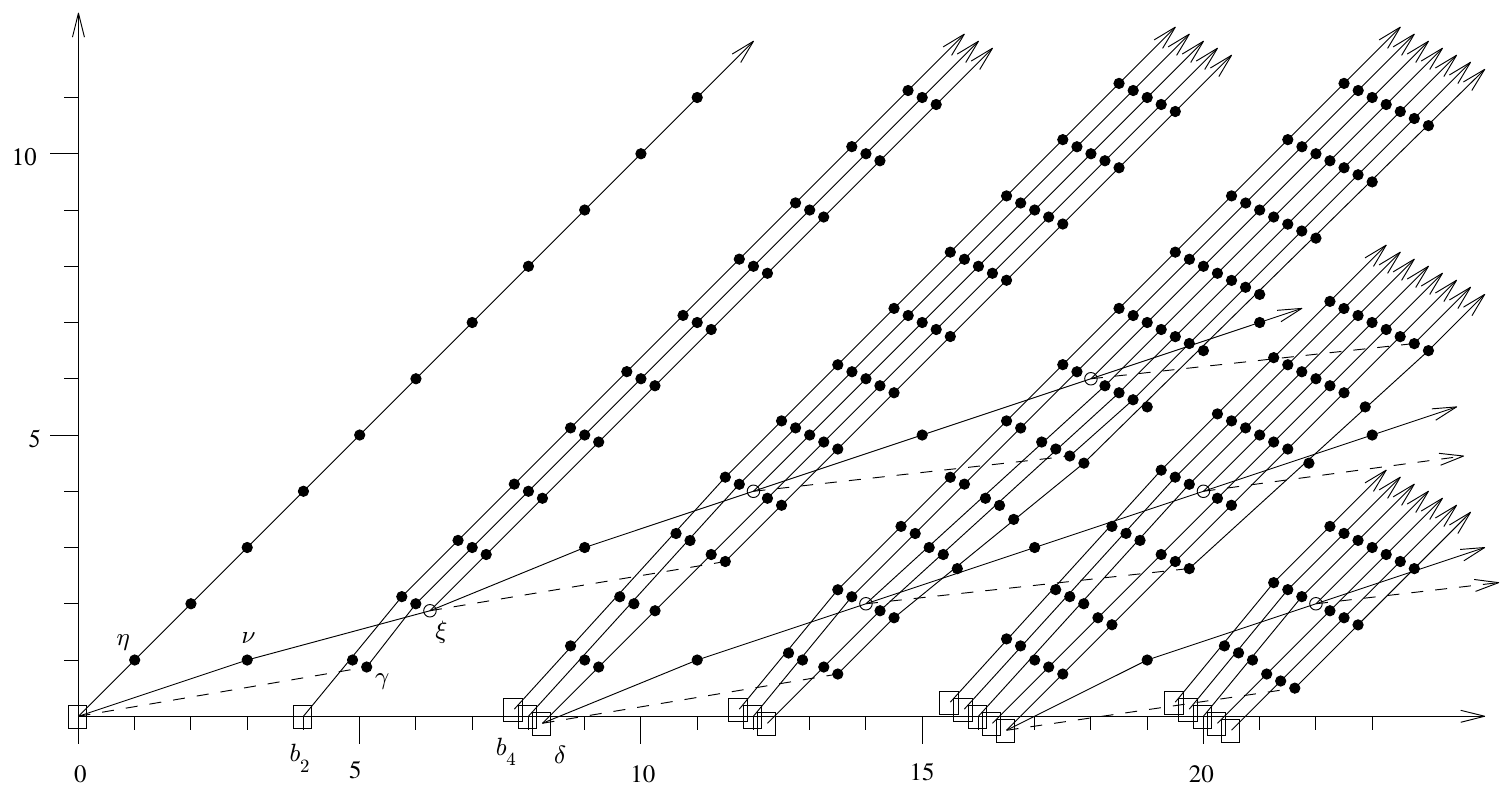}
\caption{A delocalization of the $E_2$-term of the homotopy fixed point spectral sequence for $\TMF_0(5)$ (the actual $E_2$-term is obtained from this figure by inverting $\Delta$).}\label{fig:HFPSSE2}
\end{figure}

\begin{center}
\begin{tabular}{cc}
\begin{minipage}{0.4\linewidth}
\[\begin{aligned}
  b_4^2 &= b_2^2 \delta-4\delta^2,\\
  2\eta &= 0,\\
  2\nu &= 0,\\
  2\gamma &= 0,\\
  4\xi &= 0,\\
  \nu^2 &= 2\xi,\\
  \gamma^2 &= (b_2^2+\delta)\eta^2,
\end{aligned}\]
\end{minipage}
&
\begin{minipage}{0.4\linewidth}
\[\begin{aligned}
  \eta\nu &= 0,\\
  b_2\nu &= 0,\\
  b_2\xi &= \delta\eta^2,\\
  \nu\gamma &= 0,\\
  b_4\xi &= b_2^2\xi + 2\delta\xi + \delta\eta\gamma,\\
  b_4\nu &= 0,\\
  b_4\gamma &= (b_4+\delta)b_2\eta, \\
 \gamma b_2 & = \eta (b_2^2 + b_4). \\
\end{aligned}\]
\end{minipage}
\end{tabular}
\end{center}
The generators lie in bidegrees $(t-s,s)$:
\begin{align*}
|b_2| & = (4,0), \\
|b_4| = |\delta| & =
(8,0), \\
|\eta| & = (1,1), \\
|\nu|  & = (3,1), \\
|\gamma|  & = (5,1), \\ 
|\xi| & = (6,2).
\end{align*}
\end{thm}

Figure~\ref{fig:HFPSSE2} shows a picture of the subring of the $E_2$-term of the homotopy fixed point spectral sequence for $\TMF_0(5)$ generated (as a $\ZZ[1/5]$-algebra) by 
$$ b_2, b_4, \delta, \eta, \nu, \gamma, \xi. $$
The full $E_2$-term is obtained after inverting $\Delta$.  Here and elsewhere in this paper, we use boxes $\square$ to represent $\ZZ$'s (or $\ZZ[1/5]$'s in this case), filled circles $\bullet$ to represent $\ZZ/2$'s, and open circles $\circ$ to represent $\ZZ/4$'s.

The proof of Theorem~\ref{thm:HFPE2} is a routine but fairly involved calculation following from
(\ref{eqn:action}).  We will establish this theorem with a series of lemmas.  Let $T_*$ denote the graded subring of $\pi_*\TMF_1(5)$ generated by $x$ and $y$, so that 
$$ \pi_* \TMF_1(5) = T_*[\Delta^{-1}]. $$
For a $\FF_5^\times$-module $M$, we shall use $H^*(M)$ to denote $H^*(\FF_5^\times; M)$.

The first step is to determine the structure of $T_*$ as an
$\FF_5^\times$-module.
We begin by setting some notation for $\FF_5^\times$-modules.  Let
$\ZZ[1/5]$ denote the $\FF_5^\times$-module with trivial action, let $\ZZ[1/5](-1)$ denote $\ZZ[1/5]$ with the sign action
$\sigma \cdot n = -n$, let $\tau = \ZZ[1/5]^2$
with the twist action $\sigma\cdot (m,n) = (n,m)$, and let $\psi =
\ZZ[1/5]^2$ with the cycle action $\sigma\cdot (m,n) = (n,-m)$.  

\begin{lemma}
The graded ring $T_*$ admits the following additive decomposition as an $\FF_5^\times$-module:
\begin{gather*}
T_{8n} =  \tau \{ x^{4n}, x^{4n-1}y,\ldots,x^{2n+1}y^{2n-1} \} \oplus \ZZ[1/5]\{x^{2n}y^{2n} \}, \\ 
T_{8n+4} = \tau \{ x^{4n+2},x^{4n+1}y,\ldots,x^{2n+2}y^{2n} \} \oplus \ZZ[1/5](-1) \{ x^{2n+1}y^{2n+1} \}, \\
T_{4n+2} = \psi \{ x^{2n+1},x^{2n}y,\ldots,x^{n+1}y^n \}.
\end{gather*}
\end{lemma}

Define the following $\FF_5^\times$-invariants in $T_*$:
\begin{align*}
b_2 & \defined x^2+y^2, \\
b_4 & \defined x^3y-xy^3, \\
\delta & \defined x^2y^2. 
\end{align*}
(Warning:  the $b_2$ and $b_4$ here are not related to the $b_2$ and
$b_4$ mentioned in relation to V\'{e}lu's formulae, or the $b_2$ and $b_4$ traditionally used in the theory of elliptic curves.)  Note that
$\delta$ is almost a cube root of $\Delta$: we have
$$ \Delta = \delta^2(b_4-11\delta). $$
The following lemma is fairly easily checked.

\begin{lemma}
The ring of $\FF_5^\times$-invariants of $T_*$ admits the following presentation:
$$ H^0(T_*) = \ZZ[1/5][b_2, b_4, \delta]/(b_4^2 - b_2^2\delta + 4\delta^2). $$  
\end{lemma}

We now turn our attention to the higher cohomology of $T_*$.  The following lemma gives an additive description of these cohomology groups, as a module over
$$ H^*(\ZZ[1/5]) = \ZZ[1/5,\beta]/(4\beta) $$
(where $\beta$ lies in $H^2$).

\begin{lemma}\label{eq:addHFPE2}
There is an \emph{additive} isomorphism of $H^*(\ZZ[1/5])$-modules:
$$ H^*(T_*) \cong
\ZZ[1/5][b_2, b_4, \delta, \eta, \nu, \gamma, \beta]/\sim'
$$
where $\sim'$ consists of the relations
\begin{center}
\begin{tabular}{cc}
\begin{minipage}{0.4\linewidth}
\[\begin{aligned}
  b_4^2 &= b_2^2 \delta-4\delta^2,\\
  2\eta &= 0,\\
  2\nu &= 0,\\
  2\gamma &= 0,\\
  4\beta &= 0,\\
  2 b_2\beta & = 0, \\
  2 b_4 \beta & = 0, \\
  (*) \nu^2 &= 0,\\
  (*) \gamma^2 &= 0,
\end{aligned}\]
\end{minipage}
&
\begin{minipage}{0.4\linewidth}
\[\begin{aligned}
  (*) \eta\nu &= 0,\\
  (*) b_2\nu &= 0,\\
  (*) \eta^2 &= 0,\\
  (*) \nu\gamma &= 0,\\
  (*) \eta\gamma & = 0,\\
  (*) b_4\nu &= 0,\\
  (*) b_4\gamma &= 0, \\
 (*) \gamma b_2 & = 0. 
\end{aligned}\]
\end{minipage}
\end{tabular}
\end{center}
\end{lemma}

Note, the relations marked $(*)$ in the preceding lemma are not actual multiplicative relations in $H^*(T_*)$, they just yield the correct additive answer.  To properly compute the ring structure of $H^*(T_*)$, we need to replace these `fake' relations with true relations.

\begin{proof}
The invariants introduced in the previous lemma allow for a more convenient additive description of $T_*$ as an $\FF_5^\times$-module:
\begin{gather*}
T_{8n} =  \tau\{ x^2 \} \{ b_2^{2n-1}, b_4b_2^{2n-3}, \delta b_2^{2n-3}, b_4\delta b_2^{2n-5}, \delta^2b_2^{2n-5}, \ldots \} \\
\oplus \tau\{ x^3y\}\{ \delta^{n-1} \} \oplus \ZZ[1/5]\{\delta^n \}, \\ 
T_{8n+4} = \tau \{x^2\}\{ b_2^{2n}, b_4b_2^{2n-2}, \delta b_2^{2n-2}, b_4\delta b_2^{2n-4}, \delta^2b_2^{2n-4}, \ldots \}  \\
\oplus \ZZ[1/5](-1) \{ xy \} \{\delta^n \}, \\
T_{8n+2} = \psi \{ x\}\{ b_2^{2n}, b_4b_2^{2n-2}, \delta b_2^{2n-2}, b_4\delta b_2^{2n-4}, \delta^2b_2^{2n-4}, \ldots \}, \\
T_{8n+6} = \psi \{ x\}\{ b_2^{2n+1}, b_4b_2^{2n-1}, \delta b_2^{2n-1}, b_4\delta b_2^{2n-3}, \delta^2b_2^{2n-3}, \ldots \} \\
\oplus \psi\{x^3\}\{\delta^n\}. 
\end{gather*}

To compute the higher cohomology $H^*(T_*)$ we begin by noting
that
\[\begin{aligned}
  H^*(\ZZ[1/5]) &= \ZZ[1/5][\beta]/4\beta,\\
  H^*(\ZZ[1/5](-1)) &= \ZZ[1/5][\beta]/2\beta[1],\\
  H^*(\tau) &= \ZZ[1/5][\beta]/2\beta,\\
  H^*(\psi) &= \ZZ[1/5][\beta]/2\beta[1]
\end{aligned}\]
where $\beta$ has cohomological degree 2, $[1]$ denotes a cohomological degree
shift by $1$, and each cohomology ring has
the obvious $H^*(\ZZ[1/5])$-module structure.  
We define
\begin{align*}
\eta & \in H^{1}(\psi\{x\}), 
\\
\nu & \in H^1(\ZZ[1/5](-1)\{xy\}),
\\
\gamma & \in H^1(\psi\{x^3\})
\end{align*}
to be the unique non-trivial elements in their respective cohomology groups.
We then have the following additive presentation of $H^*(T_*)$.
\begin{gather*}
H^*(T_{8n}) =  \ZZ[1/5][\beta]/(2\beta) \{ b_2^{2n}, b_4b_2^{2n-2}, \delta b_2^{2n-2}, b_4\delta b_2^{2n-4}, \delta^2b_2^{2n-4}, \ldots b_4 \delta^{n-1} \} \\
\oplus \ZZ[1/5][\beta]/(4\beta)\{\delta^n \}, \\ 
H^*(T_{8n+4}) = \ZZ[1/5][\beta]/(2\beta)\{ b_2^{2n+1}, b_4b_2^{2n-1}, \delta b_2^{2n-1}, b_4\delta b_2^{2n-3}, \delta^2b_2^{2n-3}, \ldots \}  \\
\oplus \ZZ[1/5][\beta]/(2\beta) \{\nu \delta^n \}, \\
H^*(T_{8n+2}) = \ZZ[1/5][\beta]/(2\beta)\{ \eta b_2^{2n}, \eta b_4b_2^{2n-2}, \eta \delta b_2^{2n-2}, \eta b_4\delta b_2^{2n-4}, \eta \delta^2b_2^{2n-4}, \ldots \}, \\
H^*(T_{8n+6}) = \ZZ[1/5][\beta]/(2\beta) \{ \eta b_2^{2n+1}, \eta b_4b_2^{2n-1}, \eta \delta b_2^{2n-1}, \eta b_4\delta b_2^{2n-3}, \eta \delta^2b_2^{2n-3}, \ldots \} \\
\oplus \ZZ[1/5][\beta]/(2\beta)\{\gamma \delta^n\}. 
\end{gather*}
The statement of the lemma follows.
\end{proof}

The following proposition fills in the multiplicative structure missing from the previous lemma.

\begin{prop}\label{prop:HFPE2}
There is an isomorphism of rings
$$ H^*(T_*) \cong
\ZZ[1/5][b_2, b_4, \delta, \eta, \nu, \gamma, \beta]/\sim
$$
where $\sim$ consists of the relations
\begin{center}
\begin{tabular}{cc}
\begin{minipage}{0.4\linewidth}
\[\begin{aligned}
  b_4^2 &= b_2^2 \delta-4\delta^2,\\
  2\eta &= 0,\\
  2\nu &= 0,\\
  2\gamma &= 0,\\
  4\beta &= 0,\\
  \nu^2 &= 2\delta\beta,\\
  \gamma^2 &= (b_2^2+\delta)\eta^2, \\
  \eta\nu &= 0,
\end{aligned}\]
\end{minipage}
&
\begin{minipage}{0.4\linewidth}
\[\begin{aligned}
  b_2\nu &= 0,\\
  \eta^2 &= b_2\beta,\\
  \nu\gamma &= 0,\\
  \eta\gamma & = (b_4+b_2^2+2\delta)\beta,\\
  b_4\nu &= 0,\\
  b_4\gamma &= (b_4+\delta)b_2\eta, \\
 b_2\gamma & = (b_2^2+b_4)\eta. 
\end{aligned}\]
\end{minipage}
\end{tabular}
\end{center}
\end{prop}

Note that we are able to drop the relations
\begin{align*}
  2 b_2\beta & = 0, \\
  2 b_4 \beta & = 0
\end{align*}
appearing in the left column of Lemma~\ref{eq:addHFPE2}, as they follow from the relations
\begin{align*} 
\eta^2 &= b_2\beta,\\
\eta\gamma & = (b_4+b_2^2+2\delta)\beta,
\end{align*}
respectively.
 
\begin{proof}
The following multiplicative relations are immediately deduced from dimensional considerations:
\begin{align*}
\eta\nu & = 0, \\
b_2 \nu & = 0, \\
\nu\gamma & = 0. \\
\end{align*}
Moreover, the ring structure on $T_*$ restricts to give a pairing
$$ H^1(\ZZ[1/5](-1)\{xy\}) \otimes H^0(\tau\{x^2\}) \rightarrow H^1(\tau\{x^3y\}) = 0 $$
which implies
$$ \nu b_4 = 0. $$

In order to determine most of the remaining relations, we observe that 
\begin{align*}
H^0(T_2/2) & = \FF_2\{ v_1 \}, \\
H^0(T_4/2) & = \FF_2\{ v_1^2, \delta^{1/2} \}, \\
H^0(T_6/2) & = \FF_2\{ v_1^3, v_1\delta^{1/2} \}
\end{align*}
with
\begin{align*}
v_1 & := x + y, \\
\delta^{1/2} & := xy.
\end{align*}
Note that the mod $2$ reductions of $b_2$, $b_4$, and $\delta$ are $v_1^2$, $v_1^2\delta^{1/2}$, and $(\delta^{1/2})^2$, respectively (and this explains the notation ``$\delta^{1/2}$'').
It follows easily from the long exact sequence
$$ \cdots \rightarrow H^0(T_*) \xrightarrow{\cdot 2} H^0(T_*) \rightarrow H^0(T_*/2) \xrightarrow{\partial} H^1(T_*) \xrightarrow{\cdot 2} \cdots $$
that 
\begin{align*}
\eta & = \partial(v_1), \\
\nu & = \partial(\delta^{1/2}), \\
\gamma & = \partial(v_1^3 + \delta^{1/2} v_1). 
\end{align*}
We deduce
\begin{align*}
b_4 \gamma & = \partial((v_1^2\delta^{1/2})(v_1^3 + \delta^{1/2}v_1)) \\
& = \partial ((v_1^2 \delta^{1/2}+\delta)v_1^2\cdot v_1) \\
& = (b_4+\delta)b_2\eta 
\end{align*}
and
\begin{align*}
b_2 \gamma & = \partial(v_1^2(v_1^3+v_1\delta^{1/2})) \\
& = \partial((v_1^4+v_1^2\delta^{1/2})v_1) \\
& = (b_2^2+b_4)\eta.
\end{align*}

To obtain the relation involving $\nu^2$, we note that from the exact sequence
$$ 
\xymatrix@C-1em@R-2em{
H^1(\ZZ[1/5](-1)\{x^3y^3\}) \ar[r]_{\cdot 2} &
H^1(\ZZ[1/5](-1)\{x^3y^3\}) \ar[r] & H^1(\FF_2\{x^3y^3\})
\\
\FF_2\{\delta\nu\} \ar@{=}[u] &
\FF_2\{\delta\nu\} \ar@{=}[u] &
}
$$
that the mod $2$ reduction of $\delta \nu$ is non-trivial in $H^1(T_*/2)$.  From this it follows that $\delta^{1/2} \nu$ is non-trivial in $H^1(T_*/2)$, and in particular, it must generate
$$ H^1(\FF_2\{ x^2 y^2 \}) \cong \FF_2. $$
It then follows from the long exact sequence
$$ 
\xymatrix@C-1em@R-2em{
H^1(\ZZ[1/5]\{x^2y^2\}) \ar[r] & H^1(\FF_2\{x^2y^2\}) \ar[r]_-{\partial} & H^2(\ZZ[1/5]\{x^2y^2\}) \ar[r]_{\cdot 2} & H^2(\ZZ[1/5]\{ x^2 y^2\})  
\\
0 \ar@{=}[u] & \FF_2\{ \delta^{1/2} \nu \} \ar@{=}[u] & \ZZ/4\{\delta\beta\} \ar@{=}[u] & 
\ZZ/4\{\delta\beta\} \ar@{=}[u]
}
$$
that
\begin{align*}
2\delta\beta & = \partial(\delta^{1/2} \nu) \\
& = \nu^2.
\end{align*}
A similar argument handles the relation involving $\eta^2$.  We note that from the exact sequence
$$ 
\xymatrix@C-1em@R-2em{
H^1(\psi\{b_2 x\}) \ar[r]_{\cdot 2} & H^1(\psi\{b_2 x\}) \ar[r] & H^1(\psi/2\{b_2 x\})
\\
\FF_2\{b_2 \eta\} \ar@{=}[u] & \FF_2\{b_2 \eta\} \ar@{=}[u] &
}
$$
that the mod $2$ reduction of $b_2 \eta$ is non-trivial in $H^1(T_*/2)$.  From this it follows that $v_1 \eta$ is non-trivial in $H^1(T_*/2)$, and in particular, it must generate
$$ H^1(\tau/2\{ x^2 + xy\}) \cong \FF_2. $$
It then follows from the long exact sequence
$$ 
\xymatrix@C-1em@R-2em{
H^1(\tau\{x^2+xy\}) \ar[r] & H^1(\tau/2\{x^2+xy\}) \ar[r]_-{\partial} & H^2(\tau\{x^2+xy\}) \ar[r]_{\cdot 2} & H^2(\tau\{ x^2+xy\})  
\\
0 \ar@{=}[u] & \FF_2\{ v_1 \eta \} \ar@{=}[u] & \FF_2\{\beta b_2\} \ar@{=}[u]
}
$$
that
\begin{align*}
\beta b_2 & = \partial(v_1 \eta) \\
& = \eta^2.
\end{align*}
The relation involving $\gamma^2$ now follows from the fact that multiplication by $b_2^2$ gives an injection
$$ \cdot b_2^2: H^2(T_{12}) \hookrightarrow H^2(T_{20}) $$
and we have
\begin{align*}
b_2^2 \gamma^2 & = \eta^2(b_2^2+b_4)^2 \\
& = b_2^2\eta^2(b_2^2 + \delta).
\end{align*}

The only relation left is the one involving $\eta\gamma$.
To this end we have the following $1$-cochain representatives, whose values on $\sigma^i \in \FF_5^\times$ are displayed below.
\begin{center}
\begin{tabular}{c|cccc}
$g$ & 1 & $\sigma$ & $\sigma^2$ & $\sigma^3$ \\
\hline
$\eta(g)$ & 0 & $x$ & $x+y$ & $y$ \\ 
$\gamma(g)$ & 0 & $x^3$ & $x^3+y^3$ & $y^3$
\end{tabular}
\end{center}

Each of these $1$-cochains $\phi(g)$ satisfies the $1$-cocycle condition
$$ (\delta \phi)(g_1, g_2)  = g_1\phi(g_2) - \phi(g_1g_2) + \phi(g_2) = 0. $$
We also record a $2$-cocycle $\beta(g_1,g_2)$ which represents $\beta$; its values on $(g_1, g_2)$ are recorded in the following table.
\begin{center}
\begin{tabular}{c|cccc}
${}_{g_2} \backslash {}^{g_1}$ & 1 & $\sigma$ & $\sigma^2$ & $\sigma^3$ \\
\hline
$1$ & 0 & 0 & 0 & 0 \\
$\sigma$ & 0 & 0 & 0 & 1 \\
$\sigma^2$ & 0 & 0 & 1 & 1 \\
$\sigma^3$ & 0 & 1 & 1 & 1 \\
\end{tabular}
\end{center}
Recall for $1$-cocycles $\phi(g)$ and $\psi(g)$, the explicit chain-level formula for the $2$-cocycle $\phi \cup \psi$ (see for instance, \cite{AdemMilgram}):
$$ (\phi \cup \psi)(g_1, g_2) = (g_1 \phi(g_2))\psi(g_1). $$
Using our explicit cochain representatives, we compute that  $\eta\gamma + \beta(b_4 + b_2^2- 2\delta)$ is represented by the $2$-cocycle $\psi(g_1, g_2)$ whose values are given by the following table.
\begin{center}
\begin{tabular}{c|cccc}
${}_{g_2} \backslash {}^{g_1}$ & 1 & $\sigma$ & $\sigma^2$ & $\sigma^3$ \\
\hline
$1$ & 0 & 0 & 0 & 0 \\
$\sigma$ & 0 & $x^3y$ & $-x^4 - xy^3$ & $x^4+x^3y-xy^3$ \\
$\sigma^2$ & 0 & $-x^4+x^3y$ & $-2xy^3$ & $x^4+x^3y$ \\
$\sigma^3$ & 0 & $x^3y - xy^3 + y^4$ & $x^4 - xy^3$ & $x^4 + x^3y+y^4$ \\
\end{tabular}
\end{center}
This $2$-cocycle is seen to be the coboundary of the following $1$-cochain $\phi$:
\begin{center}
\begin{tabular}{c|cccc}
$g$ & 1 & $\sigma$ & $\sigma^2$ & $\sigma^3$ \\
\hline
$\phi(g)$ & 0 & $-xy^3$ & $-xy^3$ & $x^4-xy^3$ \\ 
\end{tabular}
\end{center}
\end{proof}

We can now deduce Theorem~\ref{thm:HFPE2} from the preceding proposition by observing that
$$ H^*(\pi_*\TMF_1(5)) = H^*(T_*)[\Delta^{-1}]. $$
Since inverting $\Delta$ inverts $\delta$, we can replace the generator $\beta$ with the generator
$$ \xi := \beta\delta. $$
The relations of Theorem~\ref{thm:HFPE2} are then easily seen to be equivalent to those of the preceding proposition after inverting $\Delta$.
The authors find it easier to work with the generator $\xi$ in the homotopy fixed point spectral sequence computations which follow, as it, as well as the other generators in the presentation of Theorem~\ref{thm:HFPE2}, all lie in the first quadrant of the homotopy fixed point spectral sequence (with traditional Adams-style indexing).

As $\Delta$ is the product $\delta^2(b_4-11\delta)$, inverting $\Delta$ in Theorem~\ref{thm:HFPE2} is a rather opaque procedure.  Clearly it means that $\delta$ and $b_4-11\delta$ must be inverted.  Inverting $\delta$ is relatively straightforward: the entire cohomology is then $\delta$-periodic, and everything in $H^0$, as well as $\eta$ multiples on these classes, is a polynomial algebra\footnote{By this, we mean that in each bidegree, the resulting localized $E_2$-term takes the form $A \otimes \ZZ[\tilde{j}]$.} 
over
$$ \tilde{j} := b_4/\delta \in H^0(\pi_*\TMF_1(5)). $$
This class seems to act like a kind of $j$-invariant in the theory of modular forms for $\Gamma_0(5)$.  The relationship to the classical $j$-invariant is given by the following equation
$$ j = \frac{c_4^3}{\Delta} = \frac{(\tilde{j}^2 - 12\tilde{j} + 16)^3}{\tilde{j}-11}. $$
(We are grateful to the referee for suggesting the importance of this element.)  

However, inverting $b_4-11\delta$ (or equivalently $\tilde{j} -11$) is far more subtle, as there are many relations involving $b_4$ and hence $\tilde{j}$.  We propose two perspectives to help analyze the resulting localized cohomology groups.
\begin{description}
	\item[Perspective 1] Work $2$-locally.  The only torsion in $\TMF_0(5)$ is $2$-torsion, and arguably this spectrum is most interesting from the perspective of $2$-local homotopy theory.  We will argue that in this context, the effect of inverting $(\tilde{j}-11)$ can be analyzed with a simple set of relations.
	
	\item[Perspective 2] (We thank the referee for pointing out this alternative perspective.) Instead of focusing on $b_2$, make $b_4$ (or equivalently $\tilde{j} = b_4/\delta$) the more fundamental variable to express things in.  This perspective has the advantage of making $H^0$ a free module over the ring
	$$ \ZZ[1/5, \tilde{j}, (11 - \tilde{j})^{-1},\delta^{\pm 1}] $$
	at the expense of being able to easily identify $b_2$-periodic (i.e. $2$-primary $v_1$-periodic) classes.
\end{description}
Perspective~1 is arguably the better perspective to take if the reader is interested in $2$-local homotopy theory.  Perspective~2 is arguably more appropriate for those readers interested in $\TMF_0(5)$ from a global perspective (i.e. with only $5$ inverted).

\subsubsection*{Perspective 1: 2-local approach}

We offer the following simple corollary to Theorem~\ref{thm:HFPE2}, which is easily deduced from the relations therein.

\begin{cor}\label{cor:jrels}
In $H^*(\pi_*\TMF_1(5))$, we have
\begin{align*}
11(\tilde{j} -11)^{-1} b_4 & = b_2^2(\tilde{j}-11)^{-1} - 4\delta(\tilde{j}-11)^{-1}-b_4,
\\
(\tilde{j}-11)^{-1}\nu & = \nu,
\\
(\tilde{j}-11)^{-1}\gamma &= \gamma + (\tilde{j}-11)^{-1}(b_4b_2\delta^{-1}+b_2)\eta,
\\
(\tilde{j}-11)^{-1}\beta &= (\tilde{j}-11)^{-1}(\eta\gamma\delta^{-1}-b_2\eta^2\delta^{-1}) - \beta.
\end{align*}
\end{cor}

The appearance of the factor of $11$ in the first relation of the previous corollary complicates the situation, but this complication disappears after we invert $11$.  In particular we deduce from the above corollary that at least additively, $H^*(\pi_*\TMF_1(5)_{(2)})$ can be visualized from Figure~\ref{fig:HFPSSE2} by first inverting $\delta$, and then formally adjoining a polynomial algebra on $(\tilde{j}-11)^{-1}$ on all classes of the form
$$ \delta^ib_2^j\eta^k, \quad i \in \ZZ, j\ge 0, k \ge 0.$$

\begin{rmk}
If we complete at $(2,b_2)$ (as in the case of the $E_2$-term of the homotopy fixed point spectral sequence for $\TMF_0(5)_{K(2)}$), then the situation becomes more simple: the class $(\tilde{j}-11)^{-1}$ is already invertible in $H^*((T_*)^{\wedge}_{(2,b_2)})[\delta^{-1}]$.
\end{rmk}

\subsubsection*{Perspective 2: global approach}

The referee, in addition to suggesting the previous far more streamlined and readable approach to Theorem~\ref{thm:HFPE2}, found an alternative set of generators for $H^*(\pi_*\TMF_1(5))$ which gives a cleaner presentation if the reader does not wish to work $2$-locally.  Replace the generator $\gamma$ with the generator
$$ \tilde{\gamma} := \gamma + b_2 \eta. $$
We then have the following presentation of $H^*(\pi_*\TMF_1(5))$:
$$ H^0(\pi_*\TMF_1(5)) = \ZZ[1/5, \tilde{j}, (11 - \tilde{j})^{-1}, b_2, \delta^{\pm 1}]/(b_2^2 = (\tilde{j}^2+4)\delta) $$
and
$$ H^*(\pi_*\TMF_1(5)) = H^0(\pi_*\TMF_1(5))[\beta, \eta, \nu, \tilde{\gamma}]/\sim $$
where $\sim$ consists of relations
\begin{center}
\begin{tabular}{cc}
\begin{minipage}{0.4\linewidth}
\[\begin{aligned}
4\beta &= 0,\\
  2\tilde{j}\beta &= 0,\\
  2b_2\beta &= 0,\\
  2\eta &= 0,\\
  2\tilde{\gamma} &= 0,\\
  2\nu & = 0, \\
  b_2\tilde{\gamma} &= \delta\tilde{j}\eta,\\
  b_2 \eta &= \tilde{j}\tilde{\gamma},
\end{aligned}\]
\end{minipage}
&
\begin{minipage}{0.4\linewidth}
\[\begin{aligned}
  \tilde{j}\nu &= 0,\\
  b_2\nu &= 0,\\
  \eta^2 &= b_2\beta,\\
  \tilde{\gamma}^2 &= \delta b_2 \beta,\\
  \eta\tilde{\gamma} & = (\tilde{j}+2)\delta\beta,\\
  \nu^2 &= 2\delta\beta,\\
  \eta \nu &= 0, \\
 \nu\tilde{\gamma} & = 0. 
\end{aligned}\]
\end{minipage}
\end{tabular}
\end{center}

\subsection{The behavior of transfer and restriction in the homotopy fixed point spectral sequence}

Our next task is to compute the differentials in the homotopy fixed point spectral sequence
\begin{equation}\label{eq:HFPSS}
H^s(\FF_5^\times ; \pi_t \TMF_1(5)) \Rightarrow \pi_{t-s} \TMF_0(5).
\end{equation}
One might expect this could be accomplished by comparison with the well known descent spectral sequence for $\TMF$.  However, it will turn out that the images of many elements of $\pi_*\TMF$ in $\pi_*\TMF_0(5)$ will be detected on different lines of the respective spectral sequences.  An analysis of transfer and restriction maps relating these two spectral sequences will remedy this complication.

Let $\mM(5)$ denote the moduli space of elliptic curves with full level structure, and $\TMF(5)$ the corresponding spectrum of topological modular forms.  Utilizing the following
portion of \cite[Diagram 7.4.3]{KM}:
\begin{center}
\begin{tikzpicture}[descr/.style={fill=white,inner sep=.5pt}]
  \matrix(m) [matrix of math nodes, row sep=3em,
  column sep=3em]
  { \mM(5) \\
    \mM_0(5) \\
    \mM \\ };
  \path[-,font=\small]
  (m-1-1) edge node[auto] {$B$} (m-2-1)
  (m-2-1) edge (m-3-1)
  (m-1-1) edge [bend right=80] node[left] {$GL_2(\FF_5)$} (m-3-1);
\end{tikzpicture}
\end{center}
(where $B$ is the Borel subgroup of upper triangular matrices),  
the spectrum $\TMF(5)$ has an action of $GL_2(\FF_5)$, and we have
\begin{align*}
\TMF[1/5] & \simeq \TMF(5)^{hGL_2(\FF_5)}, \\
\TMF_0(5) & \simeq \TMF(5)^{hB}.
\end{align*}
We finally note that the moduli space $\mM(5)$ is representable by an affine scheme (see, for example, \cite{KM}).  It follows (see for example, \cite[Ch.~5]{TMF}) that the descent spectral sequences for $\TMF$ and $\TMF_0(5)$
\begin{gather*}
H^s(\mM; \omega^{\otimes t})[1/5] \Rightarrow \pi_{2t-s} \TMF[1/5] \\
H^s(\mM_0(5); \omega^{\otimes t}) \Rightarrow \pi_{2t-s} \TMF_0(5)
\end{gather*}
are isomorphic to the Cech descent spectral sequences associated to the \'etale affine covers 
\begin{gather*}
\mM(5) \rightarrow \mM \\
\mM(5) \rightarrow \mM_0(5),
\end{gather*}
respectively.
However, as these \'etale affine covers are in fact Galois, with Galois group $GL_2(\FF_5)$ and $B$, respectively, the Cech descent spectral sequences are precisely the homotopy fixed point spectral sequences:
\begin{gather*}
H^s(GL_2(\FF_5); \pi_{2t} \TMF(5)) \Rightarrow \pi_{2t-s} \TMF[1/5] \\
H^s(B; \pi_{2t} \TMF(5)) \Rightarrow \pi_{2t-s} \TMF_0(5).
\end{gather*}
Note that the we do not need to know anything about $\pi_*\TMF(5)$ to understand these spectral sequences; the $E_2$-terms are isomorphic to $H^*(\mM, \omega^{\otimes *})[1/5]$ and 
$H^*(\mM_0(5), \omega^{\otimes *})$, respectively.  

The descent spectral sequence for $\TMF$ is computed in many places.  For example, Bauer, in \cite{Tilman}, and the Hopkins-Mahowald article, in Part II of \cite{TMF}, compute the Adams-Novikov spectral sequence for $\tmf$.  It is explained in \cite{Konter} that the descent spectral sequence for $\TMF$ may be obtained from the Adams Novikov spectral sequence for $\tmf$ by inverting $\Delta$.  Alternatively, it is also explained in \cite{Konter} that the descent spectral sequence for $\TMF$ can be obtained from the descent spectral sequence for $\Tmf$ by inverting $\Delta$, and the descent spectral sequence for $\Tmf$ is described in \cite{Konter} and in \cite[Ch.~13]{TMF}.

The homotopy fixed point spectral sequence 
$$
H^s(B; \pi_{2t} \TMF(5)) \Rightarrow \pi_{2t-s} \TMF_0(5)
$$
is also isomorphic to the homotopy fixed point spectral sequence
$$
H^s(\FF_5^\times; \pi_{2t} \TMF_1(5)) \Rightarrow \pi_{2t-s} \TMF_0(5).
$$
Indeed, the latter is also a Cech descent spectral sequence, but for the affine \'etale Galois cover
$$ \mM_1(5) \rightarrow \mM_0(5). $$

\begin{lemma}\label{lem:trres}
The transfer-restriction composition
\[
  \pi_*\TMF[1/5] \xrightarrow{\mathrm{Res}} \pi_*\TMF_0(5)
  \xrightarrow{\mathrm{Tr}} \pi_*\TMF[1/5]
\]
is multiplication by $[GL_2(\FF_5):B] = 6$.  
\end{lemma}

\begin{proof}
The theorem is true on the level of homotopy fixed point spectral sequence $E_2$-terms: the composite
$$ H^s(GL_2(\FF_5); \pi_t \TMF(5)) \xrightarrow{\mathrm{Res}} H^s(B; \pi_t\TMF(5)) \xrightarrow{\mathrm{Tr}} H^s(GL_2(\FF_5); \pi_t \TMF(5)) $$
is multiplication by $[GL_2(\FF_5):B] = 6$.  Since there are no
nontrivial elements of $E^{s,t}_\infty$ with $t-s = 0$ and $s > 0$ (see, for example, \cite{Tilman}),  it follows that the transfer-restriction on the unit $1_{\TMF} \in \pi_0 \TMF[1/5]$ is given by
$$ \mathrm{Tr} \, \mathrm{Res} (1_{\TMF}) = 6 \cdot 1_{\TMF}. $$
We compute, using the projection formula, that for $a \in \pi_*\TMF[1/5]$, we have
$$ \mathrm{Tr} \, \mathrm{Res} ( a ) =  \mathrm{Tr} \, \mathrm{Res} (a \cdot 1_{\TMF}) = \mathrm{Tr} ((\mathrm{Res} \,  a) \cdot 1_{\TMF_0(5)}) = a \cdot \mathrm{Tr}(1_{\TMF_0(5)}) = 6\cdot a.$$
\end{proof}

We deduce the following corollary.

\begin{cor}\label{cor:trres}
Suppose that $ z\in \pi_*\TMF$ satisfies $2z\ne 0$, then $\mathrm{Res}(z)$ in
$\pi_*\TMF_0(5)$ is nonzero.   Morover, if  $z$ has Adams-Novikov filtration $s_1$, and $2z$ has Adams filtration $s_2$, then the Adams-Novikov filtration $s$ of $\mathrm{Res}(z)$ satisfies $s_1 \le s \le s_2$.
\end{cor}

Finally, in order to properly utilize the previous corollary, we record the behavior of the restriction.

\begin{lemma}\label{lem:restrictions}
The restriction
$$ H^s(GL_2(\FF_5); \pi_t \TMF(5)) \xrightarrow{\mathrm{Res}} H^s(B; \pi_t\TMF(5)) $$
has the following behavior on selected elements (with the notation of \cite{Tilman} being used for $H^s(GL_2(\FF_5); \pi_{2t} \TMF(5)) \cong H^s(\mM, \omega^{\otimes t})[1/5]$):
\begin{align*}
h_1 & \mapsto \eta, \\
h_2 & \mapsto \nu, \\
g & \mapsto \delta \xi^2 \mod {(2,b_2, \gamma\eta)}, \\ 
c_4 & \mapsto b_2^2-12b_4+12\delta, \\
c_6 & \mapsto -b_2^3+18b_2b_4-72b_2\delta, \\
\Delta & \mapsto \delta^2(b_4 - 11\delta).
\end{align*}
\end{lemma}

\begin{proof}
Consider the element $a_1$ of the elliptic curve Hopf algebroid.  It is primitive modulo $(2)$, and hence gives an element
$$ a_1 \in H^0(\mM_{\FF_2}, \omega) \cong H^0(GL_2(\FF_5); \pi_2 \TMF(5)/2). $$
However, in $\pi_2\TMF_1(5)$, we have $a_1= y - x$, and this gives rise in the proof of Proposition~\ref{prop:HFPE2} to an element
$$ v_1 \in H^0(\FF_5^\times; \pi_2\TMF_1(5)/2) \cong H^0(B; \pi_2 \TMF(5)/2).$$
We therefore have under the restriction map:
\begin{align*}
H^0(GL_2(\FF_5); \pi_2 \TMF(5)/2) & \rightarrow H^0(B; \pi_2 \TMF(5)/2) \\
a_1 & \mapsto v_1.
\end{align*}
Consider the diagram:
$$
\xymatrix{
H^0(GL_2(\FF_5); \pi_2 \TMF(5)/2) \ar[r]^\partial \ar[d]_{\mathrm{Res}} & 
H^1(GL_2(\FF_5); \pi_2 \TMF(5)) \ar[d]^{\mathrm{Res}}
\\
H^0(B; \pi_2 \TMF(5)/2) \ar[r]^\partial & 
H^1(B; \pi_2 \TMF(5)) 
}
$$
In the proof of Proposition~\ref{prop:HFPE2} we showed that $\partial(v_1) = \eta$, and the Bockstein spectral sequence computations of \cite{Tilman} give $\partial(a_1) = h_1$.  We deduce $\mathrm{Res}(h_1) = \eta$.

The restriction of $h_2$ must be non-trivial by Corollary~\ref{cor:trres}.  The element $\nu$ is the only nonzero element in the group $H^1(B; \pi_4 \TMF(5))$, so we must have $\mathrm{Res}(h_2) = \nu$.

The restriction of $g$ is computed by computing the restriction modulo $(2,a_1)$ (where $a_1 \in \pi_2 \TMF(5)$ is the image of $a_1 = y - x \in \pi_2 \TMF_1(5)$):
$$ \overline{\mathrm{Res}}: H^*(GL_2(\FF_5); \pi_* \TMF(5)/(2,a_1)) \rightarrow H^*(B; \pi_* \TMF(5)/(2,a_1)). $$
Since the mod 2 supersingular locus of $\mM_1(5)$ is given by 
$$ \mM_1(5)^{ss}_{\FF_2} = \Spec(\pi_0 (\TMF_1(5)/(2,a_1)), $$
the mod 2 supersingular locus of $\mM(5)$ is given by
$$ \mM(5)^{ss}_{\FF_2} = \Spec(\pi_0 (\TMF(5)/(2,a_1)). $$
As such, there are isomorphisms
\begin{align*}
H^s(GL_2(\FF_5); \pi_{2t} \TMF(5)/(2,a_1)) & \cong H^s(\mM^{ss}_{\FF_2},\omega^{\otimes t}) \cong H^s(G_{24}; \pi_{2t} E_2/(2,a_1))^{Gal}, \\
H^s(B; \pi_{2t} \TMF(5)/(2,a_1)) & \cong H^s(\mM_0(5)^{ss}_{\FF_2},\omega^{\otimes t})\cong H^s(C_4; \pi_{2t} E_2/(2,a_1))^{Gal}
\end{align*}
(where $G_{24}$ is the automorphism group of the unique supersingular curve over $\FF_4$ and $Gal = Gal(\FF_4/\FF_2)$). Under these isomorphisms the mod $(2,a_1)$ restriction map above is equivalent to the restriction map
$$ \overline{\mathrm{Res}} : H^*(G_{24}; \pi_*E_2/(2, u_1))^{Gal} \rightarrow H^*(C_4; \pi_* E_2/(2,u_1))^{Gal}. $$ 
Note that $\pi_{24} E_2/(2,u_1)$ is $\FF_4$, with trivial action by $G_{24}$.
We therefore have
\begin{align*}
H^{4}(G_{24}; \pi_{24}E_2/(2,u_1))^{Gal} & \cong H^4(Q_8; \FF_2) \\
& \cong \FF_2\{ g \}
\end{align*}
where $g$ is the image of the element 
$$ g \in H^4(\mM;\omega^{12}) \cong H^{4}(GL_2(\FF_5); \pi_{24}\TMF(5)) $$ 
of \cite{Tilman} 
under the reduction map
\begin{align*}
H^{4}(GL_2(\FF_5); \pi_{24}\TMF(5)) & \rightarrow H^{4}(GL_2(\FF_5); \pi_{24}\TMF(5)/(2,a_1)) \\
& \cong H^4(G_{24}; \pi_{24} E_2/(2,u_1))^{Gal}. 
\end{align*}
(This follows from the construction of $g$ in \cite{Tilman} using Bockstein spectral sequences.)
We also have
\begin{align*}
H^{4}(C_4; \pi_{24}E_2/(2,u_1))^{Gal} & \cong H^4(C_4; \FF_2) \\
& \cong \FF_2\{ \beta^2 \}
\end{align*}
and the restriction gives an isomorphism
$$ \overline{\mathrm{res}}: H^4(Q_8; \FF_2) \xrightarrow{\cong} H^4(C_4; \FF_2). $$
Now, consider the map
$$ \mathrm{red}: H^*(B; \pi_* \TMF(5) )/(2,b_2) \rightarrow H^*(B;\pi_*\TMF(5)/(2,a_1)). $$
Since $a_1 \equiv x+y$, and $b_4 \equiv xy(x^2+y^2)$, it follows that $\mathrm{red}(b_4) = 0$.  We therefore have (using Proposition~\ref{prop:HFPE2})
\begin{align*}
\mathrm{red}(\gamma\eta) 
& = \mathrm{red}(b_4 \gamma) \\
& = \mathrm{red}((b_4+b_2^2+2\delta)\beta) \\
& = 0.
\end{align*}
Therefore the map $\mathrm{red}$ descends to a map
$$ \overline{\mathrm{red}}: H^*(B; \pi_* \TMF(5))/(2, b_2, \gamma\eta) \rightarrow H^*(B; \pi_* \TMF(5)/(2,a_1)). $$
Now
$$ H^{4}(B; \pi_{24}\TMF(5))/(2, b_2, \gamma\eta) = \FF_2\{ \xi^2\delta \} $$
and $\overline{\mathrm{red}}(\xi^2 \delta)$ is the generator of $H^4(C_4;\FF_2)$.  We therefore have
\begin{align*}
\overline{\mathrm{red}} \mathrm{Res} (g) 
& =  \overline{\mathrm{Res}} (g) \\
& = \overline{\mathrm{red}}(\delta \xi^2), 
\end{align*}
and the result concerning the restriction of $g$ follows.

The restrictions of $c_4$, $c_6$, and $\Delta$ may be computed from the map of Hopf algebroids induced by the map $f$, computed in Theorem~\ref{thm:fqM1} (see Section~\ref{subsec:Q5}). 
\end{proof}

\begin{cor}
The elements $\eta$ and $\nu$ are permanent cycles in the homotopy fixed point spectral sequence for $\pi_* \TMF_0(5)$.
\end{cor}

\subsection{Computation of the differentials and hidden extensions}

The following sequence of propositions specifies the behavior of the
homotopy fixed point spectral sequence (\ref{eq:HFPSS}) 
culminating in Theorem \ref{thm:TMF05}, a complete description of
$\pi_*\TMF_0(5)$.

\begin{prop}\label{prop:d3}
In the homotopy fixed point spectral sequence (\ref{eq:HFPSS}),
$E_2 = E_3$ and the $d_3$-differentials are determined by
\begin{figure}
\includegraphics[width=\textwidth]{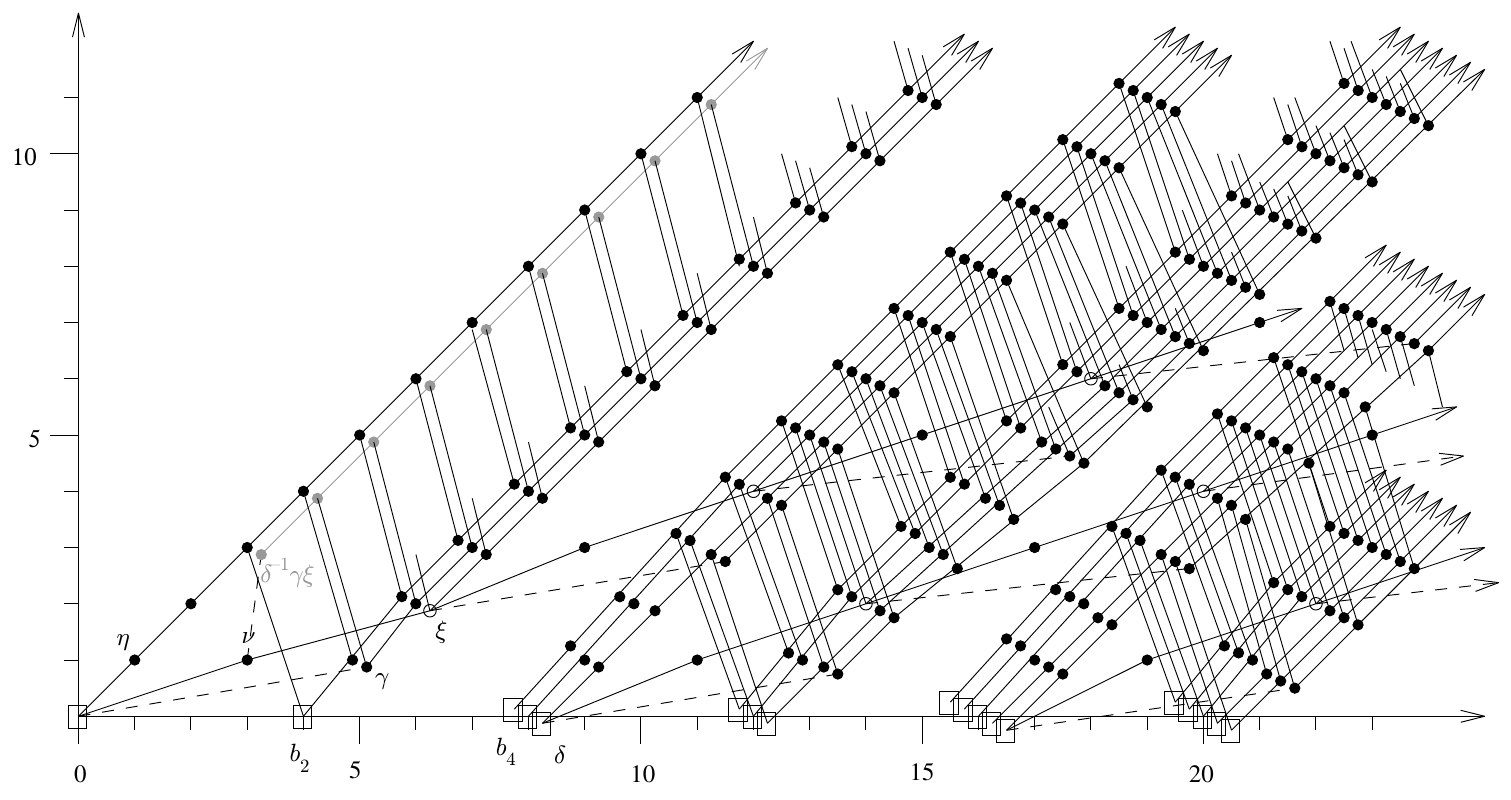}
\caption{The $d_3$-differentials in the homotopy fixed point spectral sequence for $\TMF_0(5)$.}\label{fig:HFPSSE2diffs}
\end{figure}
\[\begin{aligned}
  d_3 b_2 &= \eta^3,\\
  d_3\xi &= \delta^{-1}\eta\xi^2,\\
  d_3\gamma &= \delta^{-1}\eta\gamma\xi.
\end{aligned}\]
and $d_3(b_4) = d_3(\delta) = 0$.
\end{prop}

Figure~\ref{fig:HFPSSE2diffs} shows the $d_3$ differentials in the homotopy fixed point spectral sequence for $\TMF_0(5)$.  While most terms involving $\Delta^{-1}$ (and hence $\delta^{-1}$) are excluded, those depicted are shown in gray.

\begin{proof}
There is no room for $d_2$-differentials.

Note that $d_3 a_1^2 h_1 = h_1^4$ in the Adams-Novikov spectral
sequence for $\TMF$ (we use the notation of \cite{Tilman}).  Under the
restriction map $\TMF\to \TMF_0(5)$, this differential maps to $d_3
b_2\eta = \eta^4$, from which it follows that $d_3b_2 = \eta^3$, and therefore $d_3(b_2^2) = 0$.

Note that since the possible targets of $d_3(b_4)$ and $d_3(\delta)$ are $2$-torsion, we have $d_3(\delta^2) = d_3(b_4^2) = 0$.
The element $\Delta$ is a $d_3$-cycle in the Adams-Novikov spectral sequence for $\TMF$ \cite{Tilman}.  It follows that
$$ 0 = d_3(\delta^2(b_4 - 11\delta)) = \delta^2(d_3(b_4) + d_3(\delta)) $$
and therefore 
$$ d_3(b_4) = d_3(\delta). $$
However, we have
$$ 0 = d_3(b_4^2) = d_3(b_2^2 \delta - 4\delta^2) = b_2^2 d_3(\delta). $$
Since multiplication by $b_2^2$ is injective on the possible targets of $d_3(\delta)$, we conclude 
$$ d_3(b_4) = d_3(\delta) = 0. $$

By Corollary~\ref{cor:trres}, $2\nu$ must be detected in the 
homotopy fixed point spectral sequence for $\TMF_0(5)$ in Adams-Novikov filtration between 1 and 3.  Since $2\nu = 0$ in the $E_2$-page, it follows that in fact the filtration has to be between 2 and 3, and the only candidates live in filtration 3.  

We claim that the filtration $3$ class $\delta^{-1}\gamma\xi$ detects $2\nu$ in $\TMF_0(5)$.
To verify this claim, one can determine from Lemma~\ref{eq:addHFPE2} and Proposition~\ref{prop:HFPE2} that $E_2^{3,6}$ is an $\FF_2$-vector space.  One subtlety to determining this $\FF_2$-vector space is the fact that inverting $\Delta$ in $H^*(T_*)$ is equivalent to inverting $\delta$ and $b_4-11\delta$.  However, Corollary~\ref{cor:jrels}, and the discussion that follows, makes it clear that we have
$$ E_2^{3,6}/\eta^3 = \FF_2\{ \delta^{-1}\gamma\xi\}.  $$
Finally, as the $d_3$ differentials determined up to this point completely determine the differentials supported by the $0$-line, we can easily deduce that the image of $d_3$ in $E_2^{3,6}$ is precisely the image of $\eta^3$.  We therefore deduce that $\delta^{-1}\gamma\xi$ is the only potential candidate to detect $2\nu$ on the $E_3$-page of the spectral sequence.

Now observe that as a result of Corollary~\ref{cor:jrels}, and the discussion which follows, we have
$$
d_3\gamma = a\delta^{-1}\eta\gamma\xi + \sum_{k,l \ge 0} a'_{k,l} \tilde{j}^k(\tilde{j}-11)^{-l}\eta^4 + \sum_{m \ge 0} a_m'' \tilde{j}^m \delta^{-1} b_4 \eta^4  
$$
for coefficients $a, a_{k,l}', a_m'' \in \ZZ/2$ with all but finitely many equal to zero.  
The class representing $2\eta\nu$,
i.e. $\delta^{-1}\eta\gamma\xi$, must die in the spectral sequence.  Since we have already established all of the terms involving $\eta^4$ are the targets of established $d_3$-differentials, this is only
possible if $a = 1$.

We therefore have, using $b_2 \xi = \delta \eta^2$:
\begin{align*}
d_3 (b_2 \gamma) & = d_3(b_2)\gamma + b_2 d_3(\gamma) \\
& = \sum_{k,l \ge 0} a'_{k,l} \tilde{j}^k(\tilde{j}-11)^{-l}b_2 \eta^4 + \sum_{m \ge 0} a_m'' \tilde{j}^m b_2 \delta^{-1} b_4 \eta^4.
\end{align*}
Turning this around, we have
\begin{align*}
\sum_{k,l \ge 0} a'_{k,l} \tilde{j}^k(\tilde{j}-11)^{-l}b_2 \eta^4 + \sum_{m \ge 0} a_m'' \tilde{j}^m b_2 \delta^{-1} b_4 \eta^4 & = d_3(b_2\gamma) \\
& = d_3(\eta(b_2^2+b_4)) \\
& = 0.
\end{align*}
We deduce that the coefficients $a_{k,l}'$ and $a_m''$ are all zero.

Since $\delta^{-1}\gamma\xi$ is a permanent cycle, we have
\[
  0 = d_3\delta^{-1}\gamma\xi = (d_3\delta^{-1}\xi)\gamma -
  \delta^{-1}\xi(d_3\gamma).
\]
Hence $d_3\xi = \delta^{-1}\eta\xi^2$.
\end{proof}

\begin{figure}
\includegraphics[width=\textwidth]{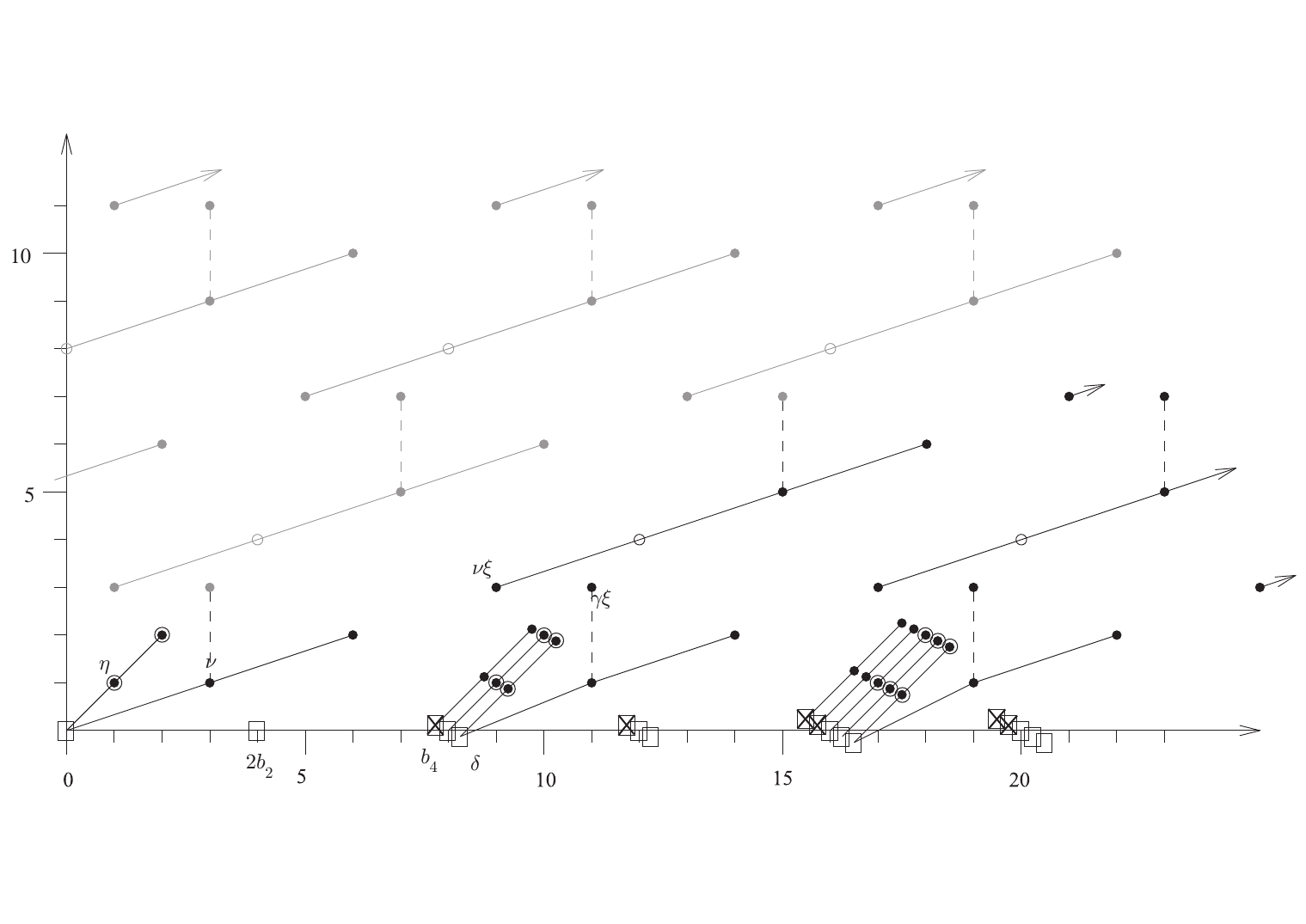}
\caption{The $E_4=E_5$ term in the homotopy fixed point spectral sequence for $\TMF_0(5)_{(2)}$.}\label{fig:HFPSSE4}
\end{figure}

\begin{cor}
The $E_4$ term of the homotopy fixed point spectral sequence is described below.
$$ E_4 = \ZZ[1/5][2b_2, b_2^2, b_4, \delta, \eta, \nu, \xi^2, \nu\xi, \gamma\xi, \delta^{-1}, (\tilde{j}-11)^{-1}]/\sim $$
where $\sim$ consists of relations

\begin{center}
\begin{tabular}{cc}
\begin{minipage}{0.4\linewidth}
\[\begin{aligned}
  b_4^2 &= b_2^2 \delta-4\delta^2,\\
  2\eta &= 0,\\
  2\nu &= 0,\\
  2\gamma\xi &= 0,\\
  4\xi^2 &= 0,\\
  \eta^3 & = 0, \\
  \nu^3 &= 0,\\
  (\gamma\xi)^2 &= 0, \\
  \eta\nu &= 0,\\
  \eta \xi^2 & = 0, \\
  \eta \gamma \xi & = 0, 
\end{aligned}\]
\end{minipage}
&
\begin{minipage}{0.4\linewidth}
\[\begin{aligned}
  b_2^2 \nu & = 0, \\
 b_4 \nu & = 0, \\
\nu (\nu\xi) & = 2\xi^2, \\
\nu (\gamma \xi) & = 0, \\
  2b_2\xi^2 &= 0,\\
b_2^2 \xi^2 & = 0, \\
  b_4\xi^2 &= 2\delta\xi^2,\\
(\nu\xi) (\gamma \xi) & = 0, \\
  b_4(\gamma \xi) &= \delta\eta^2(b_4+\delta), \\
 b_2^2 (\gamma\xi) & = 0. \\
\end{aligned}\]
\end{minipage}
\end{tabular}
\end{center}
Here we have omitted relations like $(2b_2)^2 = 4b_2^2$, $(2b_2)\nu = 0$ and $2(\nu\xi) = 0$, as they follow `from the notation'.  Everything is $\delta$-periodic, and multiplication by $(\tilde{j}-1)^{-1} = (\delta^{-1}b_4-11)^{-1}$ satisfies the following relations (which follow from those above): 
\begin{align*}
11(\tilde{j} -11)^{-1} b_4 & = b_2^2(\tilde{j}-11)^{-1} - 4\delta(\tilde{j}-11)^{-1}-b_4,
\\
(\tilde{j}-11)^{-1}\nu & = \nu,
\\
(\tilde{j}-11)^{-1}\xi^2 & = -\xi^2,
\\
(\tilde{j}-11)^{-1}\nu\xi &= \nu\xi, 
\\
(\tilde{j}-11)^{-1}\gamma\xi & = \gamma\xi.
\end{align*}
\end{cor}

Figure~\ref{fig:HFPSSE4} shows the resulting $E_4$-term in the homotopy fixed point spectral sequence for $\TMF_0(5)_{(2)}$.  The authors find this easier to visualize $(2)$-locally (i.e. ``Perspective 1'' of Section~\ref{sec:E2comp}).  Terms involving $\delta^{-1}$ are excluded on the $0$, $1$ and $2$-lines, and in lines greater than $2$ are shown in gray.  As in the other charts in this paper, solid dots denote $\ZZ/2$'s, and open circles denote $\ZZ/4$'s. If we localize at $(2)$, the other symbols in the figure denote the following:
\begin{align*}
\square & = \ZZ_{(2)}[(\tilde{j}-11)^{-1}], \\
\boxtimes & = \ZZ_{(2)}, \\
\text{\textcircled{$\bullet$}} & = \ZZ/2[(\tilde{j}-11)^{-1}].
\end{align*}

In the following sequence of propositions, we will establish the rest of the differentials in the homotopy fixed point spectral sequence. 
Figure~\ref{fig:HFPSSE8} displays these differentials.  In this figure, the gray patterns represent the (infinite rank) $bo$-patterns.

%



We will need to observe the following to compute our $d_5$-differentials.

\begin{lemma}
On the level of $E_5$-terms the restriction map (from the homotopy fixed point spectral sequence for $\TMF$ to the homotopy fixed point sequence for $\TMF_0(5)$) sends $\bar{\kappa}$ to $\delta \xi^2$.
\end{lemma}

\begin{proof}
In the homotopy fixed point spectral sequence for $\TMF$, the element $\bar{\kappa}$ is detected by $g$.
By Lemma~\ref{lem:restrictions}, we have
\begin{align*}
 \mathrm{Res}(g) & = \delta \xi^2 \mod (2, b_2, \gamma\eta).
\end{align*}
The lemma follws, as the elements of $H^4(\FF_5^\times; \pi_{24}\TMF_1(5))$ which are divisible by $2$, $b_2$, or $\gamma\eta$ are all killed by $d_3$-differentials.
\end{proof}

\begin{cor}
The element $\delta \xi^2$ is a permanent cycle in the homotopy fixed point spectral sequence for $\TMF_0(5)$.
\end{cor}

\begin{prop}\label{prop:d5}
In the homotopy fixed point spectral sequence for $\pi_*\TMF_0(5)$,
$E_4=E_5$ and the $d_5$-differentials are determined by
\begin{align*}
  d_5(2b_2) & = d_5(b_2^2) = d_5(b_4) = 0, \\
  d_5(\delta) & = \delta^{-1}\nu\xi^2, \\
  d_5(\eta) & = d_5(\nu) = d_5(\gamma \xi) = 0, \\
  d_5(\xi^2) & = \delta^{-2} \nu \xi^4, \\
  d_5(\nu \xi) & = 2\delta^{-2} \xi^4.
\end{align*}
\end{prop}

\begin{proof}[Proof of Proposition~\ref{prop:d5}, part 1]
There is no room for $d_4$-differentials.  We have already observed that $\eta$ and $\nu$ are permanent cycles.  Dimensional considerations also immediately show
\begin{align*}
d_5(2b_2) & = d_5(\gamma\xi) = 0.
\end{align*}
Note that the only possible target for a $d_5$-differential on $b_2^2$ or $b_4$ is $\nu\delta^{-1}\xi^2$.  Since $\nu^2\delta^{-1}\xi^2$ is non-trivial in $E_5$, such non-trivial differentials would only be possible if $\nu b_4$ or $\nu b_2^2$ were non-trivial, but this is not the case.  We deduce that 
$$ d_5(b_4) = d_5(b_2^2) = 0. $$ 
The element $\overline{\kappa}\in \pi_{20}S$ is in the Hurewicz image
of $\TMF$.  In the Adams-Novikov
spectral sequence for $\TMF$, $d_5 \Delta = \nu \overline{\kappa}$.  
We deduce that
\begin{align*}
\nu \delta \xi^2 & = d_5(\delta^2(b_4-11\delta)) \\
& = 2\delta d_5(\delta) (b_4-11\delta) + \delta^2 d_5(b_4) - 11\delta^2d_5(\delta) \\
& = 2\delta b_4 d_5(\delta) - 33 \delta^2 d_5(\delta). 
\end{align*}
Since the only available class for $d_5(\delta)$ to hit is $2$-torsion in the $E_5$-page, we deduce that
$$ \delta^2 d_5(\delta) = \nu \delta \xi^2. $$
We have already observed that $\delta \xi^2$ is a permanent cycle since it detects $\bar\kappa$.  We may therefore compute
\begin{align*}
0 & = d_5(\delta \xi^2), \\
& = d_5(\delta)\xi^2 + \delta d_5(\xi^2), \\
& = \delta^{-1}\nu\xi^4 + \delta d_5(\xi^2).
\end{align*}
We deduce that 
$$ d_5(\xi^2) = \delta^{-2}\nu \xi^4.  $$
The only class left to handle is $\nu\xi$.  We will defer the proof of this differential until after we establish the $d_7$-differentials.
\end{proof}

\begin{figure}
\includegraphics[width=\textwidth]{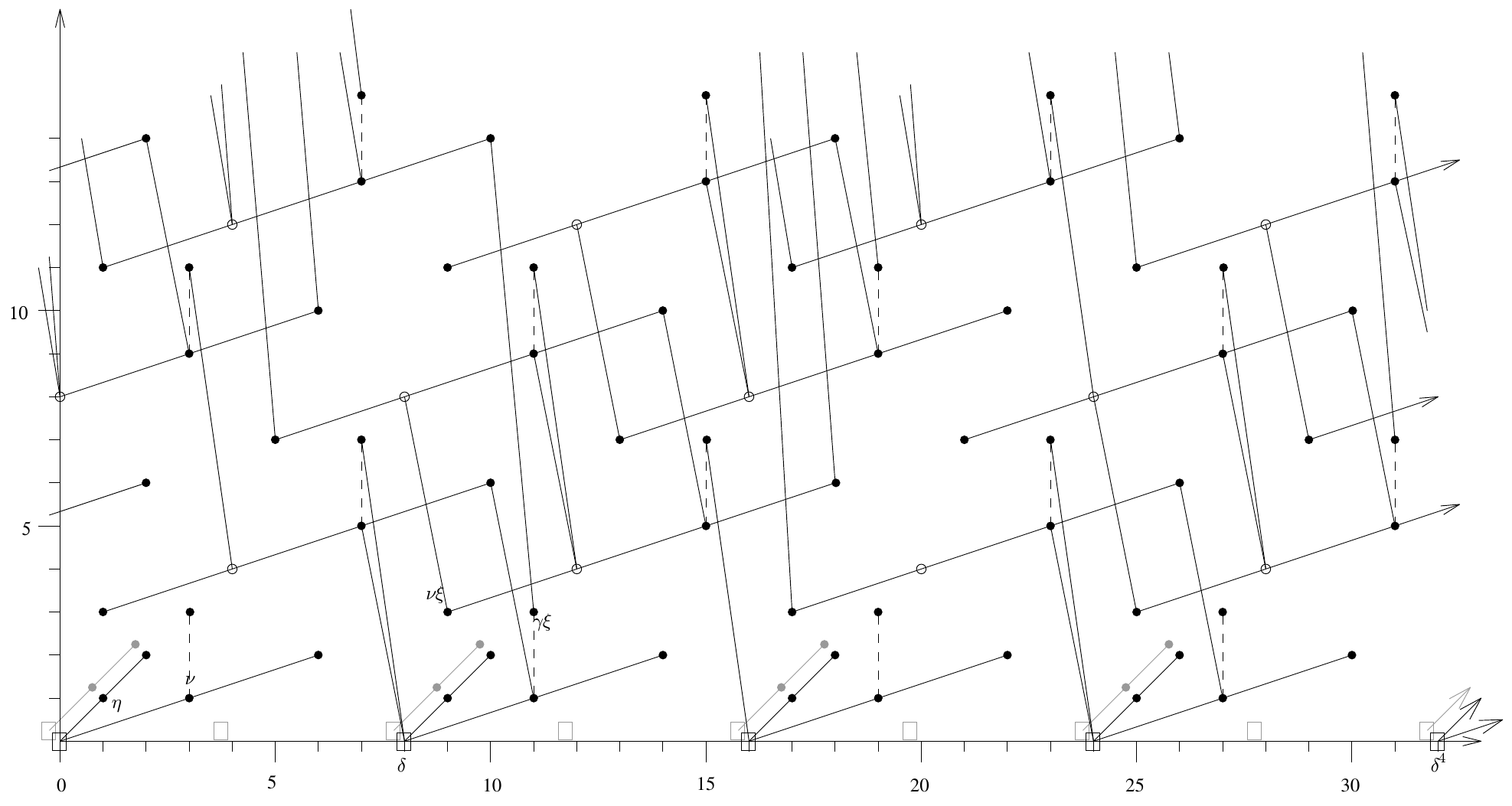}
\caption{The $E_4$ term in the homotopy fixed point spectral sequence for $\TMF_0(5)$ with $d_r$-differentials, $r \ge 4$.}\label{fig:HFPSSE8}
\end{figure}

\begin{prop}\label{prop:d7}
In the homotopy fixed point spectral sequence for $\pi_*\TMF_0(5)$,
$E_6 = E_7$ and the $d_7$-differentials are determined by
\[\begin{aligned}
d_7(2b_2) = d_7(b_2^2) = d_7(\eta) = d_7(\nu) & = d_7(\delta \xi^2) = d_7(\delta\nu\xi) = d_7(\gamma\xi) = d_7(\delta\gamma\xi)  = 0, \\
  d_7 (2\delta) = d_7(b_4) & = \delta^{-2}\gamma\xi^3,\\
  d_7(\delta^2) & = \delta^{-1}\gamma\xi^3, \\
  d_7(\delta b_4) & = 0. \\
\end{aligned}\]
\end{prop}
\begin{proof}
There is no room for $d_6$-differentials.
We have already observed that $\eta$, $\nu$, and $\delta\xi^2$ are permanent cycles, since they are in the Hurewicz image.  
The elements $2b_2$, $b_2^2$, $\delta\nu\xi$, $\gamma\xi$, and $\delta\gamma\xi$ are $d_7$-cycles for dimensional reasons.

In order to establish the next round of differentials, we will first determine $d_7(2\delta^3)$ and $d_7(\delta^2 b_4)$ (of course, these differentials are determined by $d_7(2\delta)$, $d_7(b_4)$, and $d_7(\delta^2)$).  Note that $2\nu\overline{\kappa}$ is $0$
in $\pi_*\TMF$, from which we deduce that the class represented by $\delta^{-1}\gamma\xi\overline{\kappa}$ is $0$
in $\pi_*\TMF_0(5)$ via the restriction map.  
The element $\gamma\xi^3$
detects this class, so it must be the target of a differential, and the only (not necessarily exclusive) possibilities at this point are:
\begin{align*}
\text{Case 1:} & \quad d_7(2\delta^3) = \gamma \xi^3, \\
\text{Case 2:} & \quad d_7(\delta^{2-i} b_4 b^{2i}_2) = \gamma \xi^3, \quad \text{for some} \: i \ge 0, \\
\text{Case 3:} & \quad d_7(\delta^{2-i} b^{2i+2}_2) = \gamma \xi^3, \quad \text{for some} \: i \ge 0.
\end{align*}
Multiplying by the permanent cycle $\mathrm{Res}(\bar{\kappa}) = \delta \xi^2$, Cases $2$ yields
$$ d_7(\delta^{5-i} \eta^4 b^{2i}_2 + 2\delta^{4 - i}\xi^2b_2^{2i} + \delta^{4-i} \xi \gamma \eta) \ne 0. $$
If $i > 0$, this is a contradiction because 
$$ \delta^{5-i} \eta^4 b^{2i}_2 = 2\delta^{4 - i}\xi^2b_2^{2i} = \delta^{4-i} \xi \gamma \eta = 0 $$
in the $E_7$-page for $i > 0$.  
Therefore Case $2$ for $i > 0$ cannot occur.  Similarly, multiplying Case $3$ by $\bar{\kappa}$ gives
$$ d_7(\delta^{5-i} b_2^{2i} \eta^4) \ne 0, $$
again a contradiction.  We conclude that either Case 1 or Case 2 with $i = 0$ must hold.  
Therefore
\begin{align*}
d_7(2\delta^3) & = a\gamma \xi^3, \\
d_7(\delta^2 b_4) & = b \gamma \xi^3
\end{align*}
with $a = 1$ or $b = 1$.  Multiplying both of the above differentials by $\bar{\kappa}$ yields:
\begin{align*}
d_7(2\delta^4 \xi^2) & = a\delta \gamma \xi^5, \\
d_7(2\delta^4 \xi^2) & = b \delta \gamma \xi^5.
\end{align*}
We deduce that $a = b = 1$.  Hence we deduce that
\begin{align*}
d_7(2\delta^3) & = \gamma \xi^3, \\
d_7(\delta^2 b_4) & = \gamma \xi^3.
\end{align*}
We now turn our attention to $d_7(2\delta)$, $d_7(b_4)$, and $d_7(\delta^2)$.  The only possible targets for these differentials are $\delta^{-2}\gamma \xi^3$ and $\delta^{-1} \gamma \xi^3$, respectively.  Write
\begin{align*}
d_7(2\delta) & =  c\delta^{-2}\gamma \xi^3, \\
d_7(b_4) & = d\delta^{-2}\gamma \xi^3, \\
d_7(\delta^2) & = e\delta^{-1}\gamma \xi^3.
\end{align*}
Then we have
\begin{align*}
\gamma \xi^3 & = d_7 (\delta^2 b_4) \\
& = d_7(\delta^2)b_4 + \delta^2 d_7(b_4) \\
& = e\delta^{-2}\gamma \xi^3 b_4 + d\gamma \xi^3.
\end{align*}
Using the relations we find that $\delta^{-2}\gamma \xi^3 b_4 = 0$, and we therefore deduce that $d = 1$.
Similarly, we have
\begin{align*}
\gamma \xi^3 & = d_7 (2\delta^3) \\
& = d_7(\delta^2)2\delta + \delta^2 d_7(2\delta) \\
& = 2e\delta^{-1}\gamma \xi^3 + c\gamma \xi^3.
\end{align*}
Since $2\delta^{-1}\gamma \xi^3 = 0$, we deduce that $c = 1$.
We have shown
\begin{align*}
d_7(2\delta) & = \delta^{-2}\gamma \xi^3, \\
d_7(b_4) & = \delta^{-2} \gamma \xi^3.
\end{align*}
To establish the final $d_7$ differential on $\delta^2$, note that the restriction map $\TMF \to \TMF_0(5)$ takes 
$2\nu\Delta$ to
$2\nu\Delta$ which is nonzero in $\pi_*\TMF_0(5)$.  Since
$2\nu\Delta\overline{\kappa} = 0 \in \pi_*\TMF$, we know
$2\nu\Delta\overline{\kappa} = 0 \in \pi_*\TMF_0(5)$.  The element
$\gamma\xi^3\delta^3$ detects this class.  It
follows that $\delta^{-1} \gamma \xi^3$ must be the target of a differential.
By the same argument used earlier, multiplication by $\bar{\kappa}$ shows that the only possible sources of a differential killing $\delta^{-1}\gamma \xi^3$ are $\delta^2$ and $\delta b_4$.
Write
\begin{align*}
d_7(\delta^2) & = e\delta^{-1}\gamma \xi^3, \\
d_7(\delta b_4) & = f\delta^{-1} \gamma \xi^3.
\end{align*}
so that $e$ or $f$ equals $1$ mod $2$.  Multiplying both of these differentials by $\bar\kappa$ yields
\begin{align*}
d_7(\delta^3 \xi^2) & = e\gamma \xi^5, \\
d_7(2\delta^3 \xi^2) & = f\gamma \xi^5.
\end{align*}
Thus we have $e \equiv 1$ mod $2$, and $f \equiv 0$ mod $2$, and 
\begin{align*}
d_7(\delta^2) & = \delta^{-1}\gamma \xi^3, \\
d_7(\delta b_4) & = 0.
\end{align*}
\end{proof}

\begin{proof}[Proof of Proposition~\ref{prop:d5}, part 2]
We now return to the proof of Proposition~\ref{prop:d5} to establish the one remaining differential, $d_5(\nu\xi)$.  We note that
$$ d_5(\delta \nu \xi) = 0 $$
since the only possible non-trivial target of such a differential would be $2\xi^2$, and this supports a non-trivial $d_7$-differential by Proposition~\ref{prop:d7}.  We therefore have
\begin{align*}
0 & = d_5(\delta \nu \xi) \\
& = \delta^{-1}\nu^2\xi^3 + \delta d_5(\nu \xi) \\
& = \delta^{-1}2\xi^4 + \delta d_5(\nu \xi).
\end{align*}
We conclude that we have
$$ d_5(\nu \xi) = 2\delta^{-2}\xi^4. $$
\end{proof}

To handle the next round of differentials we will need the following lemma.

\begin{lemma}
The Hurewicz image of the element $\kappa$ in $\pi_{14}\TMF$ is restricts to a non-trivial class in $\pi_{14}\TMF_0(5)$, detected by $\nu^2 \delta$ in the homotopy fixed point spectral sequence.
\end{lemma}

\begin{proof}
Applying Corollary~\ref{cor:trres} to the class $\Delta^4 \kappa \in \pi_{110} \TMF$ of order $4$, we find that $\mathrm{Res}(\Delta^4 \kappa)$ is non-trivial, and detected in the homotopy fixed point spectral sequence by a class in filtration between $4$ and $14$.  Given our $d_5$-differentials, the only candidate is $\nu^2 \delta^{13}$.
Since $E_2$ is $\delta$-periodic, and since $\kappa$ is detected in filtration $2$ in $\TMF$, it follows that on the level of $E_2$ pages $\kappa$ restricts to $\nu^2 \delta$.  The lemma follows, since $\nu^2 \delta$ is not the target of a differential.
\end{proof}

\begin{prop}\label{prop:d11}
In the homotopy fixed point spectral sequence for $\pi_*\TMF_0(5)$,
$E_8 = E_9 = E_{10}$ and the $d_{11}$-differentials are determined by
\[
  d_{11} (\gamma\xi) = \delta^{-4}\xi^7.
\]
\end{prop}

\begin{proof}
In $\pi_*\TMF$ we have $\bar{\kappa}^3 \kappa = 0$.  The restriction of this element in $\TMF_0(5)$ is detected  in the homotopy fixed point spectral sequence by $\delta^4 \xi^7$, so the latter must be the target of a differential.  The only possibility is $d_{11} (\delta^8 \gamma \xi) = \delta^4 \xi^7$.  Since $\delta^{4}$ persists to the $E_{11}$-page, and there are no non-trivial targets for $d_{11}(\delta^4)$, it follows that $E_{11}$ is $\delta^4$-periodic, and the proposition follows.
\end{proof}

\begin{prop}\label{prop:d13}
In the homotopy fixed point spectral sequence for $\pi_*\TMF_0(5)$,
$E_{12} = E_{13}$ and the $d_{13}$-differentials are determined
by
\[\begin{aligned}
  d_{13}(\delta\nu\xi) &= \delta^{-4}\xi^8,\\
  d_{13}(\delta^{3}\nu^2) &= \delta^{-2}\nu\xi^7.
\end{aligned}\]
\end{prop}

\begin{proof}
In $\pi_* \TMF$ we have $\bar{\kappa}^6 = 0$.  Since $\mathrm{Res}(\bar{\kappa}^6)$ is detected by $\delta^6 \xi^{12}$ in the homotopy fixed point spectral sequence for $\TMF_0(5)$, the latter must be the target of a differentials.  Since $\bar{\kappa}\delta^6 \xi^{12}$  is non-trivial in $E_{13}$, if $d_r(x) = \delta^6 \xi^{12}$ it must be the case that $\bar{\kappa} \cdot x \ne 0$.  The only such candidate is
$$ d_{13}(\delta^{11} \nu \xi^5) =  \delta^6 \xi^{12}. $$
Dividing by $\bar{\kappa}^2$, it follows that we have
$$ d_{13}(\delta^{9} \nu \xi) =  \delta^4 \xi^{8}. $$
Since $\delta^4$ persists to $E_{13}$ with no possible targets for a non-trivial $d_{13}(\delta^4)$, it follows that
$$ d_{13}(\delta \nu \xi) =  \delta^{-4} \xi^{8}. $$

The differential $d_{13}\delta^{3}\nu^2 = \delta^{-2}\nu\xi^7$ actually follows from the differential above, though perhaps not so obviously, so we will explain in more detail.  The element $\xi^3 \nu$ persists to the $E_{13}$-page, and there are no possibilities for it supporting a non-trivial $d_{13}$-differential.  However, by the previous paragraph,
$$ \bar{\kappa}^4 \xi^3 \nu = \delta^4 \xi^{11} \nu = d_{13}(\delta^9 \xi^4 \nu^2) \ne 0 \in E_{13}. $$
Dividing by $\bar{\kappa}^2$, we get
$$ d_{13}(\delta^7 \nu^2) = \delta^2 \xi^{7} \nu $$
and thus
$$ d_{13}(\delta^3 \nu^2) = \delta^{-2} \xi^{7} \nu. $$
\end{proof}

This concludes the determination of the differentials in the homotopy fixed point spectral sequence, there are no further possibilities.  We now turn to determining the hidden extensions in this spectral sequence.  To do this, we will recompute $\pi_* \TMF_0(5)$ using a homotopy orbit spectral sequence.  This different presentation will turn out to elucidate the multiplicative structure missed by the homotopy fixed point spectral sequence.

\begin{figure}
\includegraphics[width=\textwidth]{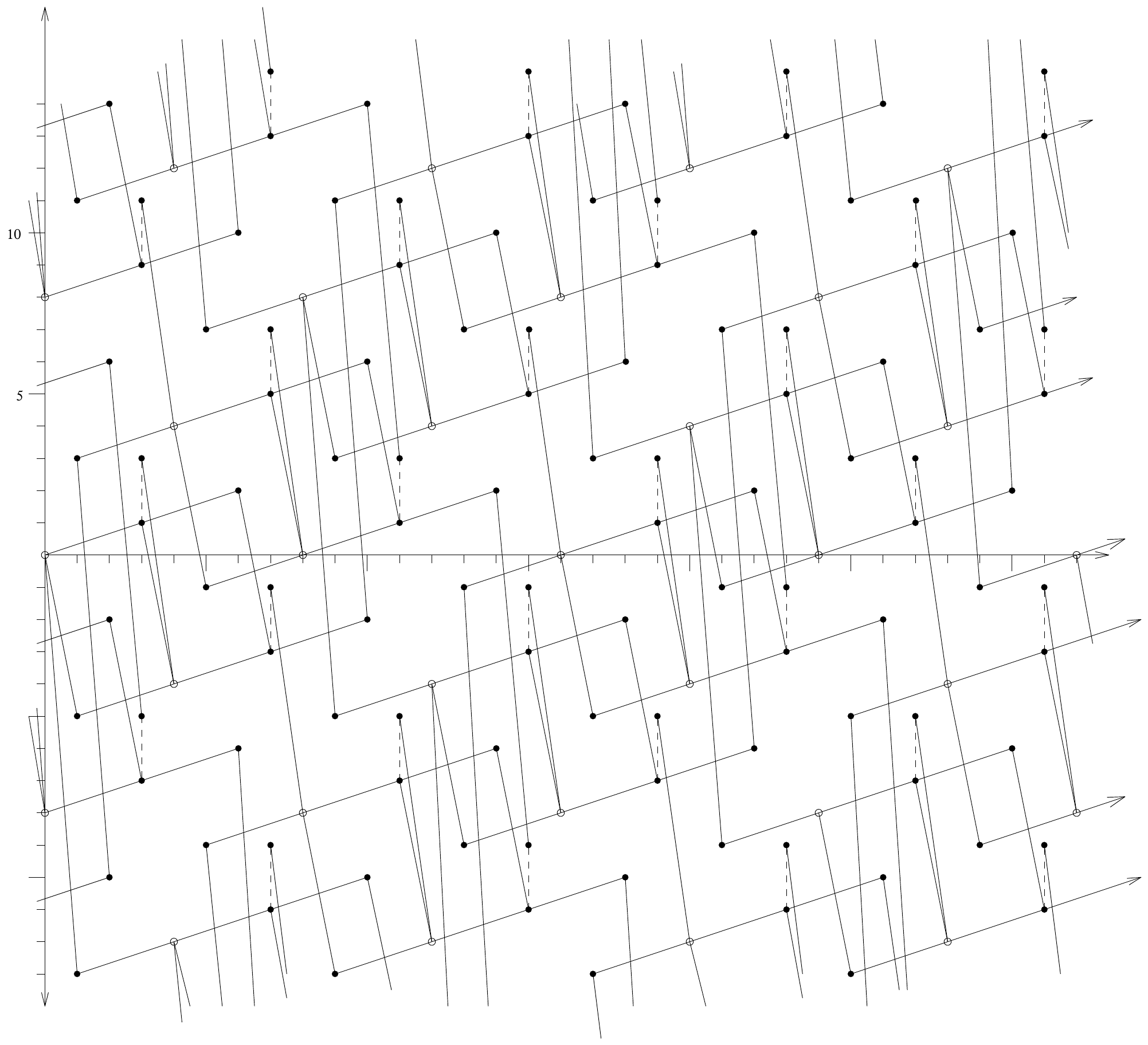}
\caption{The $E_4$ term in the Tate spectral sequence for $\TMF_1(5)^{t\FF_5^\times}$ with $d_r$-differentials, $r \ge 4$.}\label{fig:tateSSE8}
\end{figure}

\begin{figure}
\includegraphics[width=\textwidth]{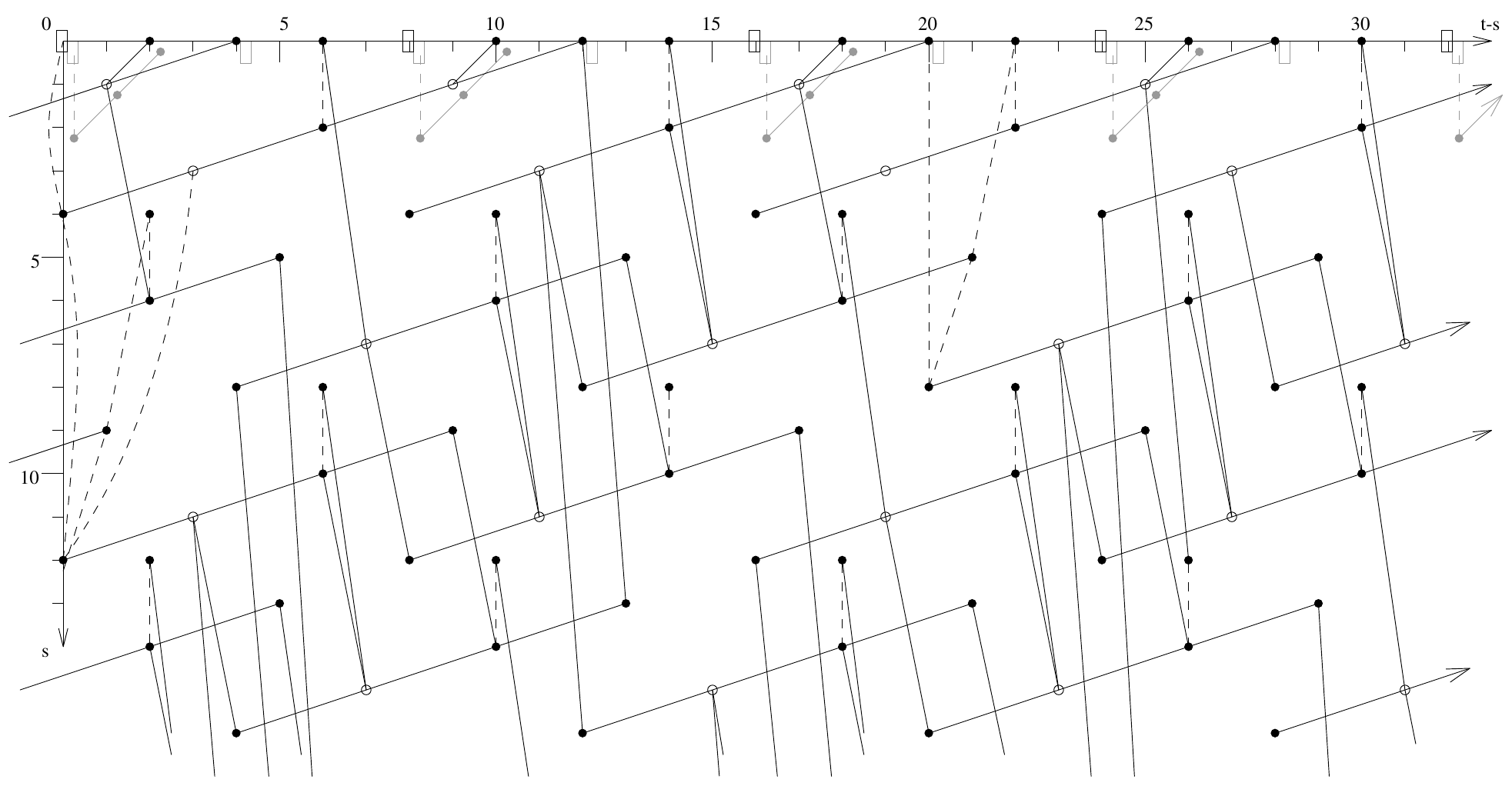}
\caption{The $E_4$ term in the homotopy orbit spectral sequence for $\TMF_1(5)_{h\FF_5^\times}$ with $d_r$-differentials, $r \ge 4$.}\label{fig:HOSSE8}
\end{figure}

The Tate spectral sequence
$$ \widehat{H}^s(\FF_5^\times ; \pi_t \TMF_1(5)) \Rightarrow \pi_{t-s} \TMF_1(5)^{t\FF^\times_5} $$
can be easily computed from the homotopy fixed point spectral sequence --- one simply has to invert $\xi$.
A picture of the resulting spectral sequence (just from $E_4$ and beyond) is displayed in Figure~\ref{fig:tateSSE8}.

Note that everything dies in this spectral sequence.  Therefore, we have established the following lemma.  (There may be other more conceptual ways of proving the following Lemma --- for instance, it is well known to hold $K(2)$-locally, and the unlocalized statement might follow from the fact that $\mathcal{M}_1(5) \rightarrow \mathcal{M}_0(5)$ is a Galois cover).

\begin{lemma}
The Tate spectrum $\TMF_1(5)^{t\FF_5^\times}$ is trivial, and therefore the norm map
$$ N: \TMF_1(5)_{h\FF_5^\times} \rightarrow \TMF_1(5)^{h\FF_5^\times} $$
is an equivalence.
\end{lemma}

Thus the homotopy groups of $\pi_*\TMF_1(5)_{h\FF_5^\times} = \pi_*\TMF_0(5)$ are isomorphic to $\pi_*\TMF_1(5)^{h\FF_5^\times}$ as modules over $\pi_* \TMF$.  However, these $\pi_*\TMF$-modules are computed in an entirely different way by the homotopy orbit spectral sequence
$$ H_s(\FF_5^\times ; \pi_t \TMF_1(5)) \Rightarrow \pi_{s+t} \TMF_0(5). $$
Nevertheless, the homotopy orbit spectral sequence (with differentials) can be computed  by simply truncating the Tate spectral sequence (and manually computing $H_0$ where appropriate).  The resulting homotopy orbit spectral sequence is displayed in Figure~\ref{fig:HOSSE8}.  As with our other spectral sequences, we are displaying the $E_4$-page, with all remaining differentials.  The (infinite rank) $bo$ patterns are displayed in gray.

\begin{figure}
\includegraphics[width=\textwidth]{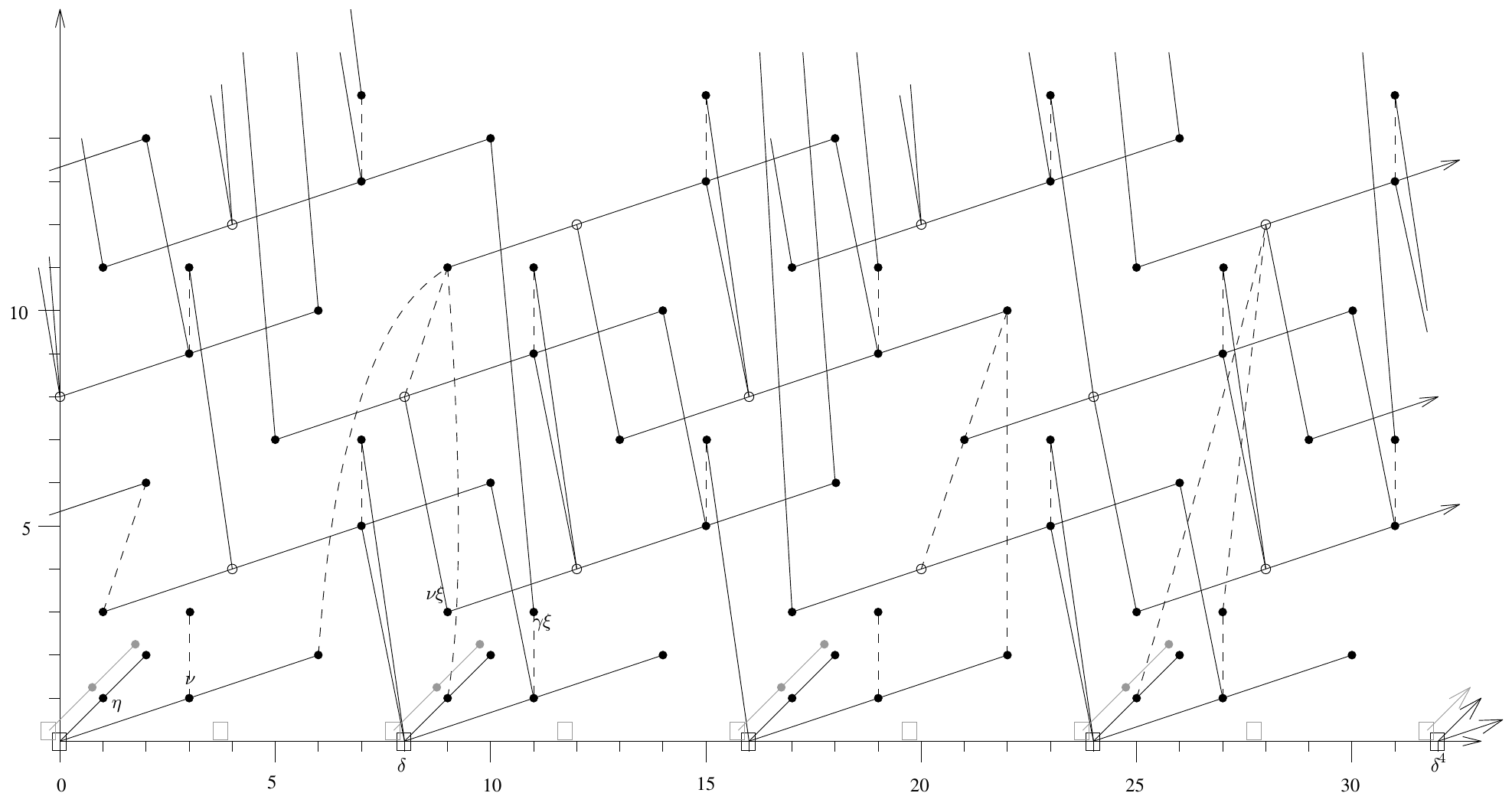}
\caption{The hidden extensions in the homotopy fixed point spectral sequence for $\TMF_0(5)$.}\label{fig:HFPSSEinfty}
\end{figure}

There are many hidden extensions (as $\pi_*\TMF$ modules) in the homotopy orbit spectral sequence (HOSS) which are not hidden in the homotopy fixed point spectral sequence (HFPSS).  Since $\pi_0 \TMF_0(5)$ is seen to be torsion free in the HFPSS, there must be additive extensions as indicated in Figure~\ref{fig:HOSSE8}, and $1 \in \pi_0 \TMF_0(5)$ must be detected on the $s = 12$ line.  Since the HFPSS shows $\eta$, $\eta^2$, and $\nu$ are nontrivial in $\pi_*\TMF_0(5)$, there must be corresponding hidden extensions in the HOSS.  Multiplying these by $\bar{\kappa}$ in the HOSS yields hidden $\eta$ and $\eta^2$ extensions supported by $\bar{\kappa}$.  

We will now deduce the hidden extensions in the HFPSS  from multiplicative structure in the HOSS.  The resulting extensions are displayed in Figure~\ref{fig:HFPSSEinfty}.

Since $\eta\bar{\kappa}$ and $\eta^2 \bar{\kappa}$ are seen to be be non-trivial in $\pi_* \TMF_0(5)$ using hidden extensions in the HOSS, we obtain corresponding new hidden extensions in the HFPSS.  With the one exception $\eta \cdot \delta^2 \gamma \xi$, all of the other hidden extensions displayed in Figure~\ref{fig:HFPSSEinfty} follow from non-hidden extensions in the HOSS.  The remaining extension is addressed in the following lemma.

\begin{lemma}
In the homotopy fixed point spectral sequence for $\TMF_0(5)$, there is a hidden extension
$$ \eta \cdot \delta^2 \gamma \xi = \delta^{-1}\xi^6. $$
\end{lemma}

\begin{proof}
Observe that since $\nu^3$ is non-trivial in $\pi_*\TMF_0(5)$, and in $\pi_*\TMF$ we have $\nu^3 = \eta \epsilon$, it must follow that $\epsilon$ is detected by $\delta^{-2} \xi^4$ in the HFPSS.  Thus $\bar{\kappa} \epsilon$ is detected by $\delta^{-1} \xi^6$.  However, $\bar{\kappa} \epsilon$ is $\eta$-divisible in $\pi_*\TMF$.  It follows that it must also be $\eta$-divisible in $\pi_*\TMF_0(5)$, and the hidden extension claimed is the only possibility to make this happen.
\end{proof}

\begin{thm}\label{thm:TMF05}
The homotopy groups $\pi_*\TMF_0(5)$ are given by the following $\delta^4$-periodic pattern.

\includegraphics[width = \textwidth]{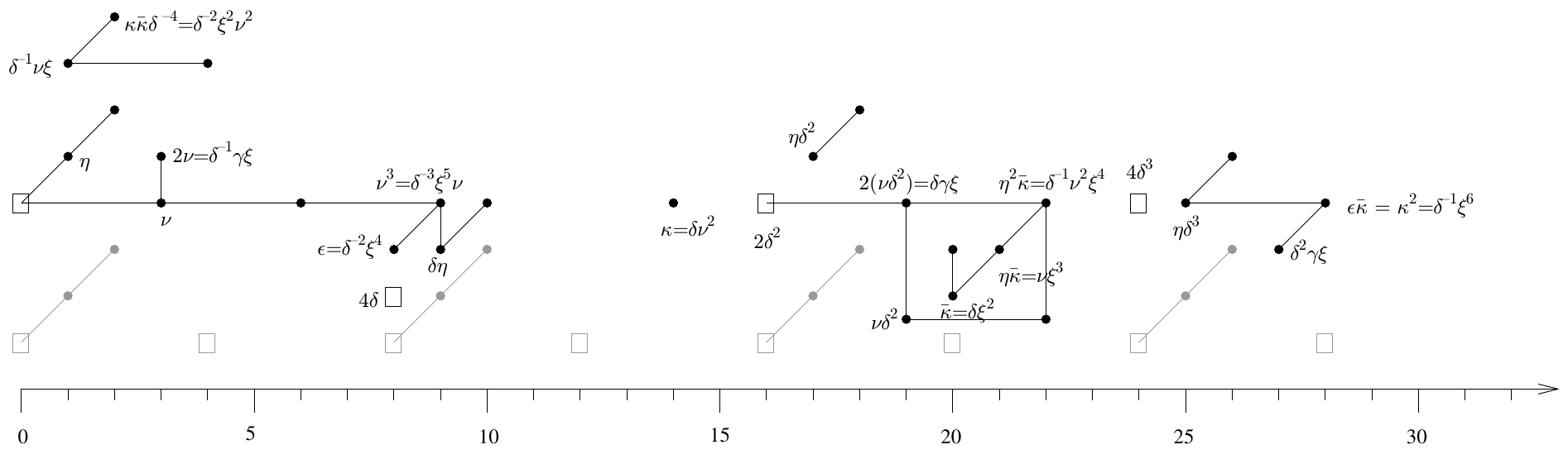}

Here, the grey classes in the figure above represent infinite direct sums of $bo$-patterns, generated ($(2)$-locally) by classes:
\begin{align*}
\delta^j b_2^{2k} \eta^a (\tilde{j}-11)^{-l}, \quad & j \in \ZZ, \: k > 0, \: 0 \le a \le 2, \: l \ge 0, \\
\delta^j b_2^{2k} b_4 \eta^a, \quad & j \in \ZZ, \: k \ge 0, \: 0 \le a \le 2, \\
2\delta^j b_2^{2k+1} (\tilde{j}-11)^{-l}, \quad & j \in \ZZ, \: k \ge 0, \: l \ge 0, \\
2\delta^j b_2^{2k+1} b_4, \quad & j \in \ZZ, \: k \ge 0, \: l \ge 0. \\
\end{align*}
\end{thm}

\begin{rmk}
One easily sees from the calculation of the $d_5$ and $d_7$-differentials that $\tilde{j}$ and hence $(\tilde{j}-11)^{-1}$ are permanent cycles in the homotopy fixed point spectral sequence.    It is reasonable to record how these act on the $\pi_*\TMF_0(5)$.  Amongst the classes of the form
$$ 2^i\delta^j b_2^k b_4^\epsilon \eta^a (\tilde{j}-11)^{-l} $$
(where we take $\epsilon \in \{0,1\}$, and  $l = 0$ if $\epsilon = 1$),
multiplication by $\tilde{j}$ is easy to compute using $\tilde{j} = b_4 \delta^{-1}$ and the relation
$$ b_4^2 = b_2^2\delta - 4\delta^2. $$
Multiplication by $(\tilde{j}-11)^{-1}$ is governed by the relation
$$ 11(\tilde{j} -11)^{-1} b_4 = b_2^2(\tilde{j}-11)^{-1} - 4\delta(\tilde{j}-11)^{-1}-b_4. $$
Amongst all other classes $x$ in the chart not of the form above, we have 
\begin{align*}
\tilde{j}x & = 0 \\
(\tilde{j}-11)^{-1} x & = x.
\end{align*}
\end{rmk}

\begin{rmk}
The relation $\epsilon \bar{\kappa} = \kappa^2$ in the chart of Theorem~\ref{thm:TMF05} corresponds to the same relation in the stable homotopy groups of spheres.  This relation represents a hidden $\epsilon$-extension in the classical Adams spectral sequence for the sphere (in the ASS, $c_0g = 0$ and $d_0^2$ detects the generator of $\pi^s_{28}$).  In the homotopy fixed point spectral sequence above, the relation 
$$ (\delta^{-2}\xi^4)(\delta \xi^2) = \delta^{-1} \xi^6 $$
implies that $\delta^{-1} \xi^6$ detects $\epsilon\bar{\kappa}$.  Actually, this gives an amusing alternative proof of the relation $\epsilon\bar{\kappa} = \kappa^2$ in $\pi_*^s$: the fact that $d_0^2$ is a permanent cycle in the ASS implies that $\kappa^2$ is non-trivial, and we have just seen that $\epsilon \bar{\kappa}$ must be non-trivial, because it is detected in the Hurewicz image of $\TMF_0(5)$.  Since $\pi^s_{28} = \ZZ/2$, the two classes must be equal.  One could make a similar argument using $\TMF$ instead of $\TMF_0(5)$, as one sees $\epsilon \bar{\kappa}$ in a similar way as a non-hidden extension in the ANSS for $\TMF$.
\end{rmk}

\section{$Q(\ell)$-spectra}
\label{sec:Ql}

We now begin working with the $Q(\ell)$ spectra in earnest.  We review
the definition of $Q(\ell)$ in \ref{subsec:Qdef} and in
\ref{subsec:double} recall the double complex that computes the
$E_2$-term of its Adams-Novikov spectral sequence.

In previous sections we have focused on data for $Q(5)$ but in
\ref{subsec:Q3} we review formulas of Mahowald and Rezk from
\cite{MahowaldRezk} related to $Q(3)$.  Finally in \ref{subsec:Q5} we
recall the formuals of Section \ref{sec:level5} in forms that will be useful
in subsequent calculations.

\subsection{Definitions}
\label{subsec:Qdef}

In \cite{Behrens}, the $p$-local spectrum $Q(\ell)$ ($p \nmid \ell$) is defined as the totalization of  an explicit semi-cosimplicial $E_\infty$-ring spectrum of the form
$$ Q(\ell)^\bullet = \bigg( \TMF \Rightarrow \TMF_0(\ell) \times \TMF \Rrightarrow \TMF_0(\ell) \bigg). $$
Here a semi-cosimplicial object is the same thing as a cosimplicial
object, but without codegeneracy maps.  The above expression is
shorthand for a semi-cosimplicial spectrum $Q(\ell)^\bullet$ in which
$Q(\ell)^k = *$ for $k>2$.  The coface maps from level $0$ to level $1$ are given by:
\begin{align*}
d_0 & = q^* \times \psi^\ell, \\
d_1 & = f^* \times 1,
\end{align*}
and the coface maps from level $1$ to level $2$ are given by
\begin{align*}
d_0 & = t^* \circ \pi_2, \\
d_1 & = f^* \circ \pi_1, \\
d_2 & = \pi_2, 
\end{align*}
where $\pi_i$ are the projections onto the components.
These maps are induced  by the maps of stacks
\begin{align*}
\psi^\ell: \mM^1 \rightarrow \mM^1, \quad & (C,\vec{v}) \mapsto (C, \ell \cdot \vec{v}), \\
f: \mM^1_0(\ell) \rightarrow \mM^1, \quad & (C,H, \vec{v}) \mapsto (C, \vec{v}), \\
q: \mM^1_0(\ell) \rightarrow \mM^1,\quad & (C,H, \vec{v}) \mapsto (C/H, (\phi_H)_* \vec{v}), \\
t: \mM^1_0(\ell) \rightarrow \mM^1_0(\ell), \quad & (C,H, \vec{v}) \mapsto (C/H, \widehat{H}, (\phi_H)_* \vec{v}),
\end{align*}
where $\phi_H: (C,H) \rightarrow C/H$ is the quotient isogeny.  (Note
that our $t$ is $\psi_d$ in \cite[p.349]{Behrens}, and we have
corrected a small typo in its presentation here.)  The map $\psi^\ell: MF_k \rightarrow MF_k$ is analogous to an Adams operation, and acts by multiplication by $\ell^k$.  Formulas for $f^*$, $q^*$, and $t^*$, on the level of modular forms are typically computed differently for different choices of $\ell$, and are more complicated.

\subsection{The double complex}
\label{subsec:double}
As done in the special case of $\ell = 2$ and $p = 3$ in \cite{Behrens}, one can form a total cochain complex to compute the $E_2$-term for the Adams-Novikov spectral sequence for $Q(\ell)$.  Let $(A, \Gamma)$ denote the usual elliptic curve Hopf algebroid, and let $(B^1(\ell), \Lambda^1(\ell))$ denote a Hopf algebroid which stackifies to give $\mM^1_0(\ell)$.  Let $C^*_{\Gamma}(A)$, $C^*_{\Lambda^1(\ell)}(B^1)$ denote the corresponding cobar complexes, so the corresponding Adams-Novikov spectral sequences take the form
\begin{gather*}
E_2^{s,2t} = H^s(\mM, \omega^{\otimes t}) = H^{s}(C^*_{\Gamma}(A)_{2t}) \Rightarrow \pi_{2t-s} \TMF, \\
E_2^{s,2t} = H^s(\mM_0(\ell), \omega^{\otimes t}) = H^{s}(C^*_{\Lambda^1(\ell)}(B^1(\ell))_{2t}) \Rightarrow \pi_{2t-s} \TMF_0(\ell). \\
\end{gather*}
Corresponding to the cosimplicial decomposition of $Q(\ell)$ we can form a double complex
$C^{*,*}(Q(\ell))$
\begin{equation}
\xymatrix{
\vdots & \vdots & \vdots \\
C^1_\Gamma(A) \ar[u] \ar[r] & C^1_{\Lambda^1(\ell)}(B^1(\ell)) \oplus C^1_\Gamma(A) \ar[r] \ar[u] & 
C^1_{\Lambda^1(\ell)}(B^1(\ell)) \ar[u] \ar[r] & \cdots \\
C^0_\Gamma(A) \ar[u] \ar[r] & C^0_{\Lambda^1(\ell)}(B^1(\ell)) \oplus C^0_\Gamma(A) \ar[r] \ar[u] & 
C^0_{\Lambda^1(\ell)}(B^1(\ell)) \ar[u] \ar[r] & \cdots
}
\end{equation}
Let $C^*_{tot}(Q(\ell))$ denote the corresponding total complex.  Then the Adams-Novikov spectral sequence for $Q(\ell)$ takes the form
$$ E_2^{s,2t} = H^s(C^*_{tot}(Q(\ell))_{2t}) \Rightarrow \pi_{2t-s} Q(\ell). $$

\subsection{Recollections about $Q(3)$}
\label{subsec:Q3}

Mahowald and Rezk \cite{MahowaldRezk} performed a study of the explicit formulas for $Q(3)$ similar to our current treatment of $Q(5)$.  We summarize some of their results here for the reader's convenience.

The moduli space $\mM^1_1(3)$ is represented by the affine scheme  $\Spec( B^1(3) )$ with
$$ B^1(3) = \ZZ[1/3, a_1, a_3, \Delta^{-1}] $$
with
$$ \Delta = a_3^3(a_1^3 - 27 a_3). $$
The corresponding universal $\Gamma_1(3)$ structure is carried by the Weierstrass curve
$$ y^2 + a_1 xy + a_3 y = x^3 $$
with point $P = (0,0)$ of order $3$.
The $\GG_m$-action on $\mM^1_1(3)$ induces a grading on $B^1(3)$, for which $a_i$ has weight $i$.  It follows that 
$$ \pi_* \TMF_1(3) = \ZZ[1/3, a_1, a_3, \Delta^{-1}] $$
with topological degrees $|a_i| = 2i$.  The spectrum $\TMF_1(3)$ admits a complex orientation with $v_1 = a_1$ and $v_2 = a_3$.

The group $\FF_3^\times = \{ \pm 1\}$ acts on $\mM_1^1(5)$ by sending an $R$-point $(C,P)$ (where $P$ is a point of exact order $3$ on $C$) to the $R$-point $(C,[-1](P))$.  This induced action of $\FF_3^\times$ on the ring 
$B^1(3)$ is given by 
\begin{align*}
[-1](a_1) & = -a_1, \\
[-1](a_3) & = -a_3.
\end{align*}
We have
$$ \mM^1_0(3) = \mM^1_1(3)// \FF_3^\times $$
and hence an equivalence
$$ \TMF_0(3) \simeq  \TMF_1(3)^{h\FF_3^\times}. $$
The resulting homotopy fixed point spectral sequence takes the form
$$ H^s(\FF_3^\times ; \pi_t \TMF_1(3)) \Rightarrow \pi_{t-s}\TMF_0(3). $$
In particular, the ring of modular forms (meromorphic at the cusps) for $\Gamma_0(3)$ is the subring 
$$ MF(\Gamma_0(3)) = H^0(\FF_3^\times ; MF(\Gamma_1(3)) = \ZZ[1/3, a_1^2, a_1 a_3, a_3^2, \Delta^{-1}] \subset B^1(3). $$

Mahowald and Rezk also compute the effects of the maps
\begin{gather*}
q^*, f^*: A \rightarrow B^1(3), \\
t^*: B^1(3) \rightarrow B^1(3)
\end{gather*}
as 
\begin{align*}
f^*(a_1) & = a_1, & q^*(a_1) & = a_1, \\
f^*(a_2) & = 0, & q^*(a_2) & = 0, \\
f^*(a_3) & = a_3, & q^*(a_3) & = 3a_3, \\
f^*(a_4) & = 0, & q^*(a_4) & = -6a_1 a_3, \\
f^*(a_6) & = 0, & q^*(a_6) & = -(9a_3^2+a_1^3 a_3), 
\end{align*}
\begin{align*}
t^* (a_1^2) & = -3 a_1^2, \\
t^* (a_1 a_3) & = \frac{1}{3} a_1^4 - 9 a_1 a_3, \\
t^* (a_3^2) & = -\frac{1}{27} a_1^6 + 2 a_1^3 a_3 - 27 a_3^2.
\end{align*}

\subsection{The formulas for $Q(5)$}
\label{subsec:Q5}

The moduli space $\mM^1_1(5)$ is represented by the affine scheme  $\Spec( B^1(5) )$ with
$$ B^1(5) = \ZZ[1/5, a_1, u, \Delta^{-1}] $$
with
$$
\Delta = -11u^{12}+64a_1u^{11}-154a^2_1 u^{10}+195a^3_1 u^9-135a^4_1 u^8+46a^5_1u^7-4a^6_1 u^6 - a^7_1 u^5.
$$
The corresponding universal $\Gamma_1(5)$ structure is carried by the Weierstrass curve
$$ y^2 + a_1 xy + (a_1 u^2 - u^3) y = x^3 + (a_1 u - u^2) x^2 $$
with point $P = (0,0)$ of order $5$.
The $\GG_m$-action on $\mM^1_1(5)$ induces a grading on $B^1(5)$, for which $a_1$ and $u$ both have weight $1$.  It follows that 
$$ \pi_* \TMF_1(5) = \ZZ[1/5, a_1, u, \Delta^{-1}] $$
with topological degrees $|a_1| = |u| = 2$.  The spectrum $\TMF_1(5)$ admits a complex orientation with $v_1 = a_1$ and $v_2 \equiv u^3 \mod (2,v_1)$.

The group $\FF_5^\times \cong C_4$ acts on $\mM_1^1(5)$: for $5 \nmid n$, the mod $5$ reduction $[n] \in \FF_5^\times$ acts by sending an $R$-point $(C,P)$ (where $P$ is a point of exact order $5$ on $C$) to the $R$-point $(C,[n](P))$.  This induced action of the generator $[2]$ of $\FF_5^\times$ on the ring 
$B^1(5)$ is given by 
\begin{align*}
[2](a_1) & = a_1 - 2u, \\
[2](u) & = a_1 - u.
\end{align*}
These have the more convenient expression
\begin{align*}
[2](u) & = b_1, \\
[2](b_1) & = - u
\end{align*}
where $b_1 := a_1 - u$. 
We have
$$ \mM^1_0(5) = \mM^1_1(5)// \FF_5^\times $$
and hence an equivalence
$$ \TMF_0(5) \simeq  \TMF_1(5)^{h\FF_5^\times}. $$
The resulting homotopy fixed point spectral sequence takes the form
$$ H^s(\FF_5^\times ; \pi_t \TMF_1(5)) \Rightarrow \pi_{t-s}\TMF_0(5). $$
In particular, the ring of modular forms (meromorphic at the cusps) for $\Gamma_0(5)$ is the subring 
$$ MF(\Gamma_0(5)) = H^0(\FF_5^\times ; MF(\Gamma_1(5)) = \frac{\ZZ[1/5, b_2, b_4, \delta][\Delta^{-1}]}{(b_4^2 = b_2^2 \delta-4\delta^{2})} \subset B^1(5) $$
where
\begin{align*}
b_2 & := u^2 + b_1^2, \\
b_4 & := u^3 b_1 - u b_1^3, \\
\delta & := u^2 b_1^2.
\end{align*}
Note that $\delta$ is almost a cube root of $\Delta$; we have
$$ \Delta = \delta^{2} b_4 - 11\delta^3. $$

The effect of the maps
\begin{gather*}
q^*, f^*: A \rightarrow B^1(5), \\
t^*: B^1(5) \rightarrow B^1(5)
\end{gather*}
is
\begin{align*}
f^*(a_1) & = a_1, & q^*(a_1) & = a_1, \\
f^*(a_2) & = a_1 u - u^2, & q^*(a_2) & = - u^2 + a_1 u, \\
f^*(a_3) & = a_1 u^2 - u^3, & q^*(a_3) & = - u^3 + a_1 u^2 , \\
f^*(a_4) & = 0, & q^*(a_4) & = -10 u^4 + 30 a_1 u^3 - 25 a_1^2 u^2 + 5 a_1^3 u, \\
f^*(a_6) & = 0, & q^*(a_6) & = -20u^6 + 59a_1 u^5 - 70a_1^2 u^4 + 45a_1^3 u^3 - 15a_1^4 u^2 + a_1^5 u, 
\end{align*}
\begin{align*}
t^*(a_1) & = \frac{1}{5}(-8 \zeta^3 - 6\zeta^2 - 14 \zeta - 7) a_1 + \frac{1}{5}(14 \zeta^3 - 2\zeta^2 + 12\zeta + 6)u, \\
t^*(u) & = \frac{1}{5}(-\zeta^3 - 7\zeta^2 - 8 \zeta - 4)a_1 + \frac{1}{5}(8 \zeta^3 + 6 \zeta^2 + 14 \zeta + 7)u.
\end{align*}
In the formulas for $t^*$, we use $\zeta$ to denote a $5$th root of
unity.  This results in the following formula for $f^*,q^*,t^*$ on
rings of modular forms:
\begin{align*}
f^*(c_4) &= b_2^2-12b_4+12\delta, & q^*(c_4) &= b_2^2+228b_4+492\delta,\\
f^*(c_6) &= -b_2^3+18b_2b_4-72b_2\delta, & q^*(c_6) &= -b_2^3+522b_2b_4+10,008b_2\delta,
\end{align*}
\begin{align*}
t^*(b_2) & = -5b_2, \\
t^*(b_4) & = \frac{1}{5}(11 b_2^2 - 117 b_4 - 88 \delta), \\
t^*(\delta) & = \frac{1}{5}(b_2^2 - 22 b_4 + 117 \delta).
\end{align*}

\section{Detection of the $\beta$-family by $Q(3)$ and $Q(5)$}
\label{sec:beta}

The Miller-Ravenel-Wilson divided $\beta$-family \cite{MillerRavenelWilson} is an important algebraic
approximation of the $K(2)$-local sphere at the prime 2.  It was
computed for the prime 2 by Shimomura in \cite{Shimomura}.  Here we use
the standard chain of Bockstein spectral sequences and the formulas of
\ref{subsec:Q3} and \ref{subsec:Q5} to compute
algebraic chromatic data in the $Q(3)$ and $Q(5)$ spectra.  These
are compared to Shimomura's calculations, resulting in Theorems
\ref{thm:Q3beta} and \ref{thm:Q5beta}.  The surprising observation is that
$Q(3)$ precisely detects the divided $\beta$-family, while the analgous family
in $Q(5)$ has extra $v_1$-divisibility.

\subsection{The chromatic spectral sequence}
\label{subsec:chrom}

Following \cite{MillerRavenelWilson}, given a $BP_*$-module $N$, we will let 
$$ M_i^{n-i} N := N/(p, \cdots, v_{i-1}, v_i^\infty,  \cdots, v^\infty_{n-1}) [v^{-1}_{n}]. $$  
If $N$ is a $BP_*BP$-comodule, then so is $M_i^{n-i} N$.  Letting $\Ext^{*,*}(N)$ denote the groups
$$ \Ext^{*,*}_{BP_*BP}(BP_*, N), $$
there is a chromatic spectral sequence
$$ E_1^{n, s,t} = \Ext^{s,t}(M_0^{n} N) \Rightarrow \Ext^{s+n,t}(N).  $$
The groups $\Ext^{0,*}(M_0^n BP_*)$ detect the $n$th Greek letter elements in $\Ext^{*,*}(BP_*)$.

The $E_1$-term of this spectral sequence may be computed by first computing the groups $\Ext^{*,*}(M_n^0)$ and then using the $v_i$-Bockstein spectral sequences (BSS) of the form
$$ \Ext^{*,*}(M_{i+1}^{n-i-1} N) \otimes \FF_p[v_i]/(v_i^\infty) \Rightarrow \Ext^{*,*}(M_i^{n-i} N). $$

\subsection{Statement of results}

For the remainder of this section we work exclusively at the prime $2$.
Shimomura used these spectral sequences to make the following computation.

\begin{thm}[\cite{Shimomura}]
\label{thm:Shim}
The groups $\Ext^0(M_0^2 BP_*)$ are spanned by the elements: 
\begin{align*}
& \frac{1}{2^k v_1^j},  &  & j \ge 1 \: \text{and} \:  k \le k(j); \\
\\  
& \frac{v_2^{m2^n}}{2^k v_1^j}, & & 2 \nmid m, \: k \le k(j), \: \text{and}\\
& & &  j  \le \begin{cases}
a(1), & k = 3, n = 2, \\
a(n-k+1), & \text{otherwise},
\end{cases}
\end{align*}
where
$$
k(j) := \begin{cases}
1,  & j \not \equiv 0 \mod 2, \\
\nu_2(j)+2, & j \equiv 0 \mod 2,
\end{cases}
$$
and
$$
a(i) := 
\begin{cases}
1, & i = 0, \\
2, & i = 1, \\
3\cdot 2^{i-1}, & i \ge 2. 
\end{cases}
$$
\end{thm}

The `names' $v_2^i/2^kv_1^j$ are not the exact names of $BP_*BP$-primitives in $M_0^2 BP_*$, but rather the names of the elements detecting them in the sequence of BSS's:
$$ \Ext^{*,*}(M_2^0 BP_*) \otimes \frac{\FF_2[v_0, v_1]}{(v_0^\infty, v_1^\infty)} \Rightarrow \Ext^{*,*}(M_1^1 BP_*) \otimes \frac{\FF_2[v_0]}{(v_0^\infty)} \Rightarrow \Ext^{*,*}(M_0^2 BP_*). $$
Put a linear order on the monomials $v_0^k v_1^j$ in $\FF_2[v_0^k, v_1^j]$ by left lexicographical ordering on the sequence of exponents $(k,j)$.
With respect to this ordering, the actual primitives correspond to elements  
$$ \frac{v_2^{i}}{2^k v_1^j} + \text{terms with smaller denominators}. $$

The main theorem of this section is the following.

\begin{thm}\label{thm:Q3beta}
The map
$$ \Ext^0(M_0^2 BP_*) \rightarrow H^0(M_0^2 C^*_{tot}(Q(3))) $$
is an isomorphism.
\end{thm}

\begin{rmk}
It was observed by Mahowald and Rezk \cite{MahowaldRezk} that the map
$$ \Ext^0(M_1^1 BP_*) \rightarrow H^0(M_1^1 C^*_{tot}(Q(3))) $$
is an isomorphism.
\end{rmk}

However, the same cannot hold for $Q(5)$.  Indeed, the following theorem implies it does not even hold on the level of $M_1^1$.

\begin{thm}\label{thm:Q5beta}
The map
$$ \Ext^0(M_1^1 BP_*) \rightarrow H^0(M_1^1 C^*_{tot}(Q(5))) $$
is \emph{not} an isomorphism.
\end{thm}

\subsection{Leibniz and doubling formulas}

The group $H^0(M_0^2 C^*_{tot}(Q(\ell)))$ is the kernel of the map
$$ M_0^2 C^0_{tot}(Q(\ell)) \xrightarrow{d_0 - d_1} M_0^2 C^1_{tot}(Q(\ell)) $$
where $d_0$ and $d_1$ are the cosimplicial coface maps of the total complex.
Explicitly, we are applying $M_0^2$ to the map
$$ D_{tot}: A_{(2)} \xrightarrow{(\eta_R - \eta_L) \oplus (q^*-f^*) \oplus (\psi^\ell - 1)} \Gamma_{(2)} \oplus B^1(\ell)_{(2)} \oplus A_{(2)}. $$
The projection of $D_{tot}$ onto the last component is very easy to understand; it is given by
$$ \psi^\ell - 1: A \rightarrow A. $$
As long as $\ell$  generates $\ZZ_2^\times/\{\pm 1\}$, in degree $2t$ the map $\psi^\ell-1$, up to a unit in $\ZZ^\times_{(2)}$, corresponds to multiplication by a factor of $2^{k(t)}$.  It therefore suffices to understand the composite $D$ of $D_{tot}$ with the projection onto the first two components:
$$ D: A_{(2)} \xrightarrow{(\eta_R - \eta_L) \oplus (q^*-f^*)} \Gamma_{(2)} \oplus B^1(\ell)_{(2)}. $$

We shall make repeated use of the following lemma about this map $D$.

\begin{lemma}\label{lem:Dleib}
The map $D$ satisfies the following two identities.
\begin{align}
D(xy) & = D(x) \eta_R(y) + x D(y), \label{eq:Dproduct} \\
D(x^2) & = 2xD(x) + D(x)^2. \label{eq:Dsquare} 
\end{align}
Here, $\Gamma$ is given the $A$-module structure induced by the map $\eta_L$, and $B^1(3)$ is given the $A$-module structure induced from the map $f^*$.  Consequently, we have
\begin{equation}\label{eq:Dleib} 
D(xy) \equiv xD(y) \mod (D(x)).
\end{equation}
\end{lemma}

\begin{proof}
These identities hold for any map $D = d_0 - d_1: R^0 \rightarrow R^1$, the difference of two ring maps:
\begin{align*}
D(xy) & = d_0(x)d_0(y)- d_1(x)d_1(y) \\
& = d_1(x)(d_0(y)- d_1(y))  +(d_0(x)- d_1(x))d_0(y) \\
& = d_1(x) D(y) + D(x)d_0(y). 
\\
D(x^2) & = 
 d_0(x)^2- d_1(x)^2 \\
& = (d_0(x)- d_1(x))^2 + 2d_0(x)d_1(x) - 2d_1(x)^2 \\
& = D(x)^2 + 2d_1(x) D(x).
\end{align*}
\end{proof}

Observe that using the fact that $a_1 = v_1$, there are isomorphisms
\begin{align*}
\Gamma_{(2)} & \cong \ZZ_{(2)}[v_1][a_2, a_3, a_4, a_6, r, s, t][\Delta^{-1}], \\
B^1(3)_{(2)} & \cong \ZZ_{(2)}[v_1][a_3][\Delta^{-1}], \\
B^1(5)_{(2)} & \cong \ZZ_{(2)}[v_1][u][\Delta^{-1}].
\end{align*}
Express elements of  $\Gamma_{(2)}$ (respectively, $B^1(3)_{(2)}$, $B^1(5)_{2}$) ``$(2,v_1)$-adically'' so that every element is expressed as a power of the discriminant times a  sum of terms
$$  \Delta^{\ell} \sum_{k \ge 0} \sum_{j \ge 0}  2^k v_1^j c_{k,j} $$
for $\ell \in \ZZ$ and $c_{k,j} \in \FF_2[a_2, a_3, a_4, a_6, r, s, t]$ (respectively $\FF_2[a_3])$, $\FF_2[u]$).  We shall compare terms by saying that
$$ 2^k v_1^j c_{j,k} \quad \text{\emph{is larger than}} \quad 2^{k'} v_1^{j'} c_{j',k'} $$
if $(k, j)$ is larger than $(k', j')$ with respect to left lexicographical ordering.  We shall be concerned with ordered sums of monomials of the form:
\begin{eqnarray*}
&& v^{i_0}_1c_{0,i_0} + \text{terms of the form $v_1^j c_{0,j}$ with $j > i_0$ }\\
&& +2v^{i_1}_1c_{1, i_1}+  \text{terms of the form $2v_1^j c_{1,j}$ with $j > i_1$}\\
&& +4v^{i_2}_1c_{2, i_2}+  \text{terms of the form $4v_1^j c_{2,j}$ with $j > i_2$}\\
&& + \cdots \\
&& +2^nv^{i_n}_1c_{n, i_n}+  \text{larger terms}
\end{eqnarray*}
for $(i_0 > i_1 > \cdots > i_n)$ and $n \ge 1$.  Note that we permit the coefficients $c_{k,i_k}$ to be zero.   We shall abbreviate such expressions as 
\begin{align*}
& v^{i_0}_1c_{0,i_0} + \cdots \\
& \quad +2v^{i_1}_1c_{1, i_1}+  \cdots \\
& \quad \quad +4v^{i_2}_1c_{2, i_2}+  \cdots \\
& \quad \quad \quad + \cdots \\
& \quad \quad \quad \quad +2^nv^{i_n}_1c_{n, i_n}+  \cdots. 
\end{align*}

The following observation justifies considering such representations.

\begin{lemma}
Suppose that $x \in A_{(2)}$ satisfies 
\begin{align*}
D(x) = & v^{i_0}_1c_{0,i_0} + \cdots \\
& \quad +2v^{i_1}_1c_{1, i_1}+  \cdots \\
& \quad \quad +4v^{i_2}_1c_{2, i_2}+  \cdots \\
& \quad \quad \quad + \cdots \\
& \quad \quad \quad \quad +2^nv^{i_n}_1c_{n, i_n}+  \cdots 
\end{align*}
Then we have
\begin{equation}\label{eq:square}
\begin{split}
D(x^2) = & v^{2i_0}_1c^2_{0,i_0} + \cdots \\
& \quad + 2v^{i_0}_1c_{0,i_0}x + \cdots \\
& \quad \quad +4v^{i_1}_1c_{1, i_1}x+  \cdots \\
& \quad \quad \quad +8v^{i_2}_1c_{2, i_2}x+  \cdots \\
& \quad \quad \quad \quad + \cdots \\
& \quad \quad \quad \quad \quad +2^{n+1}v^{i_n}_1c_{n, i_n}x+  \cdots 
\end{split}
\end{equation}
and for $m$ odd we have
\begin{equation}\label{eq:oddpower}
\begin{split}
D(x^{m}) = & v^{i_0}_1c_{0,i_0}x^{m-1} + \cdots \\
& \quad +2v^{i_1}_1c_{1, i_1}x^{m-1} +  \cdots \\
& \quad \quad +4v^{i_2}_1c_{2, i_2}x^{m-1} +  \cdots \\
& \quad \quad \quad + \cdots \\
& \quad \quad \quad \quad +2^nv^{i_n}_1c_{n, i_n}x^{m-1} +  \cdots. 
\end{split}
\end{equation}
\end{lemma}

\begin{proof}
The identity (\ref{eq:square}) follows immediately from (\ref{eq:Dsquare}).  We prove (\ref{eq:oddpower}) by induction on $m = 2j+1$.  Suppose that we know (\ref{eq:oddpower}) for all odd $m' < m$.  Write $j  = 2^t s$ for $s$ odd.  Then by the inductive hypothesis, and repeated applications of (\ref{eq:square}), we deduce that
\begin{align*}
D(x^{j}) = & v^{i_0}_1c'_{0,i_0} + \cdots \\
& \quad +2v^{i_1}_1c'_{1, i_1} +  \cdots \\
& \quad \quad +4v^{i_2}_1c'_{2, i_2} +  \cdots \\
& \quad \quad \quad + \cdots \\
& \quad \quad \quad \quad +2^nv^{i_n}_1c'_{n, i_n} +  \cdots. 
\end{align*}
Applying (\ref{eq:square}), we have
\begin{align*}
D(x^{2j}) = & v^{2i_0}_1(c'_{0,i_0})^2 + \cdots \\
& \quad + 2v^{i_0}_1c'_{0,i_0}x^j + \cdots \\
& \quad \quad +4v^{i_1}_1c'_{1, i_1}x^j+  \cdots \\
& \quad \quad \quad +8v^{i_2}_1c'_{2, i_2}x^j+  \cdots \\
& \quad \quad \quad \quad + \cdots \\
& \quad \quad \quad \quad \quad +2^{n+1}v^{i_n}_1c'_{n, i_n}x^j+  \cdots. 
\end{align*}
It follows from (\ref{eq:Dleib}) that we have
\begin{align*}
D(x^{2j+1}) = D(x^{2j}x) = & v^{i_0}_1c_{0,i_0}x^{2j} + \cdots \\
& \quad +2v^{i_1}_1c_{1, i_1}x^{2j} +  \cdots \\
& \quad \quad +4v^{i_2}_1c_{2, i_2}x^{2j} +  \cdots \\
& \quad \quad \quad + \cdots \\
& \quad \quad \quad \quad +2^nv^{i_n}_1c_{n, i_n}x^{2j} +  \cdots. 
\end{align*}
\end{proof}

\subsection{Overview of the technique}
\label{subsec:technique}

The technique for the proof of Theorem~\ref{thm:Q3beta} is as follows (following \cite{MillerRavenelWilson} and \cite{Shimomura}):
\begin{description} 
\item[Step 1] Compute the differentials from the $s = 0$ to the $s = 1$-lines in the $v_1$-BSS
\begin{equation}\label{eq:v1BSSa}
H^{s,*}(M_2^0 C^*_{tot}(Q(3))) \otimes \FF_2[v_1]/(v_1^\infty) \Rightarrow H^{s,*}(M_1^1 C^*_{tot}(Q(3))). \end{equation}
This establishes the existence and $v_1$-divisibility of $v_2^i/v^j_1$ in $H^{0,*}(C^*_{tot}(Q(3)))$.

\item[Step 2] For $i,j$ as above, demonstrate that $v_2^i/2^k v_1^j$ exists in $H^{0,*}(M_0^2 C^*_{tot}(Q(3)))$ by writing down an element
$$ x_{i/j,k} = \frac{a_3^{i}}{2^k v_1^j} + \text{terms with smaller denominators}  \in M_0^2 A $$
with $D_{tot}(x) = 0$.

\item[Step 3] Given $j$, find the maximal $k$ such that $x_{i/j,k}$ exists by using the exact sequence
$$ H^{0,*}(M_0^2 C^*_{tot}(Q(3))) \xrightarrow{\cdot 2}  H^{0,*}(M_0^2 C^*_{tot}(Q(3))) \xrightarrow{\partial} H^{1,*}(M_1^1 C^*_{tot}(Q(3))). $$
Specifically, the maximality of $k$ is established by showing that $\partial(x_{i/j,k}) \ne 0$.  The non-triviality of  $\partial(x_{i/j,k})$ can be demonstrated by considering its image under the inclusion:
$$ H^{1,*}(M_1^1 C^*_{tot}(Q(3))) \hookrightarrow \mathrm{Coker} \, M_1^1(D_{tot}) $$
where $M_1^1(D_{tot})$ is the map
$$ M_1^1 (D_{tot}): M_1^1 A \rightarrow M_1^1\Gamma \oplus M_1^1 B^1(3) \oplus M_1^1 A $$
essentially computed in Step 1  by the computation of the differentials from $s = 0$ to $s = 1$ in the spectral sequence (\ref{eq:v1BSSa}).
\end{description}

\subsection{Computation of $\pmb{H^{*,*}(M_2^0 C^*_{tot}(Q(3)))}$}

We have \cite[Sec.~7]{Tilman}
\begin{gather*}
H^{*,*}(M_2^0 C^*_{\Gamma}(A))  = \FF_2[a_3^{\pm 1},  h_{1}, h_2, g]/(h_2^3 = a_3 h_1^3), \\
H^{*,*}(M_2^0 C^*_{\Lambda^1(3)}(B^1))  = \FF_2[a^{\pm}_3, h_{2,1}]  
\end{gather*}
with $(s,t)$-bidegrees
\begin{align*}
| a_3 | & = (0,6),  \\
| h_1 | & = (1, 2),  \\
| h_2 | & = (1,4),  \\
| g | & = (4, 24), \\ 
| h_{2,1} | & = (1,6),
\end{align*} 
and $h_{2,1}^4 = g$.  
Moreover, the spectral sequence of the double complex gives
\begin{equation}\label{eq:ssdouble}
 \begin{array}{c}
H^{s,t}(M_2^0 C^*_{\Gamma}(A)) \oplus  \\
H^{s-1,t}(M_2^0 C^*_{\Gamma}(A)) \oplus H^{s-1,t}(M_2^0 C^*_{\Lambda^1(3)}(B^1)) \oplus \\ 
H^{s-2,t}(M_2^0 C^*_{\Lambda^1(3)}(B^1)) 
\end{array}
\Bigg{\}}
\Rightarrow  H^{s,t}(M_2^0 C^*_{tot}(Q(3))).
\end{equation}
In order to differentiate the terms $x$ with the same name (such as $a_3$) occurring in the different groups in the $E_1$-term of spectral sequence (\ref{eq:ssdouble}), we shall employ the following notational convention:
\begin{equation*}
\begin{split}
x & \in C^*_{\Gamma}(A) \: \text{on the $0$-line}, \\
\bar{x} & \in C^*_{\Gamma}(A) \: \text{on the $1$-line}, \\
x' & \in C^*_{\Lambda^1(\ell)}(B^1) \: \text{on the $1$-line}, \\
\bar{x}' & \in C^*_{\Lambda^1(\ell)}(B^1) \: \text{on the $2$-line}.
\end{split}
\end{equation*}
The formulas of Section~\ref{subsec:Q3} show that the only non-trivial $d_1$ differentials in spectral sequence (\ref{eq:ssdouble}) are 
$$ d_1(g^i (\bar{a}_3)^j) = h_{2,1}^{4i} (\bar{a}'_3)^j. $$
Since $g$ is the image of the element $g \in \Ext^{4,24}(BP_*)$ (the element that detects $\bar{\kappa}$ in the ANSS for the sphere), and the spectral sequence (\ref{eq:ssdouble}) is a spectral sequence of modules over $\Ext^{*,*}(BP_*)$, we deduce that there are no possible $d_r$-differentials for $r > 1$.  We deduce that we have 
\begin{align*}
H^{*,*}(M_2^0 C^*_{tot}(Q(3))) & =  \FF_2[a_3^{\pm 1},  h_{1}, h_2, g]/(h_2^3 = a_3 h_1^3) \\
& \quad \oplus \FF_2[\bar{a}_3^{\pm 1}, \bar{g}]\{ \bar{h}_1, \bar{h}_2, \bar{h}_1^2, \bar{h}^2_2, \bar{h}_2^3 = \bar{a}_3 \bar{h}_1^3\} \\
& \quad \oplus  \FF_2[(a'_3)^{\pm 1}, h'_{2,1}]  \\
& \quad \oplus  \FF_2[(\bar{a}'_3)^{\pm 1}, \bar{g}']\{ \bar{h}'_{2,1}, (\bar{h}'_{2,1})^2, (\bar{h}'_{2,1})^3 \}.
\end{align*}

\begin{rmk}\label{rmk:Q3wontwork}
Note that $H^{*,*}(M_2^0 C^*_{tot}(Q(3)))$ is less than half of $\Ext^{*,*} (M_2^0 BP_*)$.  This indicates that $Q(3)$ cannot agree with `half' of the proposed duality resolution of Goerss-Henn-Mahowald-Rezk at $p = 2$ \cite{Henn}, despite the fact that it is built from the same spectra.  In particular, the fiber of the map
$$ S_{K(2)} \rightarrow Q(3)_{K(2)} $$
\emph{cannot} be the dual of $Q(3)_{K(2)}$.
\end{rmk}

\subsection{Computation of $\pmb{H^{0,*}(M_1^1 C^*_{tot}(Q(3)))}$}\label{sec:Q3v1BSS}

We now compute the differentials in the $v_1$-BSS
\begin{equation}\label{eq:v1BSS}
H^{s,*}(M_2^0 C^*_{tot}(Q(3))) \otimes \FF_2[v_1]/(v_1^\infty) \Rightarrow H^{s,*}(M_1^1 C^*_{tot}(Q(3)))
\end{equation}
from the $s = 0$-line to the $s = 1$-line.  This computation was originally done by Mahowald and Rezk \cite{MahowaldRezk}, but we redo it here to establish notation, and to motivate the rationale behind some of the computations to follow.

One computes using the formulas of Section~\ref{subsec:Q3}:
\begin{equation}\label{eq:Dxi}
\begin{split}
D(x_0) & \equiv a_1 s^2 \mod (2, v_1^2), \\
D(x_1) & \equiv a_1^2 a_3 s \mod (2, v_1^3), \\
D(x_2) & \equiv (a'_1)^6 (a'_3)^2 \mod (2, v_1^7)
\end{split}
\end{equation}
for 
\begin{align*}
x_0 & := a_3 + a_1a_2 \equiv a_3 \mod (2, v_1) \\
x_1 & := x_0^2 + a_1^2 a_4 + a_1^2 a_2^2 \equiv a_3^2 \mod (2, v_1) , \\
x_2 & := \Delta \equiv a_3^4 \mod (2, v_1).
\end{align*}

\begin{rmk}\label{rmk:kill}
The above formulas for $x_i$ were obtained by the following method.  In the complex $M_2^0 C^*_{\Gamma}(A)$, we have
\begin{align*}
d(a_2) & = r + \cdots \\
d(a_4+a_2^2) & = s^4 + \cdots \\
d(a_6) & = t^2 + \cdots.
\end{align*}
These are used in \cite[Sec.~6]{Tilman} to produce a complex which is closely related to the cobar complex on the double of $A(1)_*$.  To arrive at $x_0$ we calculate
$$ D(a_3) = a_1 r + \cdots $$
which means that we need to add the correction term $a_1 a_2$ to arrive at $x_0$.  The expression for $x_1$ was similarly produced.  The definition $\Delta$ is a natural candidate for $x_2$, as it is an element of the form $a_3^4  + \cdots $ which is already known to be a cocycle in $C^0_{\Gamma}(A)$.
\end{rmk}

It follows from inductively applying (\ref{eq:square}) that we have
$$ D(x_2^{2^{n-2}}) \equiv (a'_1)^{3 \cdot 2^{n-1}}(a'_3)^{2^{n-1}} \mod (2, v_1^{3 \cdot 2^{n-1}+1}). $$
It follows from $(\ref{eq:oddpower})$ that for $m$ odd, we have
\begin{align*}
D(x^m_0) & \equiv a_1 s^2 a_3^{m-1} \mod (2, v_1^2), \\
D(x_1^m) & \equiv a_1^2 a^{2m-1}_3 s \mod (2, v_1^3), \\
D(x_2^{m2^{n-2}}) & \equiv (a'_1)^{3 \cdot 2^{n-1}}(a'_3)^{m 2^n - 2^{n-1}} \mod (2, v_1^{3 \cdot 2^{n-1}+1}).
\end{align*}

We deduce the following.

\begin{lemma}\label{lem:v1BSS}
The $v_1$-BSS differentials in (\ref{eq:v1BSS}) from the $(s = 0)$-line to the $(s = 1)$-line are given by
\begin{align*}
d_{1}\left( \frac{a_{3}^{m}}{v_1^j}\right) & = \frac{a^{m-1}_{3} h_2}{v_1^{j-1}}, \\
d_{2}\left( \frac{a_{3}^{2m}}{v_1^j}\right) & = \frac{a^{2m-1}_{3} h_1}{v_1^{j-2}}, \\
d_{3 \cdot 2^{n-1}} \left( \frac{a_3^{m2^n}}{v_1^j} \right) & = \frac{(a'_3)^{m2^n - 2^{n-1}}}{v_1^{j - 3\cdot 2^{n-1}}}
\end{align*}
where $m$ is odd.
\end{lemma}

\begin{cor}
The groups $H^{0,*}(M_1^1C^*_{tot}(Q(3)))$ are generated by the elements  
$$ \frac{a_3^{m2^n}}{v_1^{j}}$$ 
for $m$ odd and $j \le a(n)$.
\end{cor}

\subsection{Computation of $\pmb{H^{0,*}(M_0^2 C^*_{tot}(Q(3)))}$}

We now prove Theorem~\ref{thm:Q3beta}, which is more specifically stated below.

\begin{thm}\label{thm:Q3beta2}
The groups $H^{0,*}(M_0^2 C^*_{tot}(Q(3)))$ are spanned by elements:
\begin{align*}
& \frac{1}{2^k v_1^j},  &  & j \ge 1 \: \text{and} \:  k \le k(j); \\
\\  
& \frac{a_3^{mp^n}}{2^k v_1^j}, & & 2 \nmid m, \: k \le k(j), \: \text{and}\\
& & &  j  \le \begin{cases}
a(1), & k = 3, n = 2, \\
a(n-k+1), & \text{otherwise}.
\end{cases}
\end{align*} 
\end{thm}

In many cases, the bounds on $2$-divisibility will follow from the following simple observation.

\begin{lemma}\label{lem:psilimit}
Suppose the element
$$ \frac{a_3^i}{2^k v_1^j} \in H^{0,2t}(M_0^2C_{tot}^*(Q(3))) $$
exists. Then $k \le k(t)$.
\end{lemma}

\begin{proof}
The formula
$$ (\psi^3 - 1) \frac{a_3^i}{2^k v_1^j} = (3^t-1) \frac{\bar{a}_3^i}{2^k v_1^j} $$
implies that in order for
$$ 0 \ne D_{tot}\left( \frac{a_3^i}{2^k v_1^j} \right) \in M_0^2 C^1_{tot}(Q(3)) $$
we must have $k \le \nu_2(3^t-1)$.
\end{proof}

\begin{proof}[Proof of Theorem~\ref{thm:Q3beta2}]
Lemma~\ref{lem:v1BSS} established that for $m$ odd, $\frac{a_3^{m2^n}}{2v_1^j}$ exists for $1 \le j \le a(n)$.
In order to prove the required $2$-divisibility of these elements, we need to prove
\begin{align}
D\left( \frac{a_3^{4m}}{8 v_1^{2}} + \cdots  \right) & = 0, & &  \label{eq:Q3beta1}\\
D\left( \frac{a_3^{m2^n}}{4v_1^{2j}} + \cdots \right) & = 0, & & 2j \le a(n-1), \label{eq:Q3beta2} \\
D\left( \frac{a_3^{m2^n}}{2^k v_1^{j2^{k-2}}} + \cdots \right) & = 0, & & k \ge 3, \: j2^{k-2} \le a(n-k+1).  \label{eq:Q3beta3}
\end{align}
In light of Lemma~\ref{lem:psilimit}, to establish that these are the maximal $2$-divisibilities of these elements, we need only check that
\begin{align}
\partial\left( \frac{a_3^{m}}{2v_1} + \cdots \right) & \not\equiv 0 \mod D(M_1^1A), & &  \label{eq:Q3beta4} \\
\partial\left( \frac{a_3^{m2^n}}{2v_1^{2j}} + \cdots \right) & \not\equiv 0 \mod D(M_1^1A), & & a(n-1) < 2j \le a(n), \label{eq:Q3beta5} \\
\partial\left( \frac{a_3^{m2^n}}{2^{k-1} v_1^{j2^{k-2}}} + \cdots \right) & \not\equiv 0 \mod D(M_1^1A), & & k \ge 2, \: a(n-k+1) < j2^{k-1} \le a(n-k+2). \label{eq:Q3beta6}
\end{align}

\noindent
{\bf Proof of (\ref{eq:Q3beta4}).}
Using the formulas of Section~\ref{subsec:Q3}, we have
\begin{equation}\label{eq:Dx0}
\begin{split}
D(x_0) & = a_1 s^2 + \cdots \\
& \quad \quad + 2(t + rs + s^3 + a_2s) + \cdots \\
& \quad \quad + 2a_3' + \cdots. 
\end{split}
\end{equation}
It follows from (\ref{eq:oddpower}) that we have for $m$ odd
\begin{equation}\label{eq:Dx0^m}
\begin{split}
D(x^m_0) & = a_1 a_3^{m-1} s^2 + \cdots \\
& \quad \quad + 2a_3^{m-1}(t + rs + s^3 + a_2s) + \cdots \\
& \quad \quad + 2(a_3')^m + \cdots. 
\end{split}
\end{equation}
Since we have
\begin{equation}\label{eq:Da1}
\eta_R(a_1) = a_1 + 2s
\end{equation}
we deduce from (\ref{eq:Dx0^m}) using (\ref{eq:Dproduct}):
\begin{equation}\label{eq:Da1x0^m}
\begin{split}
D(x^m_0 a_1) & = a_1^2 a_3^{m-1} s^2 + \cdots \\
& \quad \quad + 2a_3^m s + 2a_1 a_3^{m-1}(t + rs) + \cdots \\
& \quad \quad + 2a'_1 (a_3')^m + \cdots. 
\end{split}
\end{equation}
Reducing modulo the invariant ideal $(4, v_1^2)$ we deduce
$$ \partial\left( \frac{a^m_3}{2 v_1} + \cdots \right) = \frac{a^m_3 h_1}{v_1^2} + \cdots. $$
Lemma~\ref{lem:v1BSS} implies that this expression is not in $D(M_1^1A)$ if $m \equiv 3 \mod 4$.
However, if $m \equiv 1 \mod 4$, then Lemma~\ref{lem:v1BSS} implies that $\frac{a^m_3 h_1}{v_1^2}$ is killed in the $v_1$-BSS (\ref{eq:v1BSS}) by $d_2(\frac{a^{m+1}_3}{v^4_1})$.
We compute using the formulas of Section~\ref{subsec:Q3}:
\begin{equation}\label{eq:Dx1}
\begin{split}
D(x_1) & = a_1^2 a_3 s + a_1^3(t+ rs) + \cdots \\
& \quad \quad + 2a_1 a_3 s^2 + \cdots \\
& \quad \quad + 2(a'_1)^3 a'_3 + \cdots.
\end{split}
\end{equation}
We deduce using (\ref{eq:oddpower}) that for $m$ odd we have:
\begin{equation}\label{eq:Dx1^m}
\begin{split}
D(x_1^m) & = a_1^2 a^{2m-1}_3 s + a_1^3a_3^{2m-2}(t+ rs) + \cdots \\
& \quad \quad + 2a_1 a^{2m-1}_3 s^2 + \cdots \\
& \quad \quad + 2(a'_1)^3 (a'_3)^{2m-1} + \cdots.
\end{split}
\end{equation}
We deduce that for $m \equiv 1 \mod 4$ we have
\begin{equation*}
\begin{split}
D(a_1^3 x^m_0+ 2x_1^{\frac{m+1}{2}}) & = a_1^4 a_3^{m-1} s^2 + \cdots \\
& \quad \quad + 2(a'_1)^3 (a_3')^m + \cdots. 
\end{split}
\end{equation*}
Thus we have for $m \equiv 1 \mod 4$:
$$ \partial\left( \frac{x_0^m}{2 v_1} + \cdots \right) = \frac{(a'_3)^m}{v_1} + \cdots $$
and Lemma~\ref{lem:v1BSS} implies that this expression is not in $D(M^1_1 A)$.
This establishes (\ref{eq:Q3beta4}).
\vspace{10pt}

\noindent
{\bf Proof of (\ref{eq:Q3beta5}) for $\pmb{n = 1}$.} Equation~(\ref{eq:Dx1^m}) implies that 
$$ \partial \left( \frac{a^{2m}_3}{v_1^2} + \cdots \right) = \frac{a_3^{2m-1}h_2}{v_1} + \cdots $$
which, by Lemma~\ref{lem:v1BSS}, is not in $D(M_1^1 A)$. This establishes (\ref{eq:Q3beta5}) for $n = 1$.
\vspace{10pt}

\noindent
{\bf Proof of (\ref{eq:Q3beta5}) for $\pmb{n = 2}$.}  We compute using the formulas of Section~\ref{subsec:Q3}
\begin{equation}\label{eq:Dx2}
\begin{split}
D(x_2) & = (a'_1)^6(a'_3)^2 + \cdots \\
& \quad \quad + 2(a'_1)^3(a'_3)^3 + \cdots. 
\end{split}
\end{equation}
Applying (\ref{eq:oddpower}), we get for $m$ odd:
\begin{equation}\label{eq:Dx2^m}
\begin{split}
D(x^m_2) & = (a'_1)^6(a'_3)^{4m-2} + \cdots \\
& \quad \quad + 2(a'_1)^3(a'_3)^{4m-1} + \cdots. 
\end{split}
\end{equation}
It follows that 
$$\partial \left( \frac{x_2^{m}}{2v_1^{2j}} \right) = \frac{(a'_3)^{4m-1}}{v_1^{2j-3}} + \cdots  $$
for $a(1) < 2j \le a(2)$, which is not in $D(M_1^1 A)$ by Lemma~\ref{lem:v1BSS}.  This establishes (\ref{eq:Q3beta5}) for $n = 2$.
\vspace{10pt}

\noindent
{\bf Proof of (\ref{eq:Q3beta1}).}
We deduce from (\ref{eq:Dx2^m}) that $\frac{a^{4m}_3}{4v_1^2}$ exists.  In order to understand its $2$-divisibility, we compute $\partial(\frac{a_3^{4m}}{4v_1^2})$, which is the obstruction to divisibility.   To do this we need to compute $D(\frac{x^m_2}{8v_1^2})$.  Since $(8, v_1^4)$ is an invariant ideal, we compute this from $D(a_1^2 x_2^m)$.  Since
\begin{equation}\label{eq:Da1^2}
D(a^2_1) = 4 s^2 + 4s a_1
\end{equation}
and
\begin{equation}\label{eq:x2}
x_2 \equiv a_3^4 + 2a_1^2 a_3^2 a_4 + a_3^3 a_1^3 \mod (4, v_1^4)
\end{equation}
we deduce from (\ref{eq:Dproduct}) that 
\begin{equation}
\begin{split}
D(a_1^2 x^m_2) & = (a'_1)^8(a'_3)^{4m-2} + \cdots \\
& \quad \quad + 2(a'_1)^5(a'_3)^{4m-1} + \cdots \\
& \quad \quad \quad +4a_3^{4m}s^2 + 4 a_1 a_3^{4m} s  + 4 a_1^3 a_3^{4m-1}s^2 + \cdots
\end{split}
\end{equation}
which gives
$$ D \left( \frac{x_2^{m}}{8v_1^2} \right) = \frac{a_3^{4m}s^2}{2v_1^4} + \frac{a_3^{4m}s}{2v_1^{3}} + \frac{a_3^{4m-1} s^2}{2v_1}.  $$
Lemma~\ref{lem:v1BSS} tells us that $\frac{a_3^{4m}h_2}{v_1^4}$ is killed by $\frac{x^{4m+1}_0}{v_1^{5}}$.  We compute
$$ D(x_0) \equiv a_1 s^2 + s a_1^2 \mod 2 $$
and thus
$$ D(x_0^4) \equiv a_1^4 s^8 \mod (2, v_1^5). $$
Using the fact that
$$ x_0^{4m} \equiv a_3^{4m} \mod (2, v_1^4) $$  
we have
$$ D(x_0^{4m+1}) \equiv a_1 a^{4m}_3 s^2 + a_1^2 a_3^{4m} s + a_1^4 a_3^{4m-3}s^8 \mod (2, v_1^5).  $$
and thus  
$$ D \left( \frac{x_2^{m}}{8v_1^2} + \frac{x_0^{4m+1}}{2v_1^5} \right) = \frac{a_3^{4m-3}s^8}{2v_1} + \frac{a_3^{4m-1} s^2}{2v_1}.  $$
Since $a_4 + a_2^2$ kills $s^4$ (see Remark~\ref{rmk:kill}), $(a_4 + a_2^2)^2$ kills $s^8$, and we compute
$$ D((a_4 + a_2^2)^2) \equiv s^8 + a_3^2 s^2 \mod (2, v_1). $$
Therefore we have
\begin{equation}\label{eq:Da4^2}
D \left( \frac{x_2^{m}}{8v_1^2} + \frac{x_0^{4m+1}}{2v_1^5} + \frac{a^{4m-3}_3(a_4 + a_2^2)^2}{2 v_1} \right) = 0.
\end{equation}
This establishes (\ref{eq:Q3beta1}).
\vspace{10pt}

\noindent
{\bf Proof of (\ref{eq:Q3beta2}).}  
Iterated application of (\ref{eq:Dsquare}) to (\ref{eq:Dx2^m}) yields
\begin{equation}\label{eq:Dx2^m2n}
\begin{split}
D(x_2^{m2^{n-2}}) & = (a'_1)^{3 \cdot 2^{n-1}} (a'_3)^{m2^n - 2^{n-1}} + \cdots \\
& \quad \quad + 2 (a'_1)^{3 \cdot 2^{n-2}} (a'_3)^{m2^n - 2^{n-2}} + \cdots \\
& \quad \quad \quad + 4 (a'_1)^{3 \cdot 2^{n-3}} (a'_3)^{m2^n - 2^{n-3}} + \cdots \\
& \quad \quad \quad \quad + \cdots \\
& \quad \quad \quad \quad \quad + 2^{n-1}(a'_1)^{3} (a'_3)^{m2^n-1} + \cdots.
\end{split}
\end{equation}
It follows that 
$$
D\left( \frac{x_2^{m2^{n-2}}}{4 v_1^{2j}} \right) = 0, \quad 2j \le a(n-1).
$$
This establishes (\ref{eq:Q3beta2}).
\vspace{10pt}

\noindent
{\bf Proof of (\ref{eq:Q3beta3}).}
Suppose that $j$ is even.  Then the ideal $(2^k, v_1^{j2^{k-2}})$ is invariant, and reducing (\ref{eq:Dx2^m2n}) modulo this invariant ideal gives 
$$ D \left( \frac{x_2^{m2^{n-2}}}{2^k v_1^{j2^{k-2}}} \right) = 0, \quad j2^{k-2} \le a(n-k+1). $$
This establishes (\ref{eq:Q3beta3}) for $j$ even.  

Suppose now that $j$ is odd.  Then the ideal $(2^k, v_1^{j2^{k-2}+2^{k-2}})$ is invariant, and in order to compute $D(\frac{x_2^{m2^{n-2}}}{2^k v_1^{j2^{k-2}}})$ we must compute $D(a_1^{2^{k-2}}x_2^{m2^{n-2}})$ modulo 
$(2^k, v_1^{j2^{k-2}+2^{k-2}})$.  
Repeated application of (\ref{eq:Dsquare}) to (\ref{eq:Da1^2}) yields
\begin{equation}\label{eq:Da1^2k}
D(a_1^{2^{k-2}}) \equiv 2^{k-1} a_1^{2^{k-2}-2} s^2 + 2^{k-1} a_1^{2^{k-2}-1} s \mod 2^k. 
\end{equation}
We also note that since 
$$ x_2 \equiv a_3^4 + a_1^3 a_3^3 + \cdots \mod 2 $$
we have
\begin{equation}\label{eq:x2^m2n}
\begin{split}
x_2^{m2^{n-2}} & \equiv a_3^{m2^n} + a^{3 \cdot 2^{n-2}}_1 a_3^{3\cdot 2^{n-2} + (m-1)2^{n-2}} + \cdots \mod 2 \\ 
& \equiv a_3^{m2^n} + a^{3 \cdot 2^{n-2}}_1 a_3^{2^{n-1} + m2^{n-2}} + \cdots \mod 2.
\end{split}
\end{equation}
Applying (\ref{eq:Dproduct}) to (\ref{eq:Dx2^m2n}), (\ref{eq:Da1^2k}), and (\ref{eq:x2^m2n}), we get
\begin{equation}\label{eq:Da1^2kx2^m2n}
\begin{split}
D(a_1^{2^{k-2}}x_2^{m2^{n-2}}) & = (a'_1)^{3 \cdot 2^{n-1}+2^{k-2}} (a'_3)^{m2^n - 2^{n-1}} + \cdots \\
& \quad \quad + 2 (a'_1)^{3 \cdot 2^{n-2}+2^{k-2}} (a'_3)^{m2^n - 2^{n-2}} + \cdots \\
& \quad \quad \quad + 4 (a'_1)^{3 \cdot 2^{n-3}+ 2^{k-2}} (a'_3)^{m2^n - 2^{n-3}} + \cdots \\
& \quad \quad \quad \quad + \cdots \\
& \quad \quad \quad \quad \quad + 2^{k-1}(a'_1)^{3 \cdot 2^{n-k}+2^{k-2}} (a'_3)^{m2^n-2^{n-k}} + \cdots \\
& +  2^{k-1} a_1^{2^{k-2}-2} a_3^{m2^n}s^2  + 2^{k-1} a_1^{2^{k-2}-1} a_3^{m2^n} s + 2^{k-1}a^{3 \cdot 2^{n-2}}_1 a_3^{2^{n-1} + m2^{n-2}}s^2 + \cdots.
\end{split}
\end{equation}
We deduce that for $j$ odd and $j2^{k-2} \le a(n-k+1)$ we have
$$ D\left( \frac{x_2^{m2^{n-2}}}{2^k v_1^{j2^{k-2}}}\right) = \frac{a_3^{m2^n}s^2}{2v_1^{j2^{k-2}+2}} + \frac{a_3^{m2^n}s}{2v_1^{j2^{k-2}+1}}. $$
However, Lemma~\ref{lem:v1BSS} implies that $\frac{a_3^{m2^n} h_2}{v_1^{j2^{k-2}+2}}$ is killed by $\frac{a_3^{m2^n+1}}{v_1^{j2^{k-2}+3}}$.  It follows from (\ref{eq:Dx0}) that we have
$$ D(x_0^{m2^n}) \equiv a_1^{m2^n} s^{m2^{n+1}} + \cdots \mod 2 $$
and hence
$$ D(x_0^{m2^n+1}) \equiv a_1 a^{m2^n}_3 s^2 + a_1^2 a^{m2^n}_3 s + a_1^{m2^n} a_3 s^{m2^{n+1}} + \cdots \mod 2. $$
This implies that we have 
\begin{equation}\label{eq:Dx0^m2n}
D\left( \frac{x_0^{m2^n+1}}{2 v_1^{j2^{k-2}+3}} \right) = \frac{a^{m2^n}_3 s^2}{2 v_1^{j2^{k-2}+2}}
+ \frac{a^{m2^n}_3 s}{2 v_1^{j2^{k-2}+1}}
\end{equation}
and therefore
$$ D\left( \frac{x_2^{m2^{n-2}}}{2^k v_1^{j2^{k-2}}} + \frac{x_0^{m2^n+1}}{2 v_1^{j2^{k-2}+3}} \right) = 0.$$
This establishes (\ref{eq:Q3beta3}).
\vspace{12pt}

\noindent
{\bf Proof of (\ref{eq:Q3beta5}) for $\pmb{n \ge 3}$.}
It follows from (\ref{eq:Dx2^m2n}) that we have for $a(n-1) < 2j \le a(n)$
$$ D\left( \frac{x_2^{m2^{n-2}}}{4 v_1^{2j}} \right) = \frac{(a'_3)^{m2^n-2^{n-2}}}{2 v_1^{2j-a(n-1)}} + \cdots $$
and hence
$$ \partial\left( \frac{x_2^{m2^{n-2}}}{2 v_1^{2j}} \right) = \frac{(a'_3)^{m2^n-2^{n-2}}}{v_1^{2j-a(n-1)}} + \cdots. $$
This element is not in $D(M_1^1 A)$ by Lemma~\ref{lem:v1BSS}.  This establishes (\ref{eq:Q3beta5}).
\vspace{10pt}

\noindent
{\bf Proof of (\ref{eq:Q3beta6}).}
Suppose that $j$ is even.  Then the ideal $(2^k, v_1^{j2^{k-2}})$ is invariant, and reducing (\ref{eq:Dx2^m2n}) modulo this invariant ideal gives 
$$ D \left( \frac{x_2^{m2^{n-2}}}{2^k v_1^{j2^{k-2}}} \right) = \frac{(a'_3)^{m2^n-2^{n-k}}}{2 v_1^{j2^{k-2} - a(n-k+1)}} + \cdots, \quad a(n-k+1) < j2^{k-2} \le a(n-k+2) $$
and therefore
$$ \partial \left( \frac{x_2^{m2^{n-2}}}{2^{k-1} v_1^{j2^{k-2}}} \right) = \frac{(a'_3)^{m2^n-2^{n-k}}}{ v_1^{j2^{k-2} - a(n-k+1)}}+ \cdots , \quad a(n-k+1) < j2^{k-2} \le a(n-k+2). $$
Since $k \ge 3$, this is not in $D(M_1^1 A)$ by Lemma~\ref{lem:v1BSS}.  
This establishes (\ref{eq:Q3beta3}) for $j$ even.  

Suppose now that $j$ is odd.  Then the ideal $(2^k, v_1^{j2^{k-2}+2^{k-2}})$ is invariant, and in order to compute $D(\frac{x_2^{m2^{n-2}}}{2^k v_1^{j2^{k-2}}})$ we must compute $D(a_1^{2^{k-2}}x_2^{m2^{n-2}})$ modulo 
$(2^k, v_1^{j2^{k-2}+2^{k-2}})$.    It follows from (\ref{eq:Da1^2kx2^m2n}) that for $j$ odd and $a(n-k+1) < j2^{k-2} \le a(n-k+2)$ we have
$$ D\left( \frac{x_2^{m2^{n-2}}}{2^k v_1^{j2^{k-2}}}\right) = \frac{a_3^{m2^n}s^2}{2v_1^{j2^{k-2}+2}} + \frac{a_3^{m2^n}s}{2v_1^{j2^{k-2}+1}} + \frac{(a'_3)^{m2^n-2^{n-k}}}{2 v_1^{j2^{k-2} - a(n-k+1)}} + \cdots. $$
Using (\ref{eq:Dx0^m2n}), we have
$$ D\left( \frac{x_2^{m2^{n-2}}}{2^k v_1^{j2^{k-2}}} + \frac{x_0^{m2^n+1}}{2 v_1^{j2^{k-2}+3}}\right) = \frac{(a'_3)^{m2^n-2^{n-k}}}{2 v_1^{j2^{k-2} - a(n-k+1)}} + \cdots $$
and therefore
$$ \partial\left( \frac{x_2^{m2^{n-2}}}{2^{k-1} v_1^{j2^{k-2}}}\right) = \frac{(a'_3)^{m2^n-2^{n-k}}}{ v_1^{j2^{k-2} - a(n-k+1)}} + \cdots $$
Since $k \ge 3$, this is not in $D(M_1^1 A)$ by Lemma~\ref{lem:v1BSS}.  
This establishes (\ref{eq:Q3beta3}) for $j$ odd.
\end{proof}

\subsection{Computation of $\pmb{H^{*,*}(M_2^0 C^*_{tot}(Q(5)))}$}

We have (as before)
\begin{gather*}
H^{*,*}(M_2^0 C^*_{\Gamma}(A))  = \FF_2[a_3^{\pm 1},  h_{1}, h_2, g]/(h_2^3 = a_3 h_1^3), \\
H^{*,*}(M_2^0 C^*_{\Lambda^1(5)}(B^1))  = \FF_2[u^{\pm}, h_{2,1}]  
\end{gather*}
with $(s,t)$-bidegrees
\begin{align*}
| a_3 | & = (0,6),  \\
| h_1 | & = (1, 2),  \\
| h_2 | & = (1,4),  \\
| g | & = (4, 24), \\
| u | & = (0,2), \\ 
| h_{2,1} | & = (1,6),
\end{align*} 
and $h_{2,1}^4 = g$.  
Moreover, the spectral sequence of the double complex gives
\begin{equation}\label{eq:ssdouble5}
 \begin{array}{c}
H^{s,t}(M_2^0 C^*_{\Gamma}(A)) \oplus  \\
H^{s-1,t}(M_2^0 C^*_{\Gamma}(A)) \oplus H^{s-1,t}(M_2^0 C^*_{\Lambda^1(5)}(B^1)) \oplus \\ 
H^{s-2,t}(M_2^0 C^*_{\Lambda^1(5)}(B^1)) 
\end{array}
\Bigg{\}}
\Rightarrow  H^{s,t}(M_2^0 C^*_{tot}(Q(5))).
\end{equation}
As before, we will differentiate the terms $x$ with the same name occurring in the different groups in the $E_1$-term of spectral sequence (\ref{eq:ssdouble5}), we shall employ the following notational convention:
\begin{equation*}
\begin{split}
x & \in C^*_{\Gamma}(A) \: \text{on the $0$-line}, \\
\bar{x} & \in C^*_{\Gamma}(A) \: \text{on the $1$-line}, \\
y & \in C^*_{\Lambda^1(\ell)}(B^1) \: \text{on the $1$-line}, \\
\bar{y} & \in C^*_{\Lambda^1(\ell)}(B^1) \: \text{on the $2$-line}.
\end{split}
\end{equation*}
The formulas of Section~\ref{subsec:Q5} show that the only non-trivial $d_1$ differentials in spectral sequence (\ref{eq:ssdouble}) are 
$$ d_1(g^i \bar{a}_3^j) = h_{2,1}^{4i} \bar{u}^{3j}. $$
Since the spectral sequence (\ref{eq:ssdouble5}) is a spectral sequence of modules over $\Ext^{*,*}(BP_*)$, we deduce that there are no possible $d_r$-differentials for $r > 1$.  We deduce that we have 
\begin{align*}
H^{*,*}(M_2^0 C^*_{tot}(Q(5))) & =  \FF_2[a_3^{\pm 1},  h_{1}, h_2, g]/(h_2^3 = a_3 h_1^3) \\
& \quad \oplus \FF_2[\bar{a}_3^{\pm 1}, \bar{g}]\{ \bar{h}_1, \bar{h}_2, \bar{h}_1^2, \bar{h}^2_2, \bar{h}_2^3 = \bar{a}_3 \bar{h}_1^3\} \\
& \quad \oplus  \FF_2[u^{\pm 1}, h_{2,1}]  \\
& \quad \oplus  \FF_2[\bar{u}^{\pm 3}, g]\{ \bar{h}_{2,1}, (\bar{h}_{2,1})^2, (\bar{h}_{2,1})^3 \} \\
& \quad \oplus  \FF_2[\bar{u}^{\pm 3}, \bar{h}_{2,1}]\{ \bar{u}, \bar{u}^2 \}.
\end{align*}

\subsection{Computation of $\pmb{H^{0,*}(M_1^1 C^*_{tot}(Q(5)))}$}

We now compute the differentials in the $v_1$-BSS
\begin{equation}\label{eq:Q5v1BSS}
H^{s,*}(M_2^0 C^*_{tot}(Q(5))) \otimes \FF_2[v_1]/(v_1^\infty) \Rightarrow H^{s,*}(M_1^1 C^*_{tot}(Q(5)))
\end{equation}
from the $s = 0$-line to the $s = 1$-line.  

One computes using the formulas of Section~\ref{subsec:Q5}:
\begin{equation}\label{eq:Q5Dxi}
\begin{split}
D(x_0) & \equiv a_1 s^2 \mod (2, v_1^2), \\
D(x_1) & \equiv a_1^2 a_3 s \mod (2, v_1^3), \\
D(x_2) & \equiv a_1^8 u^4 \mod (2, v_1^9)
\end{split}
\end{equation}
for $x_i$ as in Section~\ref{sec:Q3v1BSS}.  The formula for $D(x_2)$ already informs us that the $v_1$-BSS for $Q(5)$ differs from the $v_1$-BSS for $Q(3)$.

It follows from inductively applying (\ref{eq:square}) that we have
$$ D(x_2^{2^{n-2}}) \equiv a_1^{2^{n+1}}u^{2^{n}} \mod (2, v_1^{2^{n+1}+1}). $$
It follows from $(\ref{eq:oddpower})$ that for $m$ odd, we have
\begin{align*}
D(x^m_0) & \equiv a_1 s^2 a_3^{m-1} \mod (2, v_1^2), \\
D(x_1^m) & \equiv a_1^2 a^{2m-1}_3 s \mod (2, v_1^3), \\
D(x_2^{m2^{n-2}}) & \equiv a_1^{2^{n+1}}u^{3m 2^n - 2^{n+1}} \mod (2, v_1^{2^{n+1}+1}).
\end{align*}

We deduce the following.

\begin{lemma}\label{lem:Q5v1BSS}
The $v_1$-BSS differentials in (\ref{eq:v1BSS}) from the $(s = 0)$-line to the $(s = 1)$-line are given by
\begin{align*}
d_{1}\left( \frac{a_{3}^{m}}{v_1^j}\right) & = \frac{a^{m-1}_{3} h_2}{v_1^{j-1}}, \\
d_{2}\left( \frac{a_{3}^{2m}}{v_1^j}\right) & = \frac{a^{2m-1}_{3} h_1}{v_1^{j-2}}, \\
d_{2^{n+1}} \left( \frac{a_3^{m2^n}}{v_1^j} \right) & = \frac{u^{3m2^n - 2^{n+1}}}{v_1^{j - 2^{n+1}}}
\end{align*}
where $m$ is odd.
\end{lemma}

\begin{cor}
\label{cor:Q5beta}
The groups $H^{0,*}(M_1^1C^*_{tot}(Q(5)))$ are generated by the elements  
\begin{align*}
& 1/v_1^j,  & & j \ge 1, \\
\\
& \frac{a_3^{m2^n}}{v_1^{j}}, & &  
\text{$m$ odd and} \\ 
& & & j \le \begin{cases}
1, & n = 0, \\
2, & n = 1, \\
2^{n+1}, & n \ge 2.
\end{cases}
\end{align*}
In particular, the map
$$ \Ext^{0,*}(M_1^1 BP_*) \rightarrow H^{0,*}(M_1^1 C^*_{tot}(Q(5))) $$
is \emph{not} an isomorphism.
\end{cor}

\section{Low dimensional computations}
\label{sec:lowdim}

In this section we explore the $2$-primary homotopy $\pi_*Q(3)$ and $\pi_* Q(5)$ for $0 \le *
<48$ (everything is implicitly $2$-localized).  In the case of $Q(3)$, Mark Mahowald has done similar computations, over a much vaster range, for the closely related Goerss-Henn-Mahowald-Rezk conjectural resolution of the $2$-primary $K(2)$-local sphere --- there is definitely some overlap here.  In the case of $Q(5)$ the computations represent some genuinely unexplored territory, and give evidence that $Q(5)$ may detect more non-$\beta$-family $v_2$-periodic homotopy than $Q(3)$.

We do these low-dimensional computations in the most simple-minded manner, by computing the Bousfield-Kan spectral sequence
$$ E_1^{s,t}(Q(\ell)) \Rightarrow \pi_{t-s} Q(\ell) $$
with
$$
E_1^{s,t}  = \begin{cases}
\pi_t \TMF, & s = 0, \\
\pi_t \TMF_0(\ell) \oplus \pi_t \TMF, & s = 1, \\
\pi_t \TMF_0(\ell), & s = 2.
\end{cases} 
$$
Actually, as the periodic versions of $\TMF$ typically have $\pi_t$ of infinite rank, we only compute a certain ``connective cover'' of the spectral sequence --- we only include holomorphic modular forms in this low dimensional computation (i.e. we do not invert $\Delta$).  Thus we are only computing a portion of the spectral sequence, which we shall refer to as the \emph{holomorphic summand}.  Note that the authors are not claiming that there exists a bounded below version of $Q(\ell)$ whose homotopy groups the holomorphic summand converges to (it remains an interesting open question how such connective versions of $Q(\ell)$ could be obtained by extending the semi-cosimplicial complex over the cusps).  Indeed, recent advances by Hill and Lawson \cite{HillLawson} may produce such a bounded below $Q(\ell)$-spectrum, but we do not pursue this possibility here.

In the following calculations, we employ a leading term algorithm, which basically amounts to only computing the leading terms of the differentials in row echelon form.  Similarly to the previous section, we write everything $2$-adically, and employ a lexicographical ordering on monomials
$$ 2^i v_1^j x. $$
Namely we say that $2^i v_1^j x$ is \emph{lower} than $2^{i'}v_1^{j'} x'$ if $i < i'$, or if $i = i'$ and $j < j'$.
We will write ``leading term'' differentials:  the expression
$$ x \mapsto y $$
indicates that 
$$ d_r(x + \text{higher terms}) =  y + \text{higher terms}. $$

\subsection{The case of $Q(3)$}
 
In the case of $\TMF_0(3)$, recall that the modular forms of for $\Gamma_0(3)$ are spanned by those monomials $a_1^i a_3^j$ in $\ZZ[1/3, a_1, a_3]$ with $i + j$ even.  In this section we will refer to $a_1$ as $v_1$ and $a_3$ as $v_2$, because that is what they correspond to under the complex orientation.

\begin{figure}
\includegraphics[height=\textwidth, angle=270]{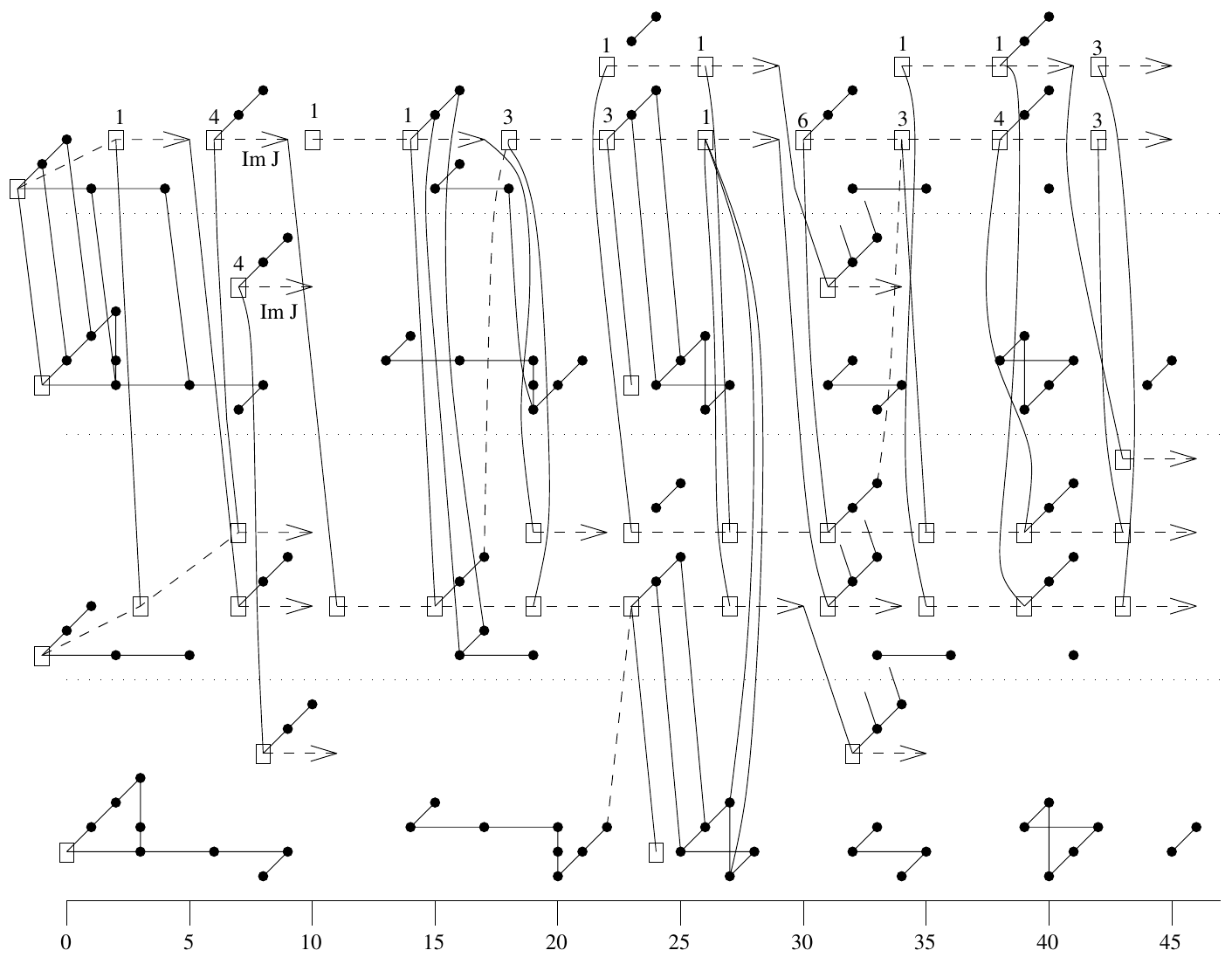}
\caption{The holomorphic summand of the spectral sequence $E_r^{s,t}(Q(3))$ in low degrees.}\label{fig:Q3low}
\end{figure}

Figure~\ref{fig:Q3low} shows a low dimensional portion of the holomorphic summand of the spectral sequence $E^{s,t}_r(Q(3))$.  There are many aspects of this chart that deserve explanation/remark.
\begin{itemize}
\item The copies of $\pi_*\TMF$ and $\pi_*\TMF_0(3)$ are separated by dotted lines.  The bottom pattern is the $s= 0$ line of the spectral sequence ($\pi_* \TMF$).  The next pattern up is the $\pi_* \TMF_0(3)$ summand of the $s = 1$ line, followed by the $\pi_* \TMF$ summand of the $s = 1$ line.  The top pattern is the $s = 2$ line of the spectral sequence ($\pi_* \TMF_0(3)$).  The spectral sequence is Adams-indexed, with the $x$-axis corresponding to the coordinate $t-s$.  
\item Dots indicate $\ZZ/2$'s.  Boxes indicate $\ZZ_{(2)}$'s.  The solid lines between the dots indicate $2$-extensions, and $\eta$ and $\nu$ multiplication.  
\item Horizontal dashed lines denote $bo$-patterns.  Arrows indicate the $bo$ patterns continue.
\item There are two $bo$-patterns which are denoted ``Im J''.  These $bo$-patterns (together with the $bo$-patterns which hit them with differentials) combine to form Im J patterns.
\item Differentials are indicated with vertical curvy lines.  All differentials displayed only indicate the leading terms of the differentials, as explained in the beginning of this section.  For example, the $d_1$ differential from the $1$-line to the $2$-line showing 
$$  v_1^2v_2^2 \mapsto 2v_2^2v_1^2 $$
actually corresponds to a differential
$$ d_1(v_1^2v_2^2+v_1^5v_2) = 2v_2^2v_1^2 + \text{higher terms}. $$
The differentials on the torsion-free portions spanned by the modular forms are computed using the Mahowald-Rezk formulas.
\item Differentials on the torsion summand can often be computed by noting that the maps $f$, $t$, $q$, and $\psi^3$ that define the coface maps of the semi-cosimplicial spectrum $Q(3)^\bullet$ are all maps of ring spectra, and in particular all have the same effect on elements in the Hurewicz image.  There are a few notable exceptions, which we explain below.
\item Dashed lines between layers indicate hidden extensions.  These (probably) do not represent all hidden extensions: there are several possible hidden extensions which we have not resolved.
\item The differentials supported by the non-Hurewicz classes $x$ and $\eta x$ in $\pi_{17} \TMF_0(3)$ and $\pi_{18} \TMF_0(3)$ are deduced because they kill the Hurewicz image of $\beta_{4/4} \eta$ and $\beta_{4/4} \eta^2$, which are zero in $\pi_*S$.
\item The $d_2$-differentials are computed by observing that there is a (zero) hidden extension $\eta^3v_1^6v_2^2[1] = 4v_1^5v_2^3[2]$ (where $[s]$ means $s$-line).
\item Up to the natural deviations introduced by computing with the Bousfield-Kan spectral sequence, and not the Adams-Novikov spectral sequence, the divided $\beta$ family is faithfully reproduced on the $2$-line with the exception of the additional copy of Im J (there in fact should be infinitely many copies of such Im J summands) and one peculiar abnormality: the element $\beta_{8/8}$, detected by $32v_1v_2^5$ is $32$-divisible.  This extra divisibility does not contradict the results of Section~\ref{sec:beta} --- the results there pertain to the monochromatic layer $M_2Q(3)$, and not $Q(3)$ directly.
\item Boxes which are targets of differentials are labeled with numbers.  A number $n$ above a box indicates that after all differentials are run, you are left with a $\ZZ/2^n$.
\item It is interesting to note that the permanent cycles on the zero line in this range are exactly the image of the $\TMF$-Hurewicz homomorphism
\end{itemize}

We did not label the modular forms generating the boxes in the spectral sequence.  In the case of $\pi_* \TMF$, the dimensions resolve this ambiguity.  The remaining ambiguity is resolved by the following table, which indicates all of the leading terms of $d_1$ differentials between torsion free classes on the $1$ and $2$-lines of the spectral sequence.
\begin{alignat*}{3}
v_1^2 & \mapsto 2v_1^2 && && \\
v_1v_2 & \mapsto v_1^4 & v_1^4 & \mapsto 16v_1v_2 && \\
v_2^2 & \mapsto 2v_2^2 && &&\\
v_1^2v_2^2 & \mapsto 2v_2^2v_1^2 && &&\\
v_1v_2^3 & \mapsto v_2^2v_1^4  & v_1^4v_2^2 & \mapsto 8v_2^3v_1 && \\
8\Delta & \mapsto 8v_2^3v_1^3 & v_2^4 & \mapsto 2v_2^4 & v_1^4v_2^2 & \mapsto 8v_2^3v_1 \\
v_2^4 & \mapsto 2v_2^4 && && \\
v_2^4v_1^2 & \mapsto 2v_2^4v_1^2 & v_1^8v_2^2 & \mapsto 8v_2^3v_1^5 && \\
c_4\Delta & \mapsto v_2^4v_1^4 & \quad  v_2^5v_1 & \mapsto v_2^3v_1^7 \quad & v_1^4v_2^4 & \mapsto 64v_2^5v_1 \\
v_2^4v_1^6 & \mapsto 8v_2^5v_1^3 & v_2^6 & \mapsto 2v_2^6 && \\
v_2^4v_1^8 & \mapsto 16v_2^5v_1^5 & v_2^6v_1^2 & \mapsto 2v_2^6v_1^2 && \\
v_2^7v_1 & \mapsto v_2^6v_1^4 & v_2^6v_1^4 & \mapsto 8v_2^7v_1 & v_1^{10}v_2^4 & \mapsto 8v_2^5v_1^7 
\end{alignat*}

\subsection{The case of $Q(5)$}

\begin{figure}
\includegraphics[height=\textwidth, angle=270]{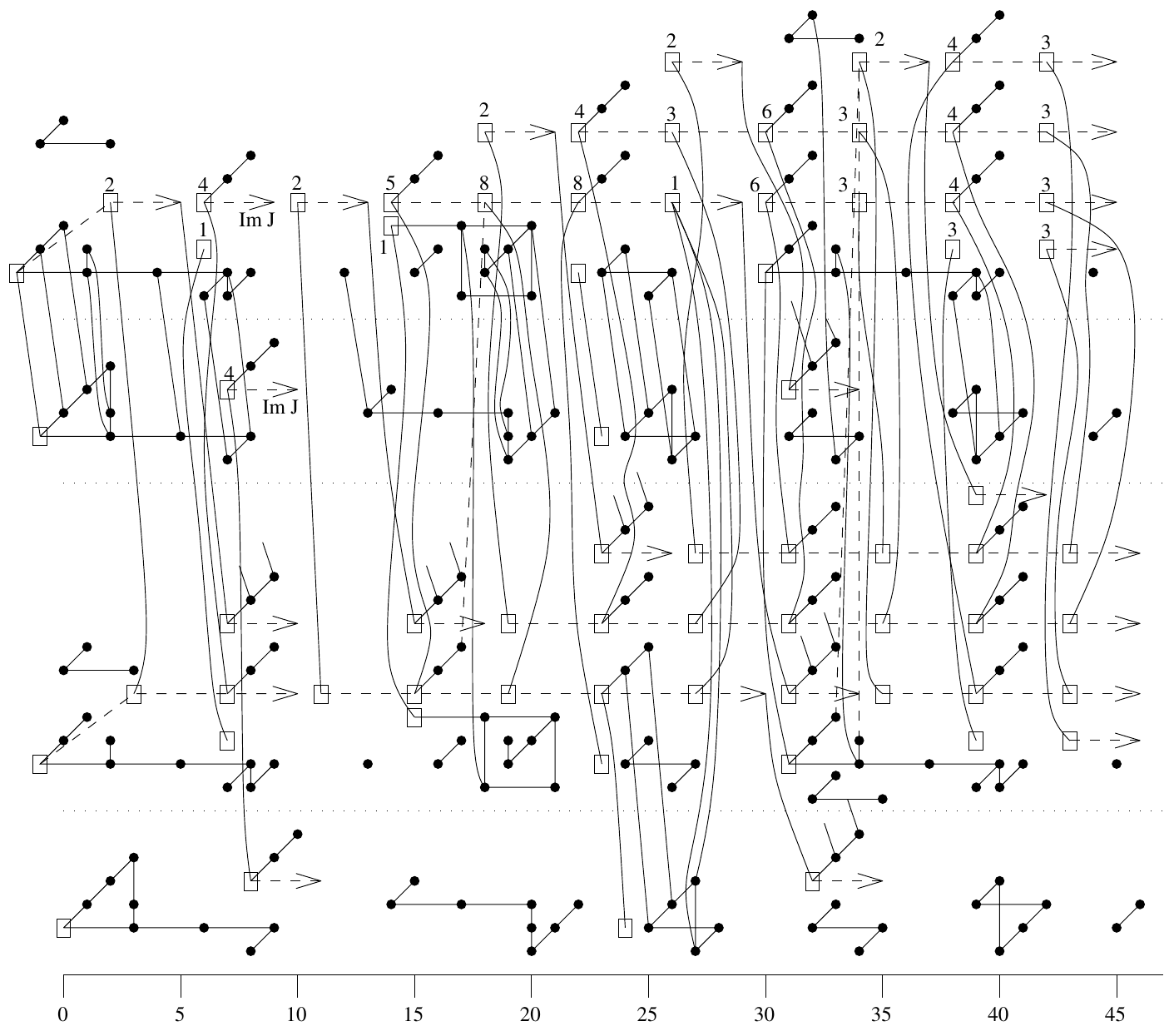}
\caption{The holomorphic summand of the spectral sequence $E_r^{s,t}(Q(5))$ in low degrees.}\label{fig:Q5low}
\end{figure}

Figure~\ref{fig:Q5low} displays the spectral sequence for $Q(5)$.  Essentially all of the conventions and remarks for the $Q(3)$ computation above extend to the $Q(5)$ computation.  Below is the corresponding table for leading terms of differentials from the torsion-free elements in the $1$-line to those in the $2$-line.

\begin{alignat*}{5}
b_2 & \mapsto 4b_2 && && && &&\\
b_4 & \mapsto b_2^2 & \delta & \mapsto 2\delta & b_2^2 & \mapsto 16b_4 && &&\\
b_2\delta & \mapsto 4b_2\delta && && && && \\
b_4\delta & \mapsto b_2^2\delta & \delta^2 & \mapsto 2\delta^2 & b_2^2\delta & \mapsto 32b_4\delta && &&\\
b_2 \delta^2 & \mapsto 4\delta^2b_2 & b_2^3\delta & \mapsto 8b_2b_4\delta && && && \\
8\Delta & \mapsto 8\delta^3 & b_4\delta^2 & \mapsto b_2^2\delta^2 & 4\delta^3 & \mapsto 8b_2^2b_4\delta & b_2^2\delta^2 & \mapsto 16b_4\delta^2 && \\
b_2^5\delta & \mapsto 8b_4\delta b_2^3 & b_2\delta^3 & \mapsto 4b_2\delta^3 & b_2^3\delta^2 & \mapsto 8b_4\delta^2b_2 && && \\
c_4\Delta & \mapsto b_2^2\delta^3 & \quad \delta^4 & \mapsto 2\delta^4 & \quad b_4\delta^3 & \mapsto b_2^4b_4\delta & \quad b_2^2\delta^3 & \mapsto 64b_4\delta^3 & \quad b_2^4\delta^2 & \mapsto 64b_2^2b_4\delta^2  \\
b_2\delta^4 & \mapsto 4b_2\delta^4 & b_2^3\delta^3 & \mapsto 8b_2b_4\delta^3 & b_2^5\delta^2 & \mapsto 8b_2^3b_4\delta^2 && && \\
b_4\delta^4 & \mapsto b_2^2\delta^4 & 4\delta^5 & \mapsto 8\delta^5 & b_2^4\delta^3 & \mapsto 16b_2^2b_4\delta^3 & b_2^2 \delta^4 & \mapsto 16b_4\delta^4 & b_2^6\delta^2 & \mapsto 16b_4\delta^2b_2^4 \\
b_2\delta^5 & \mapsto 4b_2b_4\delta^4 & \quad b_2^3\delta^4 & \mapsto 8b_2\delta^5 & b_2^5\delta^3 & \mapsto 8\delta^3b_4b_2^6 & b_2^7\delta^2 & \mapsto 8b_2^5b_4\delta^2 &&  \\
\end{alignat*}

We make the following remarks.

\begin{itemize}
\item The $2$-line now bears little resemblance to the divided $\beta$-family.  This is in sharp contrast with the situation with $Q(3)$.  This fits well with our premise that while $Q(3)$ reproduces the divided $\beta$ family almost flawlessly, $Q(5)$ does not.
\item The much more robust torsion in $\pi_* \TMF_0(5)$ gives a significant source of homotopy in $\pi_*Q(5)$ which does not appear in $\pi_*Q(3)$.  In particular, the elements
$$ \nu \delta^4, \quad \nu^2\delta^4, \quad \epsilon \delta^4 $$
seem like candidates to detect the elements in $\pi_*S$ with Adams spectral sequence names
$$ h_5 h_2^2, \quad h_5 h_2^3, \quad h_5h_3h_1, $$
though the ambiguity resulting from the leading term algorithm makes it difficult to resolve this in the affirmative.  These classes are \emph{not} seen by $Q(3)$.
\item Just as in the case of $Q(3)$, the permanent cycles on the zero line in this range are exactly the image of the $\TMF$-Hurewicz homomorphism.
\end{itemize}

\bibliographystyle{plain} 
\bibliography{Q5revised}

\end{document}